\documentclass{article}
\usepackage{amsmath}
\usepackage{amssymb}
\usepackage{amsthm}
\usepackage{xcolor}
\usepackage{bbm}
\usepackage{cite}
\usepackage{stmaryrd}
\usepackage{titling}
\SetSymbolFont{stmry}{bold}{U}{stmry}{m}{n}
\usepackage{euscript}
\usepackage[arrow,curve,matrix,arc,2cell]{xy}
\UseAllTwocells
\usepackage[utf8]{inputenc}
\usepackage{slashed}
\usepackage[unicode]{hyperref}
\DeclareFontFamily{U}{rsfs}{} 
\DeclareFontShape{U}{rsfs}{n}{it}{<->
rsfs10}{} \DeclareSymbolFont{mscr}{U}{rsfs}{n}{it}
\DeclareSymbolFontAlphabet{\scr}{mscr}
\def\mathscr{\scr}
\begin{document}
\def\e#1\e{\begin{equation}#1\end{equation}}
\def\ea#1\ea{\begin{align}#1\end{align}}
\def\eq#1{{\rm(\ref{#1})}}
\theoremstyle{plain}
\newtheorem{thm}{Theorem}[section]
\newtheorem{lem}[thm]{Lemma}
\newtheorem{prop}[thm]{Proposition}
\newtheorem{cor}[thm]{Corollary}
\theoremstyle{definition}
\newtheorem{dfn}[thm]{Definition}
\newtheorem{ex}[thm]{Example}
\newtheorem{rem}[thm]{Remark}
\numberwithin{figure}{section}
\numberwithin{equation}{section}
\def\dim{\mathop{\rm dim}\nolimits}
\def\codim{\mathop{\rm codim}\nolimits}
\def\vdim{\mathop{\rm vdim}\nolimits}
\def\Re{\mathop{\rm Re}\nolimits}
\def\Im{\mathop{\rm Im}\nolimits}
\def\det{\mathop{\rm det}\nolimits}
\def\Ker{\mathop{\rm Ker}}
\def\Coker{\mathop{\rm Coker}}
\def\HSta{\mathop{\bf HSta}\nolimits}
\def\DSta{\mathop{\bf DSta}\nolimits}
\def\Spec{\mathop{\rm Spec}}
\def\Perf{\mathop{\rm Perf}\nolimits}
\def\Aut{\mathop{\rm Aut}}
\def\End{\mathop{\rm End}}
\def\Map{{\mathop{\rm Map}\nolimits}}
\def\Topho{{\mathop{\bf Top^{ho}}}}
\def\Ho{\mathop{\rm Ho}}
\def\Pd{\mathop{\rm Pd}}
\def\Gr{\mathop{\rm Gr}}
\def\PGL{\mathop{\rm PGL}}
\def\GL{\mathop{\rm GL}}
\def\SL{\mathop{\rm SL}}
\def\SO{\mathop{\rm SO}}
\def\SU{\mathop{\rm SU}}
\def\Sp{\mathop{\rm Sp}}
\def\Spin{\mathop{\rm Spin}}
\def\Spinc{\mathop{\rm Spin^c}}
\def\Tr{\mathop{\rm Tr}}
\def\U{{\mathbin{\rm U}}}
\def\vol{\mathop{\rm vol}}
\def\inc{\mathop{\rm inc}}
\def\ran{{\rm an}}
\def\cs{{\rm cs}}
\def\ind{\mathop{\rm ind}\nolimits}
\def\tind{{\text{\rm t-ind}}}
\def\Pic{\mathop{\rm Pic}}
\def\Stab{\mathop{\rm Stab}\nolimits}
\def\Hol{\mathop{\rm Hol}}
\def\diag{\mathop{\rm diag}}
\def\Or{\mathop{\rm Or}}
\def\rank{\mathop{\rm rank}\nolimits}
\def\IndSch{\mathop{\bf IndSch}\nolimits}
\def\Sta{\mathop{\bf Sta}\nolimits}
\def\Hom{\mathop{\rm Hom}\nolimits}
\def\Ext{\mathop{\rm Ext}\nolimits}
\def\cExt{\mathop{{\mathcal E}\mathit{xt}}\nolimits}
\def\id{{\mathop{\rm id}\nolimits}}
\def\Id{{\mathop{\rm Id}\nolimits}}
\def\Sch{\mathop{\bf Sch}\nolimits}
\def\TopSta{{\mathop{\bf TopSta}\nolimits}}
\def\Ad{\mathop{\rm Ad}}
\def\top{{\rm top}}
\def\cla{{\rm cla}}
\def\irr{{\rm irr}}
\def\red{{\rm red}}
\def\virt{{\rm virt}}
\def\coh{{\rm coh}}
\def\Top{{\mathop{\bf Top}\nolimits}}
\def\ul{\underline}
\def\bs{\boldsymbol}
\def\db{\bar\partial}
\def\ge{\geqslant}
\def\le{\leqslant\nobreak}
\def\boo{{\mathbin{\mathbbm 1}}}
\def\O{{\mathcal O}}
\def\bA{{\mathbin{\mathbb A}}}
\def\bG{{\mathbin{\mathbb G}}}
\def\bL{{\mathbin{\mathbb L}}}
\def\bT{{\mathbin{\mathbb T}}}
\def\H{{\mathbin{\mathbb H}}}
\def\R{{\mathbin{\mathbb R}}}
\def\Z{{\mathbin{\mathbb Z}}}
\def\Q{{\mathbin{\mathbb Q}}}
\def\N{{\mathbin{\mathbb N}}}
\def\C{{\mathbin{\mathbb C}}}
\def\CP{{\mathbin{\mathbb{CP}}}}
\def\HP{{\mathbin{\mathbb{HP}}}}
\def\RP{{\mathbin{\mathbb{RP}}}}
\def\A{{\mathbin{\cal A}}}
\def\G{{\mathbin{\cal G}}}
\def\M{{\mathbin{\cal M}}}
\def\B{{\mathbin{\cal B}}}
\def\ovB{{\mathbin{\smash{\,\overline{\!\mathcal B}}}}}
\def\cC{{\mathbin{\cal C}}}
\def\cD{{\mathbin{\cal D}}}
\def\cE{{\mathbin{\cal E}}}
\def\cF{{\mathbin{\cal F}}}
\def\cG{{\mathbin{\cal G}}}
\def\cH{{\mathbin{\cal H}}}
\def\E{{\mathbin{\cal E}}}
\def\F{{\mathbin{\cal F}}}
\def\cG{{\mathbin{\cal G}}}
\def\cH{{\mathbin{\cal H}}}
\def\cI{{\mathbin{\cal I}}}
\def\cJ{{\mathbin{\cal J}}}
\def\cK{{\mathbin{\cal K}}}
\def\cL{{\mathbin{\cal L}}}
\def\cN{{\mathbin{\cal N}\kern .04em}}
\def\cP{{\mathbin{\cal P}}}
\def\cQ{{\mathbin{\cal Q}}}
\def\cR{{\mathbin{\cal R}}}
\def\cS{{\mathbin{\cal S}}}
\def\T{{{\cal T}\kern .04em}}
\def\cU{{\mathbin{\cal U}}}
\def\cV{{\cal V}}
\def\cW{{\mathbin{\cal W}}}
\def\cX{{\cal X}}
\def\cY{{\cal Y}}
\def\cZ{{\cal Z}}
\def\bcM{{\mathbin{\bs{\cal M}}}}
\def\g{{\mathfrak g}}
\def\bS{{\bs S}}
\def\ovY{\kern .04em\overline{\kern -.04em Y}\kern -.2em{}}
\def\al{\alpha}
\def\be{\beta}
\def\ga{\gamma}
\def\de{\delta}
\def\io{\iota}
\def\ep{\epsilon}
\def\la{\lambda}
\def\ka{\kappa}
\def\th{\theta}
\def\ze{\zeta}
\def\up{\upsilon}
\def\vp{\varphi}
\def\si{\sigma}
\def\om{\omega}
\def\De{\Delta}
\def\Ka{{\rm K}}
\def\La{\Lambda}
\def\Om{\Omega}
\def\Ga{\Gamma}
\def\Si{\Sigma}
\def\Th{\Theta}
\def\Up{\Upsilon}
\def\Chi{{\rm X}}
\def\Tau{{\rm T}}
\def\Nu{{\rm N}}
\def\pd{\partial}
\def\ts{\textstyle}
\def\st{\scriptstyle}
\def\sst{\scriptscriptstyle}
\def\w{\wedge}
\def\sm{\setminus}
\def\lt{\ltimes}
\def\bu{\bullet}
\def\sh{\sharp}
\def\di{\diamond}
\def\he{\heartsuit}
\def\od{\odot}
\def\op{\oplus}
\def\ot{\otimes}
\def\bt{\boxtimes}
\def\ov{\overline}
\def\bigop{\bigoplus}
\def\bigot{\bigotimes}
\def\iy{\infty}
\def\es{\emptyset}
\def\ra{\rightarrow}
\def\rra{\rightrightarrows}
\def\Ra{\Rightarrow}
\def\Longra{\Longrightarrow}
\def\ab{\allowbreak}
\def\longra{\longrightarrow}
\def\hookra{\hookrightarrow}
\def\dashra{\dashrightarrow}
\def\lb{\llbracket}
\def\rb{\rrbracket}
\def\ha{{\ts\frac{1}{2}}}
\def\t{\times}
\def\ci{\circ}
\def\ti{\tilde}
\def\d{{\rm d}}
\def\md#1{\vert #1 \vert}
\def\ms#1{\vert #1 \vert^2}
\def\bmd#1{\big\vert #1 \big\vert}
\def\an#1{\langle #1 \rangle}
\pagenumbering{roman}
\title{Erratum to `Orientability of moduli spaces of $\Spin(7)$-instantons and coherent sheaves on Calabi--Yau 4-folds'}
\author{Yalong Cao, Jacob Gross and Dominic Joyce}
\date{arXiv:1811.09658. \\ 
Published as Advances in Mathematics 368 (2020), no. 107134}
\maketitle

\phantomsection
\label{erratum}

The authors regret that there is a mistake in the proof of Theorem \ref{ss1thm1} of this paper, and the theorem itself is false. The mistake in the proof of Theorem \ref{ss1thm1} comes in \S\ref{ss24}, where we claim that the natural map $\pi_5(\SU(4))\cong\Z\ra H_5(\SU(4),\Z)\cong\Z$ is an isomorphism. However, one can prove using Bott \cite[\S 8, final corollary]{Bott} that this map is $24\cdot-:\Z\ra\Z$. Corollaries \ref{ss1cor1} and \ref{ss1cor2} of this paper are deduced from Theorem \ref{ss1thm1}, so they may also be false, though we do not have counterexamples. The authors would like to apologize for all this.

The third author and Markus Upmeier \cite{JoUp2} have developed a new theory for studying orientability and canonical orientations of moduli spaces using `bordism categories'. The mistake in Theorem \ref{ss1thm1} of this paper became clear because it contradicted some of the results of \cite{JoUp2}. In \cite[Th.~12.6(b)]{JoUp2}, they prove the following corrected version of Theorem \ref{ss1thm1} of this paper:
\smallskip

\noindent{\bf Theorem 1.11$'$ (Joyce--Upmeier \cite{JoUp2})} {\it Let\/ $X$ be a compact, oriented, spin Riemannian $8$-manifold, satisfying the condition
\begin{itemize}
\setlength{\itemsep}{0pt}
\setlength{\parsep}{0pt}
\item[$(*)$] There does not exist a class $\al\in H^3(X,\Z)$ such that\/ $\int_X\bar\al\cup\mathop{\rm Sq}^2(\bar\al)=\ul{1}$ in $\Z_2,$ where $\bar\al\in H^3(X,\Z_2)$ is the mod\/ $2$ reduction of\/ $\al,$ and\/ $\mathop{\rm Sq}^2(\bar\al)$ in $H^5(X,\Z_2)$ is its Steenrod square.
\end{itemize}
Let\/ $E_\bu$ be the positive Dirac operator $\slashed{D}_+:\Ga^\iy(S_+)\ra\Ga^\iy(S_-)$ on $X$ in Definition\/ {\rm\ref{ss1def2}}. Suppose\/ $P\ra X$ is a principal\/ $G$-bundle for $G=\U(m)$ or $\SU(m)$. Then $\B_P$ is orientable, that is, $O_P^{E_\bu}\ra\B_P$ is a trivializable principal\/ $\Z_2$-bundle.}

\smallskip

Here condition $(*)$ does not appear in Theorem \ref{ss1thm1} of this paper.

In \cite[Ex.~12.8]{JoUp2}, Joyce--Upmeier show that if $X$ is the compact, oriented, spin 8-manifold $\SU(3)$, which does not satisfy $(*)$, and $P=X\t\SU(3)\ra X$ is the trivial principal $\SU(3)$-bundle, then $\B_P$ is not orientable. This is a counterexample to the original statement of Theorem \ref{ss1thm1}.

Condition $(*)$ is a sufficient condition on $X$ for orientability of $\B_P$ in Theorem 1.11$'$, but it may not be be necessary. The theory of \cite{JoUp2} enables us to give complicated necessary and sufficient conditions for $\B_P$ to be orientable, in terms of spin bordism groups of certain classifying spaces, see \cite{JoUp2} for details.

Using Theorem 1.11$'$ instead of Theorem \ref{ss1thm1}, we obtain the following corrected versions of 
Corollaries \ref{ss1cor1} and \ref{ss1cor2}. (The last parts of Corollary \ref{ss1cor2}, which we have not reproduced, also hold.)
\smallskip

\noindent{\bf Corollary 1.12$'$} {\it Let\/ $(X,\Om,g)$ be a compact\/ $\Spin(7)$-manifold satisfying condition $(*)$ above. Then for any principal\/ $G$-bundle\/ $P\ra X$ for\/ $G=\U(m)$ or\/ $\SU(m),$ the moduli space\/ $\M_P^{\Spin(7)}$ of\/ $\Spin(7)$-instantons on $P$ is orientable, as a manifold or derived manifold.}
\medskip

\noindent{\bf Corollary 1.17$'$} {\it Let\/ $(X,\th)$ be an algebraic Calabi--Yau\/ $4$-fold satisfying condition $(*)$ above. Then the Borisov--Joyce orientation bundle $O^\om\ra\M$ is a trivializable algebraic principal\/ $\Z_2$-bundle, i.e.\ $\M$ is orientable.}
\smallskip

We will not correct the proof of Theorem \ref{ss1thm1} here, as this would take many pages, so we refer interested readers to Joyce--Upmeier \cite{JoUp2}. But we do not wish to withdraw the paper, as it is already published in Advances in Mathematics, and because Theorem \ref{ss1thm2} of this paper (which is unaffected by the mistake in Theorem \ref{ss1thm1}), is very useful, particularly in combination with Theorem 1.11$'$.

Instead, in the version of the paper which follows this Erratum, the text of the paper (including mistakes) is unchanged, \color{red}but the problematic parts of the paper are highlighted in red, like this, \color{black}and the problems are explained in new footnotes in red.\footnote{\color{red}Like this footnote.\color{black}}

\renewcommand{\refname}{References for the Erratum}

\renewcommand{\refname}{References}
\newpage

\pagenumbering{arabic}
\setcounter{page}{1}
\title{Orientability of moduli spaces of $\Spin(7)$-instantons and \\ coherent sheaves on Calabi--Yau 4-folds}
\author{Yalong Cao, Jacob Gross and Dominic Joyce}
\date{}
\maketitle

\begin{abstract}  Suppose $(X,\Om,g)$ is a compact $\Spin(7)$-manifold, e.g.\ a Riemannian 8-manifold with holonomy $\Spin(7)$, or a Calabi--Yau 4-fold. Let $G$ be $\U(m)$ or $\SU(m)$, and $P\ra X$ be a principal $G$-bundle. We show that the infinite-dimensional moduli space $\B_P$ of all connections on $P$ modulo gauge is orientable, in a certain sense. We deduce that the moduli space $\M_P^{\Spin(7)}\subset\B_P$ of irreducible $\Spin(7)$-instanton connections on $P$ modulo gauge, as a manifold or derived manifold, is orientable. This improves theorems of Cao and Leung \cite{CaLe2} and Mu\~noz and Shahbazi~\cite{MuSh}.

If $X$ is a Calabi--Yau 4-fold, the derived moduli stack $\bs\M$ of (complexes of) coherent sheaves on $X$ is a $-2$-shifted symplectic derived stack $(\bs\M,\om)$ by Pantev--To\"en--Vaqui\'e--Vezzosi \cite{PTVV}, and so has a notion of orientation by Borisov--Joyce \cite{BoJo}. We prove that $(\bs\M,\om)$ is orientable, by relating algebro-geometric orientations on $(\bs\M,\om)$ to differential-geometric orientations on $\B_P$ for $\U(m)$-bundles $P\ra X$, and using orientability of~$\B_P$.

This has applications to defining Donaldson--Thomas type invariants counting semistable coherent sheaves on a Calabi--Yau 4-fold, as in Donaldson and Thomas \cite{DoTh}, Cao and Leung \cite{CaLe1}, and Borisov and Joyce~\cite{BoJo}.
\end{abstract}

\setcounter{tocdepth}{2}
\tableofcontents

\section{Introduction}
\label{ss1}

Suppose $(X,g)$ is a compact, connected, oriented, spin Riemannian 8-manifold, with positive Dirac operator $\slashed{D}_+:\Ga^\iy(S_+)\ra\Ga^\iy(S_-)$, an elliptic operator on $X$ that we write as $E_\bu$ for short. Let $G=\U(m)$ or $\SU(m)$ and $P\ra X$ be a principal $G$-bundle. Write $\A_P$ for the infinite-dimensional affine space of connections $\nabla_P$ on $P$. The gauge group $\G_P=\Aut(P)$ acts on $\A_P$, and the centre $Z(G)\subset\cG$ acts trivially. We call $\nabla_P$ {\it irreducible\/} if $\Stab_{\G_P}(\nabla_P)=Z(G)$, and write $\A_P^\irr\subset\A_P$ for the subset of irreducible connections.

We write $\B_P=[\A_P/\G_P]$ for the moduli space of gauge equivalence classes of connections on $P$, considered as a {\it topological stack\/} in the sense of \cite{Metz,Nooh1,Nooh2}. Write $\B_P^\irr=[\A_P^\irr/\G_P]$ for the substack $\B_P^\irr\subseteq\B_P$ of irreducible connections. We also define variations $\ovB_P=[\A_P/(\G_P/Z(G))]$, $\ovB_P^\irr=[\A_P^\irr/(\G_P/Z(G))]$. As $\G_P/Z(G)$ acts freely on $\A_P^\irr$, we may consider $\ovB_P^\irr$ as a topological space (which is an example of a topological stack). There are natural morphisms $\Pi_P:\B_P\ra\ovB_P$, $\Pi_P^\irr:\B^\irr_P\ra\ovB^\irr_P$  of topological stacks, fibrations with fibre~$[*/Z(G)]$.

For each $[\nabla_P]\in\B_P$ we have a twisted operator $\slashed{D}_+^{\nabla_P}:\Ga^\iy(\Ad(P)\ot S_+)\ra\Ga^\iy(\Ad(P)\ot S_-)$ on $X$, a continuous family of elliptic operators over the base $\B_P$. Thus as in Atiyah and Singer \cite{AtSi} we have a family index in $KO^0(\B_P)$, which has an {\it orientation bundle\/} $O_P^{E_\bu}\ra\B_P$, a principal $\Z_2$-bundle parametrizing orientations on $\Ker(\slashed{D}_+^{\nabla_P})\op \Coker(\slashed{D}_+^{\nabla_P})$ at $[\nabla_P]$. An {\it orientation\/} on $\B_P$ is a trivialization $O_P^{E_\bu}\cong\B_P\t\Z_2$. Similarly we define $\bar O_P^{E_\bu}\ra\ovB_P$. There is a natural isomorphism $O_P^{E_\bu}\cong\Pi_P^*(\bar O_P^{E_\bu})$, and orientations on $\B_P$ and $\ovB_P$ are equivalent. 

Our first main result Theorem \ref{ss1thm1} is that $\B_P$ is orientable, for any such $(X,g)$ and $P\ra X$. This extends results by Cao and Leung \cite[Th.~2.1]{CaLe2} and Mu\~noz and Shahbazi \cite{MuSh}. As in \cite[\S 2.5]{JTU} and Remark \ref{ss1rem4}(b) below, once we know the $\B_P$ are orientable, there is a method to choose particular orientations on all $\B_P$, depending only on a finite arbitrary choice. We then apply this to show that two interesting classes of moduli spaces in gauge theory and (derived) algebraic geometry are orientable, and we can choose particular orientations for them in a systematic way.

Suppose $X$ is a compact 8-manifold with a $\Spin(7)$-structure $(\Om,g)$, and $G=\U(m)$ or $\SU(m)$, and $P\ra X$ is a principal $G$-bundle. As in Donaldson and Thomas \cite{DoTh}, a $\Spin(7)$-{\it instanton\/} on $P$ is a connection $\nabla_P$ on $P$ with $\pi^2_7(F^{\nabla_P})=0$, where $\pi^2_7(F^{\nabla_P})$ is a certain component of the curvature $F^{\nabla_P}$ of $\nabla_P$. The deformation theory of $\Spin(7)$-instantons is elliptic, and therefore moduli spaces $\M_P^{\Spin(7)}$ of irreducible $\Spin(7)$-instantons on $P$ modulo gauge are derived manifolds \cite{Joyc2,Joyc4,Joyc5,Joyc6}, and smooth manifolds if $\Om$ is generic. 

We have an inclusion $\M_P^{\Spin(7)}\hookra\ovB^\irr_P\subset\ovB_P$, and orientations on $\ovB_P$ (equivalently, on $\B_P$) restrict to orientations on $\M_P^{\Spin(7)}$ in the usual sense. Thus our theorem implies that all such moduli spaces $\M_P^{\Spin(7)}$ are orientable. This will be important in any future programme to define enumerative invariants of $\Spin(7)$-manifolds by `counting' moduli spaces of $\Spin(7)$-instantons,~\cite{DoTh}.

The analogous problem of orienting anti-self-dual instanton moduli spaces $\M_P^{\rm asd}$ on 4-manifolds $X$ was solved by Donaldson \cite{Dona1,Dona2,DoKr}, and our proof is based on his techniques. However, the 8-dimensional case is considerably more difficult. This is because orientability for $\B_P$ depends on phenomena happening on submanifolds $Z\subset X$ of codimension 3 in $X$. When $X$ is a 4-manifold, such $Z$ are just circles, which are simple. But when $X$ is an 8-manifold, $Z$ is a 5-manifold, and so is much more complicated. Our proof uses the classification of compact, simply-connected 5-manifolds in Crowley~\cite{Crow}.

Next suppose $X$ is a Calabi--Yau $4m$-fold, a projective complex $4m$-manifold with trivial canonical bundle $K_X\cong\O_X$. Write $\bs\M$ for the derived moduli stack of coherent sheaves (or complexes of coherent sheaves) on $X$, as a derived stack in the sense of To\"en and Vezzosi \cite{Toen1,Toen2,ToVe1,ToVe2}, and $\M=t_0(\bs\M)$ for its classical truncation. By Pantev--To\"en--Vaqui\'e--Vezzosi \cite{PTVV}, $\bs\M$ has a $(2-4m)$-shifted symplectic structure $\om$. Borisov and Joyce \cite[\S 2.4]{BoJo} define a notion of orientation on $k$-shifted symplectic derived stacks $(\bs\M,\om)$ for even $k$, which form a principal $\Z_2$-bundle~$O^\om\ra\M$.

Write $\cC=\Map_{C^0}(X,B\U\t\Z)$, where $B\U\t\Z$ is the classifying space for complex K-theory with $B\U=\varinjlim_{n\ra\iy}B\U(n)$, and take $E_\bu$ to be the positive Dirac operator $\slashed{D}_+:\Ga^\iy(X_+)\ra\Ga^\iy(S_-)$ on $X$. Then as in \cite[\S 2.4]{JTU} we can use the principal $\Z_2$-bundles $O^{E_\bu}_P\ra\B_P$ for all principal $\U(n)$-bundles $P\ra X$ to build a natural principal $\Z_2$-bundle $O^{E_\bu}\ra\cC$. If $X$ is a Calabi--Yau 4-fold then the $O^{E_\bu}_P\ra\B_P$ are trivial by Theorem \ref{ss1thm1}, so $O^{E_\bu}\ra\cC$ is trivial.

Our second main result, Theorem \ref{ss1thm2}, says (in part) that for $X$ a Calabi--Yau $4m$-fold we can build a continuous map $\Ga:\M^\top\ra\cC$, where $\M^\top$ is the `topological realization' of $\M$ as in \cite{Blan,Simp1}, and an isomorphism of principal $\Z_2$-bundles $\ga:(O^\om)^\top\ra\Ga^*(O^{E_\bu})$ on $\M^\top$. If $X$ is a Calabi--Yau 4-fold then $O^{E_\bu}$ is trivial from above, so $(O^\om)^\top\ra\M^\top$ is a trivial topological principal $\Z_2$-bundle, and therefore $O^\om\ra\M$ is a trivial algebraic principal $\Z_2$-bundle.

This has applications to the programme of defining Donaldson--Thomas type `DT4 invariants' counting semistable coherent sheaves on Calabi--Yau 4-folds proposed by Donaldson and Thomas \cite{DoTh}, Cao and Leung \cite{CaLe1}, and Borisov and Joyce \cite{BoJo}. The main result of \cite{BoJo} is that given a proper $-2$-shifted symplectic derived scheme $(\bS,\om)$ with an orientation, one can define a virtual cycle $[\bS]_\virt$ in the homology $H_*(\bS,\Z)$. Applying this to semistable derived moduli schemes $(\bs\M_\al^{\rm ss},\om)$ (at least in the case when semistable=stable), using the orientations this paper provides, and pairing the virtual cycle $[\bs\M_\al^{\rm ss}]_\virt$ with certain cohomology classes on $\bs\M_\al^{\rm ss}$, should yield the proposed DT4 invariants.

Sections \ref{ss11}--\ref{ss14} summarize background material on the general theory of orientations in gauge theory from Joyce, Tanaka and Upmeier \cite{JTU}, and on $\Spin(7)$-manifolds and $\Spin(7)$-instantons, and on Calabi--Yau $m$-folds and moduli spaces of (complexes of) coherent sheaves. The main results are stated in \S\ref{ss15}, and the proofs of Theorems \ref{ss1thm1} and \ref{ss1thm2} are given in \S\ref{ss2} and~\S\ref{ss3}.
\medskip

\noindent{\it Acknowledgements.} The first author was supported by a Royal Society Newton International Fellowship. This research was partly funded by a Simons Collaboration Grant on `Special Holonomy in Geometry, Analysis and Physics'. The authors would like to thank Aleksander Doan, Simon Donaldson, Andriy Haydys, Yuuji Tanaka, Markus Upmeier, and Thomas Walpuski for helpful conversations. 

\subsection{\texorpdfstring{Connection moduli spaces $\B_P$ and orientations}{Connection moduli spaces ℬᵨ and orientations}}
\label{ss11}

The following definitions are taken from Joyce, Tanaka and Upmeier~\cite[\S 1--\S 2]{JTU}.

\begin{dfn} Suppose we are given the following data:
\begin{itemize}
\setlength{\itemsep}{0pt}
\setlength{\parsep}{0pt}
\item[(a)] A compact, connected manifold $X$, of dimension $n>0$.
\item[(b)] A Lie group $G$, with $\dim G>0$, and centre $Z(G)\subseteq G$, and Lie algebra $\g$.
\item[(c)] A principal $G$-bundle $\pi:P\ra X$. We write $\Ad(P)\ra X$ for the vector bundle with fibre $\g$ defined by $\Ad(P)=(P\t\g)/G$, where $G$ acts on $P$ by the principal bundle action, and on $\g$ by the adjoint action.
\end{itemize}

Write $\A_P$ for the set of connections $\nabla_P$ on the principal bundle $P\ra X$. This is a real affine space modelled on the infinite-dimensional vector space $\Ga^\iy(\Ad(P)\ot T^*X)$, and we make $\A_P$ into a topological space using the $C^\iy$ topology on $\Ga^\iy(\Ad(P)\ot T^*X)$. Here if $E\ra X$ is a vector bundle then $\Ga^\iy(E)$ denotes the vector space of smooth sections of $E$. Note that $\A_P$ is contractible.
 
Write $\G_P=\Aut(P)$ for the infinite-dimensional Lie group of $G$-equivariant diffeomorphisms $\ga:P\ra P$ with $\pi\ci\ga=\pi$. Then $\G_P$ acts on $\A_P$ by gauge transformations, and the action is continuous for the topology on~$\A_P$. 

There is an inclusion $Z(G)\hookra\G_P$ mapping $z\in Z(G)$ to the principal bundle action of $z$ on $P$. This maps $Z(G)$ into the centre $Z(\G_P)$ of $\G_P$, so we may take the quotient group $\G_P/Z(G)$. The action of $Z(G)\subset\G_P$ on $\A_P$ is trivial, so the $\G_P$-action on $\A_P$ descends to a $\G_P/Z(G)$-action. 

Each $\nabla_P\in\A_P$ has a (finite-dimensional) {\it stabilizer group\/} $\Stab_{\G_P}(\nabla_P)\subset\G_P$ under the $\G_P$-action on $\A_P$, with $Z(G)\subseteq\Stab_{\G_P}(\nabla_P)$. As $X$ is connected, $\Stab_{\G_P}(\nabla_P)$ is isomorphic to a closed Lie subgroup $H$ of $G$ with $Z(G)\subseteq H$. As in \cite[p.~133]{DoKr} we call $\nabla_P$ {\it irreducible\/} if $\Stab_{\G_P}(\nabla_P)=Z(G)$, and {\it reducible\/} otherwise. Write $\A_P^\irr,\A_P^\red$ for the subsets of irreducible and reducible connections in $\A_P$. Then $\A_P^\irr$ is open and dense in $\A_P$, and $\A_P^\red$ is closed and of infinite codimension in the infinite-dimensional affine space $\A_P$.

We write $\B_P=[\A_P/\G_P]$ for the moduli space of gauge equivalence classes of connections on $P$, considered as a {\it topological stack\/} in the sense of Metzler \cite{Metz} and Noohi \cite{Nooh1,Nooh2}. Write $\B_P^\irr=[\A_P^\irr/\G_P]$ for the substack $\B_P^\irr\subseteq\B_P$ of irreducible connections. 

Define variations $\ovB_P=[\A_P/(\G_P/Z(G))]$, $\ovB_P^\irr=[\A_P^\irr/(\G_P/Z(G))]$ of $\B_P,\ab\B_P^\irr$. Then $\ovB_P$ is a topological stack, but as $\G_P/Z(G)$ acts freely on $\A_P^\irr$, we may consider $\ovB_P^\irr$ as a topological space (which is an example of a topological stack). There are natural morphisms $\Pi_P:\B_P\ra\ovB_P$, $\Pi_P^\irr:\B^\irr_P\ra\ovB^\irr_P$.

The inclusions $\B^\irr_P\hookra\B_P$, $\ovB^\irr_P\hookra\ovB_P$ are weak homotopy equivalences of topological stacks in the sense of \cite{Nooh2}. Also $\Pi_P:\B_P\ra\ovB_P$ is a fibration with fibre $[*/Z(G)]$, which is connected and simply-connected, so $\B_P,\ovB_P$ are interchangeable for questions about orientations. Therefore, for the algebraic topological questions that concern us, working on $\ovB^\irr_P$ and on $\B_P$ are essentially equivalent, so we could just consider the topological space $\ovB^\irr_P$, and not worry about topological stacks at all, following most other authors in the area.

The main reason we do not do this in \cite{JTU} is that to relate orientations on different moduli spaces we consider direct sums of connections, which give a morphism $\Phi:\B_P\t\B_Q\ra\B_{P\op Q}$, but this and similar morphisms do not make sense for the spaces $\B_P^\irr,\ovB_P,\ovB_P^\irr$, so we prefer to work with the~$\B_P$.
\label{ss1def1}	
\end{dfn}

We define orientation bundles $O^{E_\bu}_P,\bar O^{E_\bu}_P$ on the moduli spaces~$\B_P,\ovB_P$:

\begin{dfn} 
 Work in the situation of Definition \ref{ss1def1}, with the same notation. Suppose we are given real vector bundles $E_0,E_1\ra X$, of the same rank $r$, and a linear elliptic partial differential operator $D:\Ga^\iy(E_0)\ra\Ga^\iy(E_1)$, of degree $d$. As a shorthand we write $E_\bu=(E_0,E_1,D)$. With respect to connections $\nabla_{E_0}$ on $E_0\ot\bigot^iT^*X$ for $0\le i<d$, when $e\in\Ga^\iy(E_0)$ we may write
\e
D(e)=\ts\sum_{i=0}^d a_i\cdot \nabla_{E_0}^ie,
\label{ss1eq1}
\e
where $a_i\in \Ga^\iy(E_0^*\ot E_1\ot S^iTX)$ for $i=0,\ldots,d$. The condition that $D$ is {\it elliptic\/} is that $a_d\vert_x\cdot\ot^d\xi:E_0\vert_x\ra E_1\vert_x$ is an isomorphism for all $x\in X$ and $0\ne\xi\in T_x^*X$, and the {\it symbol\/} $\si(D)$ of $D$ is defined using~$a_d$.

Let $\nabla_P\in\A_P$. Then $\nabla_P$ induces a connection $\nabla_{\Ad(P)}$ on the vector bundle $\Ad(P)\ra X$. Thus we may form the twisted elliptic operator
\e
\begin{split}
D^{\nabla_{\Ad(P)}}&:\Ga^\iy(\Ad(P)\ot E_0)\longra\Ga^\iy(\Ad(P)\ot E_1),\\
D^{\nabla_{\Ad(P)}}&:e\longmapsto \ts\sum_{i=0}^d (\id_{\Ad(P)}\ot a_i)\cdot \nabla_{\Ad(P)\ot E_0}^ie,
\end{split}
\label{ss1eq2}
\e
where $\nabla_{\Ad(P)\ot E_0}$ are the connections on $\Ad(P)\ot E_0\ot\bigot^iT^*X$ for $0\le i<d$ induced by $\nabla_{\Ad(P)}$ and~$\nabla_{E_0}$.

Since $D^{\nabla_{\Ad(P)}}$ is a linear elliptic operator on a compact manifold $X$, it has finite-dimensional kernel $\Ker(D^{\nabla_{\Ad(P)}})$ and cokernel $\Coker(D^{\nabla_{\Ad(P)}})$. The {\it determinant\/} $\det(D^{\nabla_{\Ad(P)}})$ is the 1-dimensional real vector space
\begin{equation*}
\det(D^{\nabla_{\Ad(P)}})=\det\Ker(D^{\nabla_{\Ad(P)}})\ot\bigl(\det\Coker(D^{\nabla_{\Ad(P)}})\bigr)^*,
\end{equation*}
where if $V$ is a finite-dimensional real vector space then $\det V=\La^{\dim V}V$.

These operators $D^{\nabla_{\Ad(P)}}$ vary continuously with $\nabla_P\in\A_P$, so they form a family of elliptic operators over the base topological space $\A_P$. Thus as in Atiyah and Singer \cite{AtSi}, there is a natural real line bundle $\hat L{}^{E_\bu}_P\ra\A_P$ with fibre $\hat L{}^{E_\bu}_P\vert_{\nabla_P}=\det(D^{\nabla_{\Ad(P)}})$ at each $\nabla_P\in\A_P$. It is equivariant under the actions of $\G_P$ and $\G_P/Z(G)$ on $\A_P$, and so pushes down to real line bundles $L^{E_\bu}_P\ra\B_P$, $\bar L^{E_\bu}_P\ra\ovB_P$ on the topological stacks $\B_P,\ovB_P$, with $L^{E_\bu}_P\cong\Pi_P^*(\bar L_P^{E_\bu})$. We call $L^{E_\bu}_P,\bar L^{E_\bu}_P$ the {\it determinant line bundles\/} of $\B_P,\ovB_P$. The restriction $\bar L^{E_\bu}_P\vert_{\ovB_P^\irr}$ is a topological real line bundle in the usual sense on the topological space~$\ovB_P^\irr$.

Define the {\it orientation bundle\/} $O^{E_\bu}_P$ of $\B_P$ by $O^{E_\bu}_P=(L^{E_\bu}_P\sm 0(\B_P))/(0,\iy)$. That is, we take the complement $L^{E_\bu}_P\sm 0(\B_P)$ of the zero section $0(\B_P)$ in $L^{E_\bu}_P$, and quotient by $(0,\iy)$ acting on the fibres of $L^{E_\bu}_P\sm 0(\B_P)\ra\B_P$ by multiplication. Then $L^{E_\bu}_P\ra\B_P$ descends to $\pi:O^{E_\bu}_P\ra\B_P$, which is a bundle with fibre $(\R\sm\{0\})/(0,\iy)\cong\{1,-1\}=\Z_2$, since $L^{E_\bu}_P\ra\B_P$ is a fibration with fibre $\R$. That is, $\pi:O^{E_\bu}_P\ra\B_P$ is a {\it principal\/ $\Z_2$-bundle}, in the sense of topological stacks.

Similarly we define a principal $\Z_2$-bundle $\bar\pi:\bar O^{E_\bu}_P\ra\ovB_P$ from $\bar L^{E_\bu}_P$, and as $L^{E_\bu}_P\cong\Pi_P^*(\bar L_P^{E_\bu})$ we have a canonical isomorphism $O^{E_\bu}_P\cong\Pi_P^*(\bar O_P^{E_\bu})$. The fibres of $O^{E_\bu}_P\ra\B_P$, $\bar O^{E_\bu}_P\ra\ovB_P$ are orientations on the real line fibres of $L^{E_\bu}_P\ra\B_P$, $\bar L^{E_\bu}_P\ra\ovB_P$. The restriction $\bar O^{E_\bu}_P\vert_{\ovB^\irr_P}$ is a principal $\Z_2$-bundle on the topological space $\ovB^\irr_P$, in the usual sense.

We say that $\B_P$ is {\it orientable\/} if $O^{E_\bu}_P$ is isomorphic to the trivial principal $\Z_2$-bundle $\B_P\t\Z_2\ra\B_P$. An {\it orientation\/} $\om$ on $\B_P$ is an isomorphism $\om:O^{E_\bu}_P\,{\buildrel\cong\over\longra}\,\B_P\t\Z_2$ of principal $\Z_2$-bundles. We make the same definitions for $\ovB_P$ and $\bar O^{E_\bu}_P$. Since $\Pi_P:\B_P\ra\ovB_P$ is a fibration with fibre $[*/Z(G)]$, which is connected and simply-connected, and $O^{E_\bu}_P\cong\Pi_P^*(\bar O_P^{E_\bu})$, we see that $\B_P$ is orientable if and only if $\ovB_P$ is, and orientations of $\B_P$ and $\ovB_P$ correspond. As $\B_P$ is connected, if $\B_P$ is orientable it has exactly two orientations.

We also define the {\it normalized orientation bundle\/} $\check O_P^{E_\bu}\ra\B_P$ by
\begin{equation*}
\check O_P^{E_\bu}=O_P^{E_\bu}\ot_{\Z_2}O_{X\t G}^{E_\bu}\vert_{[\nabla^0]}.
\end{equation*}
That is, we tensor the orientation bundle with the orientation torsor $O_{X\t G}^{E_\bu}\vert_{[\nabla^0]}$ of the trivial principal $G$-bundle $X\t G\ra X$ at the trivial connection $\nabla^0$. Then $\check O_{X\t G}^{E_\bu}\vert_{[\nabla^0]}=\Z_2$ is canonically trivial. Since we have natural isomorphisms
\begin{equation*}
\Ker(D^{\nabla^0_{\Ad(P)}})\cong\g\ot\Ker D,\qquad 
\Coker(D^{\nabla^0_{\Ad(P)}})\cong\g\ot\Coker D,
\end{equation*}
we see that (using an orientation convention) there is a natural isomorphism
\begin{equation*}
O_{X\t G}^{E_\bu}\vert_{[\nabla^0]}\cong \Or(\det D)^{\ot^{\dim\g}}\ot_{\Z_2}\Or(\g)^{\ot^{\ind D}},
\end{equation*}
where $\Or(\det D),\Or(\g)$ are the $\Z_2$-torsors of orientations on $\det D$ and $\g$. Thus, choosing orientations for $\det D$ and $\g$ gives an isomorphism $\check O_P^{E_\bu}\cong O_P^{E_\bu}$.

Normalized orientation bundles are convenient because they behave nicely under the Excision Theorem, Theorem \ref{ss2thm1} below. Note that $O^{E_\bu}_P$ is trivializable if and only if $\check O_P^{E_\bu}$ is, so for questions of orientability there is no difference.
\label{ss1def2}	
\end{dfn}

\begin{rem}{\bf(i)} Up to continuous isotopy, and hence up to isomorphism, $L^{E_\bu}_P,O^{E_\bu}_P$ in Definition \ref{ss1def2} depend on the elliptic operator $D:\Ga^\iy(E_0)\ra\Ga^\iy(E_1)$ up to continuous deformation amongst elliptic operators, and thus only on the {\it symbol\/} $\si(D)$ of $D$ (essentially, the highest order coefficients $a_d$ in \eq{ss1eq1}), up to deformation.
\smallskip

\noindent{\bf(ii)} For orienting moduli spaces of `instantons' in gauge theory, as in \S\ref{ss12}, we usually start not with an elliptic operator on $X$, but with an {\it elliptic complex\/}
\e
\smash{\xymatrix@C=28pt{ 0 \ar[r] & \Ga^\iy(E_0) \ar[r]^{D_0} & \Ga^\iy(E_1) \ar[r]^(0.55){D_1} & \cdots \ar[r]^(0.4){D_{k-1}} & \Ga^\iy(E_k) \ar[r] & 0. }}
\label{ss1eq3}
\e
If $k>1$ and $\nabla_P$ is an arbitrary connection on a principal $G$-bundle $P\ra X$ then twisting \eq{ss1eq3} by $(\Ad(P),\nabla_{\Ad(P)})$ as in \eq{ss1eq2} may not yield a complex (that is, we may have $D^{\nabla_{\Ad(P)}}_{i+1}\ci D^{\nabla_{\Ad(P)}}_i\ne 0$), so the definition of $\det(D_\bu^{\nabla_{\Ad(P)}})$ does not work, though it does work if $\nabla_P$ satisfies the appropriate instanton-type curvature condition. To get round this, we choose metrics on $X$ and the $E_i$, so that we can take adjoints $D_i^*$, and replace \eq{ss1eq3} by the elliptic operator
\e
\smash{\xymatrix@C=90pt{ \Ga^\iy\bigl(\bigop_{0\le i\le k/2}E_{2i}\bigr) \ar[r]^(0.48){\sum_i(D_{2i}+D_{2i-1}^*)} & \Ga^\iy\bigl(\bigop_{0\le i< k/2}E_{2i+1}\bigr), }}
\label{ss1eq4}
\e
and then Definition \ref{ss1def2} works with \eq{ss1eq4} in place of~$E_\bu$.
\label{ss1rem1}	
\end{rem}

\subsection{Orienting moduli spaces in gauge theory}
\label{ss12}

In gauge theory one studies moduli spaces $\M_P^{\rm ga}$ of (irreducible) connections $\nabla_P$ on a principal bundle $P\ra X$ (perhaps plus some extra data, such as a Higgs field) satisfying a curvature condition. Under suitable genericity conditions, these moduli spaces $\M_P^{\rm ga}$ will be smooth manifolds, and the ideas of \cite{JTU} can often be used to prove $\M_P^{\rm ga}$ is orientable, and construct a canonical orientation on $\M_P^{\rm ga}$. These orientations are important in defining enumerative invariants such as Casson invariants, Donaldson invariants, and Seiberg--Witten invariants. 
We illustrate this with the example of instantons on 4-manifolds,~\cite{DoKr}:

\begin{ex}
\label{ss1ex1}	
Let $(X,g)$ be a compact, oriented Riemannian 4-manifold, and $G$ a Lie group (e.g.\ $G=\SU(2)$), and $P\ra X$ a principal $G$-bundle. For each connection $\nabla_P$ on $P$, the curvature $F^{\nabla_P}$ is a section of $\Ad(P)\ot\La^2T^*X$. We have $\La^2T^*X=\La^2_+T^*X\op\La^2_-T^*X$, where $\La^2_\pm T^*X$ are the subbundles of 2-forms $\al$ on $X$ with $*\al=\pm\al$. Thus $F^{\nabla_P}=F^{\nabla_P}_+\op F^{\nabla_P}_-$, with $F^{\nabla_P}_\pm$ the component in $\Ad(P)\ot\La^2_\pm T^*X$. We call $(P,\nabla_P)$ an ({\it anti-self-dual\/}) {\it instanton\/} if $F^{\nabla_P}_+=0$.

Write $\M_P^{\rm asd}$ for the moduli space of gauge isomorphism classes $[\nabla_P]$ of irreducible instanton connections $\nabla_P$ on $P$, modulo $\G_P/Z(G)$. The deformation theory of $[\nabla_P]$ in $\M_P^{\rm asd}$ is governed by the Atiyah--Hitchin--Singer complex \cite{AHS}: 
\e
\begin{gathered} 
\xymatrix@C=8pt@R=4pt{ 0 \ar[rr] && 
\Ga^{\iy} ( \Ad(P) \ot \La^0T^*X ) \ar[rrr]^{\d^{\nabla_P}} &&& 
\Ga^{\iy} ( \Ad(P) \ot \La^1T^*X  ) \\
&& {\qquad\qquad\qquad} \ar[rrr]^{\d^{\nabla_P}_+} &&&
\Ga^{\iy} ( \Ad(P) \ot \La^2_+T^*X  ) \ar[rr] && 0, } 
\end{gathered}
\label{ss1eq5}
\e
where $\d^{\nabla_P}_+\ci\d^{\nabla_P}=0$ as $F^{\nabla_P}_+=0$. Write $\cH^0,\cH^1,\cH^2_+$ for the cohomology groups of \eq{ss1eq5}. Then $\cH^0$ is the Lie algebra of $\Aut(\nabla_P)$, so $\cH^0=Z(\g)$, the Lie algebra of the centre $Z(G)$ of $G$, as $\nabla_P$ is irreducible. Also $\cH^1$ is the Zariski tangent space of $\M_P^{\rm asd}$ at $[\nabla_P]$, and $\cH^2_+$ is the obstruction space. If $g$ is generic then as in \cite[\S 4.3]{DoKr}, for non-flat connections $\cH^2_+=0$ for all $[\nabla_P]\in\M_P^{\rm asd}$, and $\M_P^{\rm asd}$ is a smooth manifold, with tangent space $T_{[\nabla_P]}\M_P^{\rm asd}=\cH^1$. Note that $\M_P^{\rm asd}\subset\ovB_P$ is a subspace of the topological stack $\ovB_P$ from Definition~\ref{ss1def1}.

Take $E_\bu$ to be the elliptic operator on $X$
\begin{equation*}
D=\d+\d_+^*:\Ga^\iy(\La^0T^*X\op\La^2_+T^*X)\longra\Ga^\iy(\La^1T^*X).	
\end{equation*}
Turning the complex \eq{ss1eq5} into a single elliptic operator as in Remark \ref{ss1rem1}(ii) yields the twisted operator $D^{\nabla_{\Ad(P)}}$ from \eq{ss1eq2}. Hence the line bundle $\bar L^{E_\bu}_P\ra\ovB_P$ in Definition \ref{ss1def2} has fibre at $[\nabla_P]$ the determinant line of \eq{ss1eq5}, which (after choosing an isomorphism $\det Z(\g)\cong\R$) is $\det(\cH^1)^*=\det T^*_{[\nabla_P]}\M_P^{\rm asd}$. It follows that $\bar O_P^{E_\bu}\vert_{\M_P^{\rm asd}}$ is the orientation bundle of the manifold $\M_P^{\rm asd}$, and an orientation on $\ovB_P$ in Definition \ref{ss1def2} restricts to an orientation on the manifold $\M_P^{\rm asd}$ in the usual sense of differential geometry. This is a very useful way of defining orientations on $\M_P^{\rm asd}$, first used by Donaldson~\cite{Dona1,Dona2,DoKr}.
\end{ex}

There are several other important classes of gauge-theoretic moduli spaces $\M_P^{\rm ga}$ which have elliptic deformation theory, and so are generically smooth manifolds, for which orientations can be defined by pullback from $\ovB_P$. These include $\Spin(7)$-instantons, as in~\S\ref{ss13}.

\begin{rem} If we omit the genericness/transversality conditions, gauge theory moduli spaces $\M_P^{\rm ga}$ are generally not smooth manifolds. However, as long as their deformation theory is given by an elliptic complex similar to \eq{ss1eq5} whose cohomology is constant except at the second and third terms, $\M_P^{\rm ga}$ will still be a {\it derived smooth manifold\/} ({\it d-manifold}, or {\it m-Kuranishi space\/}) in the sense of Joyce \cite{Joyc2,Joyc4,Joyc5,Joyc6}. Orientations for derived manifolds are defined and well behaved, and we can define orientations on $\M_P^{\rm ga}$ by pullback of orientations on $\ovB_P$ exactly as in the case when $\M_P^{\rm ga}$ is a manifold.
\label{ss1rem2}
\end{rem}

\subsection{\texorpdfstring{$\Spin(7)$-manifolds and $\Spin(7)$-instantons}{Spin(7)-manifolds and Spin(7)-instantons}}
\label{ss13}

Next we discuss the exceptional holonomy group $\Spin(7)$ in 8 dimensions, and $\Spin(7)$-instantons on compact 8-manifolds with holonomy in $\Spin(7)$. See Joyce \cite[\S 10]{Joyc1} for background on the exceptional holonomy group~$\Spin(7)$.

\begin{dfn} Let $\R^8$ have coordinates $(x_1,\dots,x_8)$. Write $\d{\bf x}_{ijkl}$ for the 4-form $\d x_i\w\d x_j\w\d x_k\w\d x_l$ on $\R^8$. Define a 4-form $\Om_0$ on $\R^8$ by
\begin{align*}
\Om_0=\phantom{-}&\d{\bf x}_{1234}+\d{\bf x}_{1256} +\d{\bf
x}_{1278}+\d{\bf
x}_{1357}-\d{\bf x}_{1368}-\d{\bf x}_{1458}-\d{\bf x}_{1467}\\
-\,&\d{\bf x}_{2358}-\d{\bf x}_{2367}-\d{\bf x}_{2457}+\d{\bf
x}_{2468}+\d{\bf x}_{3456}+\d{\bf x}_{3478}+\d{\bf x}_{5678}.
\end{align*}
The subgroup of $\GL(8,\R)$ preserving $\Om_0$ is the holonomy
group $\Spin(7)$. It is a compact, connected, simply-connected,
semisimple, 21-dimensional Lie group, which is isomorphic to the double cover of $\SO(7)$. This group also preserves the orientation on $\R^8$ and the Euclidean metric $g_0=\d x_1^2+\cdots+\d x_8^2$ on~$\R^8$. 

Let $X$ be an 8-manifold. A $\Spin(7)$-{\it structure\/} $(\Om,g)$ on $X$ is a 4-form $\Om$ and Riemannian metric $g$ on $X$, such that for all $x\in X$ there exist isomorphisms $T_xX\cong\R^8$ identifying $\Om\vert_x\cong\Om_0$ and $g\vert_x\cong g_0$. We call $(\Om,g)$ {\it torsion-free\/} if $\d\Om=0$. This implies that $\Hol(g)\subseteq\Spin(7)$. A $\Spin(7)$-structure $(\Om,g)$ induces a splitting $\La^2T^*X=\La^2_7T^*X \op\La^2_{21}T^*X$ into vector subbundles of ranks $7,21$, the eigenspaces of~$\al\mapsto *(\al\w\Om)$.

A $\Spin(7)$-{\it manifold\/} $(X,\Om,g)$ is an 8-manifold $X$ with a torsion-free $\Spin(7)$-structure $(\Om,g)$. Examples of compact $\Spin(7)$-manifolds with holonomy $\Spin(7)$ were constructed by Joyce~\cite[\S 13--\S 15]{Joyc1}. Calabi--Yau 4-folds, and hyperk\"ahler 8-manifolds, are also $\Spin(7)$-manifolds.
\label{ss1def3}	
\end{dfn}

\begin{dfn} Let $(X,\Om,g)$ be a compact $\Spin(7)$-manifold, $G$ a Lie group, and $P\ra X$ a principal $G$-bundle. A $\Spin(7)$-{\it instanton\/} on $P$ is a connection $\nabla_P$ on $P$, whose curvature satisfies $\pi^2_7(F^{\nabla_P})=0$ in $\Ga^\iy(\Ad(P)\ot\La^2_7T^*X)$. Write $\M_P^{\Spin(7)}$ for the moduli space of irreducible $\Spin(7)$-instantons on $P$, modulo gauge transformations of $P$. Then $\M_P^{\Spin(7)}$ is a derived manifold, which is a manifold for non-flat connections if $\Om$ is generic. (We do not need $\d\Om=0$ here.)

Donaldson and Thomas \cite{DoTh} discussed $\Spin(7)$-instantons, proposing research directions, and examples of $\Spin(7)$-instantons on compact $\Spin(7)$-manifolds with holonomy $\Spin(7)$ were given by Lewis \cite{Lewi}, Tanaka \cite{Tana}, and Walpuski~\cite{Walp}. 

\label{ss1def4}	
\end{dfn}

To apply \S\ref{ss12} to $\Spin(7)$-instantons, we replace \eq{ss1eq5} by the complex:
\begin{equation*} 
\xymatrix@C=8pt@R=3pt{ 0 \ar[rr] && 
\Ga^{\iy} ( \Ad(P) \ot \La^0T^*X ) \ar[rrr]^{\d^{\nabla_P}} &&& 
\Ga^{\iy} ( \Ad(P) \ot \La^1T^*X  ) \\
&& {\qquad\qquad\qquad} \ar[rrr]^{\d^{\nabla_P}_7} &&&
\Ga^{\iy} ( \Ad(P) \ot \La^2_7T^*X  ) \ar[rr] && 0. } 
\end{equation*}
The orientation bundle of $\M_P^{\Spin(7)}$ is the pullback of $\bar O_P^{E_\bu}\ra\ovB_P$ in Definition \ref{ss1def2} under the inclusion $\M_P^{\Spin(7)}\hookra\ovB_P$, where $E_\bu$ is the elliptic operator
\begin{equation*}
D=\d+\d^*_7:\Ga^\iy(\La^0T^*X\op\La^2_7T^*X)\longra\Ga^\iy(\La^1T^*X).	
\end{equation*}
The symbol of $E_\bu$ is that of the positive Dirac operator $\slashed{D}_+:\Ga^\iy(S_+)\ra\Ga^\iy(S_-)$ on $X$, which makes sense on general oriented, spin Riemannian 8-manifolds.

\subsection{\texorpdfstring{Calabi--Yau $m$-folds and coherent sheaves}{Calabi--Yau m-folds and coherent sheaves}}
\label{ss14}

\begin{dfn} An {\it algebraic Calabi--Yau\/ $m$-fold\/} $(X,\th)$ is a connected smooth projective $\C$-scheme of complex dimension $m$ with a section $\th\in H^0(K_X)$ of the canonical bundle $K_X\ra X$ inducing an isomorphism~$\th:\O_X\ra K_X$.

A {\it differential-geometric Calabi--Yau\/ $m$-fold\/} $(X^\ran,J,g,\th)$ is a connected projective complex manifold $(X^\ran,J)$ of complex dimension $m$, equipped with a K\"ahler metric $g$ and a holomorphic $(m,0)$-form $\th\in\Ga(K_{X^\ran})$ satisfying
\e
\om^m=(-1)^{m(m-1)/2}i^m2^{-m}m!\cdot\th\w\bar\th,
\label{ss1eq6}
\e
where $\om$ is the K\"ahler form of $g$. This implies that $g$ is Ricci-flat with holonomy $\Hol(g)\subseteq\SU(m)$, and~$K_{X^\ran}\cong\O_{X^\ran}$.
\label{ss1def5}	
\end{dfn}

\begin{rem}{\bf(a)} Algebraic and differential-geometric Calabi--Yau $m$-folds are more-or-less equivalent. If $(X,\th)$ is an algebraic Calabi--Yau $m$-fold then the underlying complex analytic topological space $X^\ran$ is a smooth manifold with complex structure $J$, so that $(X^\ran,J)$ is a projective complex manifold, and $\th$ is equivalent to a non-vanishing holomorphic $(m,0)$-form $\th\in\Ga(K_{X^\ran})$. As $(X^\ran,J)$ is projective it admits K\"ahler metrics $g$. Using the Calabi Conjecture \cite[\S 5]{Joyc1} we can choose $g$ to be Ricci-flat, and rescaling $g$ we can arrange that \eq{ss1eq6} holds, so $(X^\ran,J,g,\th)$ is a differential-geometric Calabi--Yau $m$-fold.

Conversely, if $(X^\ran,J,g,\th)$ is a differential-geometric Calabi--Yau $m$-fold then as $(X^\ran,J)$ is a projective complex manifold it comes from a smooth projective $\C$-scheme $X$, and $(X,\th)$ is an algebraic Calabi--Yau $m$-fold.
\smallskip

\noindent{\bf(b)} Our definitions of Calabi--Yau $m$-fold require only that $\Hol(g)\subseteq\SU(m)$, not that $\Hol(g)=\SU(m)$, so they include examples such as compact hyperk\"ahler manifolds with $\Hol(g)=\Sp(k)$, and flat tori $T^{2m}$ with $\Hol(g)=\{1\}$. We allow this as Theorem \ref{ss1thm2} and Corollary \ref{ss1cor2} below hold in this generality.
\label{ss1rem3}	
\end{rem}

One can study Calabi--Yau 4-folds using Differential Geometry, or Algebraic Geometry (including Derived Algebraic Geometry), or a combination of both, as we do in this paper. The (Derived) Algebraic Geometry we will use involves a large amount of difficult background material, such as derived categories $D^b\coh(X)$, stacks, higher stacks, and derived stacks, which we do not have space to explain. So we give a list of references here, and later assume that those reading Theorem \ref{ss1thm2} and its proof have the necessary background.
\begin{itemize}
\setlength{\itemsep}{0pt}
\setlength{\parsep}{0pt}
\item[(i)] For Calabi--Yau manifolds from a mostly differential-geometric point of view, see Joyce~\cite[\S 6]{Joyc1}.
\item[(ii)] For foundations of Algebraic Geometry, and schemes, see Hartshorne~\cite{Hart}.
\item[(iii)] Let $X$ be a smooth projective $\C$-scheme, for instance, an algebraic Calabi--Yau $m$-fold $(X,\th)$. Then one can consider coherent sheaves on $X$, including (algebraic) vector bundles (i.e.\ locally free coherent sheaves). Write $\coh(X)$ for the abelian category of coherent sheaves on $X$. See Hartshorne \cite[\S II.5]{Hart} and Huybrechts and Lehn~\cite{HuLe}.

We can also consider the bounded derived category $D^b\coh(X)$ of complexes of coherent sheaves, as a triangulated category. For triangulated categories and derived categories see Gelfand and Manin \cite{GeMa}, and for properties of $D^b\coh(X)$ see Huybrechts~\cite{Huyb}.
\item[(iv)] Suppose we wish to define a moduli space $\M$ of objects in $\coh(X)$ or $D^b\coh(X)$. There are four different classes of space we use could do this:
\begin{itemize}
\setlength{\itemsep}{0pt}
\setlength{\parsep}{0pt}
\item[(a)] We could take $\M$ to be a scheme. This usually requires restricting to moduli of stable or semistable coherent sheaves. See Huybrechts and Lehn~\cite{HuLe}.
\item[(b)] The moduli space $\M$ of all objects in $\coh(X)$ is an Artin stack. See G\'omez \cite{Gome}, Laumon and Moret-Bailly \cite{LaMo} and Olsson~\cite{Olss}.
\item[(c)] The moduli space $\M$ of all objects in $D^b\coh(X)$ is a higher stack.
\item[(d)] We can also form a moduli space $\bs\M$ of all objects in $\coh(X)$ or $D^b\coh(X)$ in the sense of Derived Algebraic Geometry, as a derived stack. It has a classical truncation $\M=t_0(\bs\M)$ which is an Artin stack or higher stack, as in (b),(c). For Derived Algebraic Geometry and derived and higher stacks, see To\"en and Vezzosi~\cite{Toen1,Toen2,ToVe1,ToVe2}.
\end{itemize}
\item[(v)] As in Simpson \cite{Simp1} and Blanc \cite[\S 3.1]{Blan}, any Artin $\C$-stack or higher $\C$-stack $\M$ has a {\it topological realization\/} $\M^\top$, which is a topological space (in fact, a CW-complex) natural up to homotopy equivalence. Topological realization gives a functor $(-)^\top:\Ho(\HSta_\C)\ra\Top^{\bf ho}$ from the homotopy category $\Ho(\HSta_\C)$ of the $\iy$-category $\HSta_\C$ of higher $\C$-stacks to the category $\Top^{\bf ho}$ of topological spaces with morphisms homotopy classes of continuous maps. Algebraic principal $\Z_2$-bundles $P\ra\M$ lift to topological principal $\Z_2$-bundles $P^\top\ra\M^\top$, in such a way that algebraic trivializations $P\cong\M\t\Z_2$ correspond naturally to topological trivializations~$P^\top\cong\M^\top\t\Z_2$.
\item[(vi)] Pantev, To\"en, Vaqui\'e and Vezzosi \cite{PTVV} introduced a theory of shifted symplectic Derived Algebraic Geometry, defining $k$-{\it shifted symplectic structures\/} $\om$ on a derived stack $\bs\cS$ for $k\in\Z$. If $(X,\th)$ is an algebraic Calabi--Yau $m$-fold and $\bs\M$ is a derived moduli stack of objects in $\coh(X)$ or $D^b\coh(X)$ then $\bs\M$ has a $(2-m)$-shifted symplectic structure,~\cite[Cor.~2.13]{PTVV}.
\end{itemize}

Borisov and Joyce \cite[\S 2.4]{BoJo} define orientations in shifted symplectic geometry.

\begin{dfn}
\label{ss1def6}
Let $(\bS,\om)$ be a $k$-shifted symplectic derived $\C$-stack for $k$ even, e.g. $k=-2$, and $S=t_0(\bS)$ the classical truncation. All derived stacks $\bS$ in this paper are assumed to be locally finitely presented, as in To\"en and Vezzosi \cite[Def.~1.3.6.4]{ToVe2}. This implies that $\bS$ has a cotangent complex $\bL_\bS$ which is perfect \cite[Prop.~2.2.2.4]{ToVe2}, and has a dual tangent complex $\bT_\bS=(\bL_\bS)^\vee$. As $\bL_\bS,\bT_\bS$ are perfect the restrictions $\bL_\bS\vert_S,\bT_\bS\vert_S$ have determinant line bundles $\det(\bL_\bS\vert_S),\det(\bT_\bS\vert_S)$ over $S$, with~$\det(\bT_\bS\vert_S)\cong\det(\bL_\bS\vert_S)^*$. 

The symplectic structure $\om$ induces an isomorphism $\om\cdot{}:\bT_\bS[-1]\ra\bL_\bS[k-1]$, which yields an isomorphism $\det(\bT_\bS\vert_S)^{-1}\ra(\det\bL_\bS\vert_S)^{(-1)^{k-1}}=\det\bL_\bS\vert_S^{-1}$ on determinant line bundles as $k$ is even. Combining this with $\det(\bT_\bS\vert_S)\cong\det(\bL_\bS\vert_S)^{-1}$ gives a natural isomorphism~$\io^\om:\det(\bL_\bS\vert_S)\ra \det(\bL_\bS\vert_S)^{-1}$. 

Following Borisov and Joyce \cite[Def.~2.12]{BoJo}, an {\it orientation\/} for $(\bS,\om)$ is a choice of isomorphism $o^\om:\det(\bL_\bS\vert_S)\ra\O_S$ such that $(o^\om)^*\ci o^\om=\io^\om$. There is a principal $\Z_2$-bundle $\pi:O^\om\ra S$ called the {\it orientation bundle\/} of $(\bS,\om)$, which parametrizes \'etale local choices of $o^\om$, such that orientations for $(\bS,\om)$ correspond to global trivializations~$O^\om\cong S\t\Z_2$.
\end{dfn}

\subsection{The main results}
\label{ss15}

Here is our first main result. It will be proved in \S\ref{ss2} using a wide range of ideas and techniques, including much of the general theory of orientations in \cite{JTU,JoUp,Upme}, some surgery theory, some special geometry of $\SU(4)$, and the classification of compact simply-connected 5-manifolds in Crowley~\cite{Crow}.

\color{red}
\begin{thm}\footnote{\color{red}Unfortunately, Theorem \ref{ss1thm1} is false without extra assumptions on $X$. There is a mistake in the proof, in \S\ref{ss24}, which is highlighted there. See the \hyperref[erratum]{Erratum} for a corrected version of Theorem \ref{ss1thm1}, and for more details.\color{black}} Let\/ $X$ be a compact, oriented, spin Riemannian $8$-manifold, and\/ $E_\bu$ be the positive Dirac operator $\slashed{D}_+:\Ga^\iy(S_+)\ra\Ga^\iy(S_-)$ on $X$ in Definition\/ {\rm\ref{ss1def2}}. Suppose\/ $P\ra X$ is a principal\/ $G$-bundle for $G=\U(m)$ or $\SU(m)$. Then $\B_P$ is orientable, that is, $O_P^{E_\bu}\ra\B_P$ is a trivializable principal\/ $\Z_2$-bundle.

The orientation bundle $O^{E_\bu}\ra\cC$ over the mapping space $\cC=\Map_{C^0}(X,\ab B\U\t\Z)$ defined in\/ {\rm\cite[\S 2.4.2]{JTU}} and\/ {\rm\S\ref{ss32}} below is also trivializable.
\label{ss1thm1}
\end{thm}
\color{black}

The first part was previously proved by Cao and Leung \cite[Th.~2.1]{CaLe2} in the special case that $G=\U(m)$ and $H_{\rm odd}(X,\Z)=0$, and by Mu\~noz and Shahbazi \cite{MuSh} in the special case that $G=\SU(m)$ and $\Hom(H^3(X,\Z),\Z_2)=0$.

As in \S\ref{ss11}, orientations on $\B_P$ are equivalent to orientations on $\ovB_P$, and as in \S\ref{ss13}, if $(X,\Om,g)$ is a $\Spin(7)$-manifold and $P\ra X$ a principal $G$-bundle then orientations on $\ovB_P$ restrict to orientations on $\M_P^{\Spin(7)}$, giving:

\color{red}\begin{cor}\footnote{\color{red} Corollary \ref{ss1cor1} depends on Theorem \ref{ss1thm1}, which is false. So Corollary \ref{ss1cor1} may also be false, although we do not have a counterexample. See the \hyperref[erratum]{Erratum} for a corrected version of Corollary \ref{ss1cor1}.\color{black}} Let\/ $(X,\Om,g)$ be a compact\/ $\Spin(7)$-manifold. Then for any principal\/ $G$-bundle\/ $P\ra X$ for\/ $G=\U(m)$ or\/ $\SU(m),$ the moduli space\/ $\M_P^{\Spin(7)}$ of\/ $\Spin(7)$-instantons on $P$ is orientable, as a manifold or derived manifold.
\label{ss1cor1}
\end{cor}
\color{black}

Corollary \ref{ss1cor1} will be an important ingredient in any future programme to define Donaldson-invariant style enumerative invariants of $\Spin(7)$-manifolds $(X,\Om,g)$ by `counting' suitably compactified moduli spaces $\M_P^{\Spin(7)}$, as in~\cite{DoTh}.

\begin{rem}{\bf(a)} In a companion paper, Joyce and Upmeier \cite{JoUp} prove that if $(X,g)$ is a compact, oriented, spin Riemannian 7-manifold, and $E_\bu$ is the Dirac operator on $X$, and we choose an orientation on $\det D$ and a `flag structure' on $X$ (an algebro-topological structure on odd-dimensional manifolds defined in Joyce \cite[\S 3.1]{Joyc3}), then we can construct canonical orientations on $\B_P$ for all principal $\U(m)$- or $\SU(m)$-bundles $P\ra X$. Thus if $(X,\vp,g)$ is a compact torsion-free $G_2$-manifold, we can construct canonical orientations on moduli spaces $\M_P^{G_2}$ of $G_2$-instantons on~$P$.

The authors know how to define an analogue of flag structures for compact, spin 8-manifolds $X$, such that if we choose one of these structures on $X$ and an orientation of $\det \slashed{D}_+$, then we can improve Theorem \ref{ss1thm1} and Corollary \ref{ss1cor1} to construct canonical orientations on $\B_P$ and $\M_P^{\Spin(7)}$. However, these analogues of flag structures are more complicated and less attractive than flag structures, and we have decided not to write them up for the present.
\smallskip

\noindent{\bf(b)} In Definition \ref{ss1def2}, if we assume that $\B_P$ is orientable for all principal $\U(m)$-bundles $P\ra X$ (as in Theorem \ref{ss1thm1} for $\slashed{D}_+$ on spin 8-manifolds), then Joyce, Tanaka and Upmeier \cite[\S 2.5]{JTU} give a method to define canonical orientations on $\B_P$ for all principal $\U(m)$- and $\SU(m)$-bundles $P\ra X$, depending on a finite arbitrary choice. Thus, even without the flag-type structures discussed in {\bf(a)}, we can easily upgrade the orientability in Theorem \ref{ss1thm1} and Corollary \ref{ss1cor1} to particular choices of orientation on all such moduli spaces~$\B_P,\M_P^{\Spin(7)}$.
\label{ss1rem4}	
\end{rem}

It is natural to want to extend Theorem \ref{ss1thm1} and Corollary \ref{ss1cor1} to moduli spaces of connections on principal $G$-bundles $P\ra X$ for Lie groups $G$ other than $\U(m)$ and $\SU(m)$, but this is not always possible.

\begin{ex} Joyce and Upmeier \cite[\S 2.4]{JoUp} give an example of a compact, oriented, spin Riemannian 7-manifold $(Y,h)$ with $Y=\Sp(2)\t_{\Sp(1)\t\Sp(1)}\Sp(1)$, such that taking $F_\bu=D:\Ga^\iy(S_Y)\ra\Ga^\iy(S_Y)$ to be the Dirac operator on $(Y,h)$ and $Q=Y\t\Sp(m)\ra Y$ the trivial $\Sp(m)$-bundle for $m\ge 2$, then $O_Q^{F_\bu}\ra\B_Q$ is not orientable.

Let $\pi:X\ra Y$ be any principal $\U(1)$-bundle and $g$ be a $\U(1)$-invariant Riemannian metric on $X$ with $\pi_*(g)=h$. For instance, we could take $X=Y\t\cS^1$ and $g=h+\d\th^2$. Then $(X,g)$ is a compact, oriented, spin Riemannian 8-manifold with a free $\U(1)$-action. Write $E_\bu=\slashed{D}_+:\Ga^\iy(S_+)\ra\Ga^\iy(S_-)$ for the positive Dirac operator on $(X,g)$. Then $E_\bu$ is $\U(1)$-equivariant, with $S_+\cong S_-\cong \pi^*(S_Y)$, and $F_\bu$ is the pushdown of $E_\bu$ to $Y\cong X/\U(1)$ by restricting $\slashed{D}_+$ to $\U(1)$-invariant sections, as in~\cite[\S 2.2.6]{JTU}.

Let $P=\pi^*(Q)$ be the trivial $\Sp(m)$-bundle $P=X\t\Sp(m)\ra X$. If $\nabla_Q$ is a connection on $Q$ then $\nabla_P=\pi^*(\nabla_Q)$ is a connection on $P$. This defines an injective map $\pi^*:\B_Q\ra\B_P$ of topological stacks. Joyce, Tanaka and Upmeier \cite[Prop.~2.8]{JTU} construct an isomorphism $O_Q^{F_\bu}\cong(\pi^*)^*(O_P^{E_\bu})$ of principal $\Z_2$-bundles on $\B_Q$. Since $O_Q^{F_\bu}\ra\B_Q$ is not orientable \cite[\S 2.4]{JoUp}, this implies that $O_P^{E_\bu}\ra\B_P$ is not orientable. Hence the analogue of Theorem \ref{ss1thm1} for the Lie groups $G=\Sp(m)$, $m\ge 2$ is false.
\label{ss1ex2}
\end{ex}

Here is our second, rather long, main result, which will be proved in \S\ref{ss3}. It provides a bridge between differential geometry of connections on manifolds, and (derived) algebraic geometry of coherent sheaves and complexes on projective $\C$-schemes. Part (a) is essentially immediate from the definitions and references, and is mainly there to establish notation. But we state it separately as it works for general projective $\C$-schemes $X$, not just Calabi--Yau $4m$-folds.

\begin{thm}{\bf(a)} Further details of the following will be given in {\rm\S\ref{ss31}--\S\ref{ss33}}. Let\/ $X$ be a smooth projective $\C$-scheme, and\/ $X^\ran$ be the complex analytic space of\/ $X,$ considered as a compact smooth manifold. Write $\M$ for the moduli stack of objects in $D^b\coh(X)$ as a higher\/ $\C$-stack, as in\/ {\rm\cite{Toen1,ToVa,ToVe1}}. Then $\C$-points of\/ $\M$ correspond to isomorphism classes $[F^\bu]$ of\/~$F^\bu\in D^b\coh(X)$. 

As any complex\/ $F^\bu\in D^b\coh(X)$ has a class\/ $\lb F^\bu\rb\in K^0(X^\ran)$ which is locally constant in $F^\bu,$ there is a decomposition $\M=\coprod_{\al\in K^0(X^\ran)}\M_\al$ with\/ $\M_\al\subset\M$ the open and closed substack of objects $F^\bu$ with\/~$\lb F^\bu\rb=\al$.

There is a universal complex\/ $\cU^\bu\ra X\t\M$ with\/ $\cU^\bu\vert_{X\t[F^\bu]}\cong F^\bu$ for any $\C$-point\/ $[F^\bu]$ in $\M$. This corresponds to a morphism $u:X\t\M\ra\Perf_\C,$
where $\Perf_\C$ is a higher stack which classifies perfect complexes, as in To\"en and Vezzosi\/ {\rm\cite[Def.~1.3.7.5]{ToVe2},} which is just $\M$ for $X=\Spec\C$ the point. In fact\/ $u$ realizes $\M$ as the mapping stack\/~$\M=\Map_{\HSta_\C}(X,\Perf_\C)$.

Apply the topological realization functor $(-)^\top:\Ho(\HSta_\C)\ra\Top^{\bf ho},$ as in Simpson {\rm\cite{Simp1},} Blanc\/ {\rm\cite[\S 3.1]{Blan},} and\/ {\rm\S\ref{ss14}(v)}. As topological realizations only matter up to homotopy equivalence, we may take $X^\top=X^\ran,$ and by\/ {\rm\cite[Th.s 4.5 \& 4.21]{Blan}} we may take\/ $\Perf_\C^\top=B\U\t\Z$ to be the classifying space for complex K-theory, where\/ $B\U=\varinjlim_{n\ra\iy}B\U(n)$. Thus $(-)^\top$ gives a continuous map
\begin{equation*}
u^\top:X^\ran\t\M^\top\longra B\U\t\Z,
\end{equation*}
natural up to homotopy. Write $\Ga:\M^\top\ra\cC:=\Map_{C^0}(X^\ran,B\U\t\Z)$ for the induced map, which is natural up to homotopy.

Since $B\U\t\Z$ is the classifying space for complex K-theory there is a natural isomorphism $\pi_0(\cC)\cong K^0(X^\ran)$. Write $\cC_\al$ for the connected component of\/ $\cC$ corresponding to $\al\in K^0(X^\ran),$ so that\/ $\cC=\coprod_{\al\in K^0(X)}\cC_\al$. Then $\Ga$ maps $\M_\al^\top\ra\cC_\al$ for all\/ $\al\in K^0(X^\ran)$.

Direct sum in $D^b\coh(X)$ induces a natural morphism $\Phi:\M\t\M\ra\M$ acting on $\C$-points by $\Phi:([F^\bu],[G^\bu])\mapsto[F^\bu\op G^\bu]$. Then $\Phi$ is commutative and associative in $\Ho(\HSta_\C),$ with identity $[0]\in\M$. Hence $\Phi^\top:\M^\top\t\M^\top\ra\M^\top$ is commutative and associative up to homotopy, with a homotopy identity. That is, $\Phi^\top$ makes $\M^\top$ into a \begin{bfseries}commutative, associative H-space\end{bfseries}.

There is also a morphism $\Psi:\cC\t\cC\ra\cC$ on $\cC=\Map_{C^0}(X^\ran,B\U\t\Z),$ natural up to homotopy, induced by the usual multiplication on $B\U\t\Z,$ which makes $\cC$ into a commutative, associative H-space. Then there is a homotopy $\Ga\ci\Phi^\top\simeq\Psi\ci(\Ga\t\Ga)$ of maps $\M^\top\t\M^\top\ra\cC,$ and\/ $\Ga$ preserves homotopy identities. Hence $\Ga:\M^\top\ra\cC$ is an H-space morphism.
\smallskip

\noindent{\bf(b)} Now let\/ $(X,\th)$ be an algebraic Calabi--Yau $4m$-fold, and\/ $(X^\ran,J,g,\th)$ a corresponding differential-geometric Calabi--Yau\/ $4m$-fold as in Remark\/ {\rm\ref{ss1rem3}(a),} and use the notation of\/ {\bf(a)} for $X$.

Write $\bs\M$ for the derived moduli stack of objects in $D^b\coh(X),$ as in\/ {\rm\cite{Toen1,Toen2,ToVa,ToVe1,ToVe2},} so that\/ $\M=t_0(\bs\M)$ is its classical truncation.

By Pantev--To\"en--Vaqui\'e--Vezzosi\/ {\rm\cite[Cor.~2.13]{PTVV},} $\bs\M$ has a $(2-4m)$-shifted symplectic structure $\om,$ so Definition\/ {\rm\ref{ss1def6}} defines a notion of orientation on $(\bs\M,\om)$, which form an algebraic principal\/ $\Z_2$-bundle $\pi:O^\om\ra\M$. Write\/ $O^\om_\al=O^\om\vert_{\M_\al}$ for each\/ $\al\in K^0(X^\ran)$. Then $\pi^\top:O^{\om,\top}\ra\M^\top$ is homotopy equivalent to a topological principal\/ $\Z_2$-bundle over $\M^\top,$ so as topological realizations only matter up to homotopy equivalence, we may choose $O^{\om,\top}$ so that\/ $\pi^\top:O^{\om,\top}\ra\M^\top$ is a topological principal\/ $\Z_2$-bundle.

The differential-geometric Calabi--Yau structure induces a spin structure on $(X^\ran,g)$. Write $E_\bu$ for the positive Dirac operator $\slashed{D}_+:\Ga^\iy(S_+)\ra\Ga^\iy(S_-),$ and\/ $O^{E_\bu}\ra\cC$ for the topological principal\/ $\Z_2$-bundle of orientations on\/ $\cC$ defined in\/ {\rm\cite[\S 2.4.2]{JTU}} and\/ {\rm\S\ref{ss32}} below. Then there exists an isomorphism
\e
\ga:O^{\om,\top}\longra\Ga^*(O^{E_\bu})
\label{ss1eq7}
\e
of topological principal\/ $\Z_2$-bundles on $\M^\top$. Thus if\/ $O^{E_\bu}\ra\cC$ is trivializable, then $O^{\om,\top}\ra\M^\top$ is trivializable, and so $O^\om\ra\M$ is trivializable as an algebraic principal\/ $\Z_2$-bundle. 
\smallskip

\noindent{\bf(c)} Continue in the situation of\/ {\bf(b)\rm,} but now suppose $O^{E_\bu}\ra\cC$ is trivializable. Then there is a canonical choice of isomorphism $\ga$ in \eq{ss1eq7}. Thus an orientation $o^{E_\bu}_\al$ for\/ $\cC_\al$ induces by $\ga$ an orientation $o^\om_\al$ for $\M_\al,$ for each\/~$\al\in K^0(X^\ran)$.

In\/ {\rm\S\ref{ss32}--\S\ref{ss33}} we define isomorphisms of principal\/ $\Z_2$-bundles on $\M\t\M$ and\/ $\cC\t\cC,$ which compare orientations under direct sums: 
\e
\phi:O^\om\bt_{\Z_2} O^\om\longra\Phi^*(O^\om),\qquad \psi:O^{E_\bu}\bt_{\Z_2} O^{E_\bu}\longra\Psi^*(O^{E_\bu}).
\label{ss1eq8}
\e
Then in principal\/ $\Z_2$-bundles on $\M^\top\t\M^\top$ we have
\e
\begin{split}
&(\Phi^\top)^*(\ga)\ci\phi^\top\cong (\Ga\t\Ga)^*(\psi)\ci(\ga\bt\ga):\\
&O^{\om,\top}\bt_{\Z_2} O^{\om,\top}\longra
(\Ga\ci\Phi^\top)^*(O^{E_\bu})\cong(\Psi\ci(\Ga\t\Ga))^*(O^{E_\bu}).
\end{split}
\label{ss1eq9}
\e
Here we mean that any homotopy $h:\Ga\ci\Phi^\top{\buildrel\simeq\over\Longra}\,\Psi\ci(\Ga\t\Ga),$ which exists by {\bf(a)\rm,} induces an isomorphism $\io_h:(\Ga\ci\Phi^\top)^*(O^{E_\bu})\ra(\Psi\ci(\Ga\t\Ga))^*(O^{E_\bu})$ of principal\/ $\Z_2$-bundles on $\M^\top\t\M^\top$ by parallel translation along $h,$ and composition with $\io_h$ identifies $(\Phi^\top)^*(\ga)\ci\phi^\top$ and\/ $(\Ga\t\Ga)^*(\psi)\ci(\ga\bt\ga)$. As $O^{E_\bu}$ is trivializable, $\io_h$ is independent of the choice of homotopy~$h$.

Equation \eq{ss1eq9} implies that if we pull back orientations\/ $o^{E_\bu}_\al$ on\/ $\cC_\al$ to orientations $o^\om_\al$ on $\M_\al$ for all\/ $\al\in K^0(X^\ran)$ using $\ga$ as above, then for $\al,\be\in K^0(X^\ran)$ and\/ $\ep_{\al,\be}\in\{\pm 1\}$ we have
\e
\psi(o^{E_\bu}_\al\bt o^{E_\bu}_\be)=\ep_{\al,\be}\cdot o^{E_\bu}_{\al+\be}\quad\Longra\quad
\phi(o^\om_\al\bt o^\om_\be)=\ep_{\al,\be}\cdot o^\om_{\al+\be}.
\label{ss1eq10}
\e
That is, relations between orientations on $\cC_\al,\cC_\be,\cC_{\al+\be}$ under direct sum imply the analogous relations between orientations on\/ $\M_\al,\M_\be,\M_{\al+\be}$.
\label{ss1thm2}
\end{thm}

\begin{rem}
\label{ss1rem5}
The following issue, discussed in \cite[\S 2.3.5]{JTU}, is why we separate parts (b),(c) in the theorem above. 

Suppose $Y,Z$ are topological spaces, $f_0,f_1:Y\ra Z$ are homotopic maps linked by a homotopy $(f_t)_{t\in[0,1]}$, and $P\ra Z$ is a principal $\Z_2$-bundle. Then $f_0^*(P),f_1^*(P)$ are principal $\Z_2$-bundles on $Y$ linked by a 1-parameter family $f_t^*(P)_{t\in[0,1]}$. Parallel translation along this family induces an isomorphism $f_0^*(P)\cong f_1^*(P)$ of principal $\Z_2$-bundles on $Y$.

If $P$ is trivializable this isomorphism $f_0^*(P)\cong f_1^*(P)$ is independent of the choice of homotopy $(f_t)_{t\in[0,1]}$, but if $P$ is nontrivial it may depend on the homotopy \cite[Ex.~2.18]{JTU}. Thus, as $\Ga:\M^\top\ra\cC$ is only natural up to homotopy, it is only reasonable for $\ga$ in \eq{ss1eq7} to be canonical if $O^{E_\bu}$ is trivializable.  

In fact, as in \cite[Rem.~2.17(b)]{JTU}, we could get round this issue by including extra structure. Rather than regarding $\M^\top,\cC$ as {\it H-spaces}, as in \S\ref{ss31}, we could make them into $\Ga$-{\it spaces}, as in Segal \cite[\S 1]{Sega}, or $E_\iy$-{\it spaces}, as in May \cite{May2}, and $\Ga:\M^\top\ra\cC$ into a morphism of $\Ga$- or $E_\iy$-spaces. Then whenever two maps $f_0,f_1$ in our theory are homotopic, the $\Ga$- or $E_\iy$-space structure would provide a homotopy $(f_t)_{t\in[0,1]}$ natural up to homotopies of homotopies, so the corresponding isomorphism of principal $\Z_2$-bundles would be canonical.
\end{rem}

Combining Theorems \ref{ss1thm1} and \ref{ss1thm2} yields:

\color{red}\begin{cor}\footnote{\color{red} Corollary \ref{ss1cor2} depends on Theorem \ref{ss1thm1}, which is false. So Corollary \ref{ss1cor2} may also be false, although we do not have a counterexample. See the \hyperref[erratum]{Erratum} for a corrected version of Corollary \ref{ss1cor2}.\color{black}} Let\/ $(X,\th)$ be an algebraic Calabi--Yau\/ $4$-fold. Then the orientation bundle $O^\om\ra\M$ from Theorem\/ {\rm\ref{ss1thm2}(b)} and Borisov and Joyce {\rm\cite[\S 2.4]{BoJo}} is a trivializable algebraic principal\/ $\Z_2$-bundle, i.e.\ $\M$ is orientable.

If we choose orientations $o_\al^{E_\bu}$ on\/ $\cC_\al$ for all\/ $\al\in K^0(X^\ran),$ which may satisfy compatibility conditions under direct sums as in\/ {\rm\cite[\S 2.5]{JTU},} then we obtain corresponding orientations\/ $o_\al^\om$ on\/ $\M_\al$ for all\/ $\al,$ which satisfy the analogous compatibility conditions under direct sums by\/~\eq{ss1eq10}. 

A systematic way of choosing such orientations $o_\al^{E_\bu}$ for all\/ $\al\in K^0(X^\ran),$ depending only on a finite arbitrary choice, is explained in\/~{\rm\cite[Th.~2.27]{JTU}}.
\label{ss1cor2}
\end{cor}\color{black}

\begin{rem} The higher $\C$-stack $\M$ in Theorem \ref{ss1thm2} and Corollary \ref{ss1cor2} contains as open Artin $\C$-substacks the moduli stacks $\M^{\rm coh},\M^{\rm coh,ss},\M^{\rm vect}$ of coherent sheaves, and semistable coherent sheaves, and algebraic vector bundles on $X$, respectively. The principal $\Z_2$-bundle $O^\om\ra\M$, and orientations on $\M$, may be restricted to $\M^{\rm coh},\ldots,\M^{\rm vect}$. Thus, Theorem \ref{ss1thm2} and Corollary \ref{ss1cor2} are still interesting if we only care about $\M^{\rm coh},\ldots,\M^{\rm vect}$ rather than~$\M$.
\label{ss1rem6}
\end{rem}

Corollary \ref{ss1cor2} has important applications in the programme of defining and studying `DT4 invariants' of Calabi--Yau 4-folds proposed by Borisov and Joyce \cite{BoJo} and Cao and Leung \cite{CaLe1}, which we now discuss. We summarize the main results of Borisov and Joyce~\cite{BoJo}:

\begin{thm}[Borisov and Joyce {\cite{BoJo}}] Let\/ $(\bS,\om)$ be a $-2$-shifted symplectic derived\/ $\C$-scheme, with complex virtual dimension\/ $\vdim_\C\bS=n$ in $\Z,$ and write\/ $S_{\rm an}$ for the set of\/ $\C$-points of\/ $S=t_0(\bS),$ with the complex analytic topology. Suppose that\/ $S$ is separated, and\/ $S_{\rm an}$ is second countable. Then after making some arbitrary choices, we can make the topological space $S_{\rm an}$ into a \begin{bfseries}derived smooth manifold\end{bfseries}\/ $\bS_{\rm dm}$ (d-manifold, or m-Kuranishi space) in the sense of Joyce {\rm\cite{Joyc2,Joyc4,Joyc5,Joyc6},} of real virtual dimension\/~$\vdim_\R\bS_{\rm dm}=n=\ha\vdim_\R\bS$.

There is a natural\/ {\rm 1-1} correspondence between orientations on $(\bS,\om)$ in the sense of Remark\/ {\rm\ref{ss1def6},} and orientations on the derived manifold\/~$\bS_{\rm dm}$.

If\/ $\bS$ is proper and\/ $(\bS,\om)$ is oriented then $\bS_{\rm dm}$ is a compact, oriented derived manifold, and so has a \begin{bfseries}virtual class\end{bfseries} $[\bS_{\rm dm}]_\virt$ in $H_n(S_{\rm an},\Z)$. These virtual classes are deformation-invariant under variations of\/ $(\bS,\om)$ in families.
\label{ss1thm3}
\end{thm}

Theorem \ref{ss1thm3} is far from obvious. The two geometric structures are apparently unrelated, and the virtual dimension of $\bS_{\rm dm}$ is half that of $\bS$. A heuristic explanation is that there is a real `Lagrangian fibration' $\pi:\bS\ra\bS_{\rm dm}$, whose `Lagrangian' fibres are points with derived structure in degrees $-1,-2$, and $\vdim_\R\bS_{\rm dm}=\ha\vdim_\R\bS$ holds as the dimension of the base of a Lagrangian fibration is half the dimension of the symplectic manifold.

The usual notion of virtual class in algebraic geometry is defined by Behrend and Fantechi \cite{BeFa}. The virtual classes in Theorem \ref{ss1thm3} are new, and very different. The construction of \cite{BeFa} does not apply in this case, as $\bL_\bS\vert_S$ is perfect in $[-2,0]$ not $[-1,0]$. The virtual class $[\bS_{\rm dm}]_\virt$ has half the dimension one would expect from \cite{BeFa}, and may even have odd real dimension. 

Borisov and Joyce propose to use the virtual classes in Theorem \ref{ss1thm3} for derived moduli schemes of semistable coherent sheaves $(\bs\M_\al^{\rm ss},\om)$ on an algebraic Calabi--Yau 4-fold $(X,\th)$ to define Donaldson--Thomas \cite{DoTh} style `DT4 invariants' of Calabi--Yau 4-folds. Cao and Leung \cite{CaLe1} make a similar proposal using gauge theory. An essential ingredient in this programme is orientations on moduli spaces $(\bs\M_\al^{\rm ss},\om)$. Corollary \ref{ss1cor2} implies that such orientations always exist, and can be chosen in a systematic way.

\section{Proof of Theorem \ref{ss1thm1}}
\label{ss2}

Let $X$ be a compact, oriented, spin Riemannian 8-manifold, $E_\bu$ be the positive Dirac operator on $X$, and $P\ra X$ be a principal $G$-bundle for $G=\U(m)$ or $\SU(m)$. We must prove the orientation bundle $O_P^{E_\bu}\ra\B_P$ in Definition \ref{ss1def2} is trivial. As in Definition \ref{ss1def2}, this is equivalent to the normalized orientation bundle $\check O_P^{E_\bu}\ra\B_P$ being trivial. We will do this in the following steps:
\smallskip

\noindent{\bf Step 1.} Use results of Joyce, Tanaka and Upmeier \cite{JTU} to show that $\B_Q$ is orientable for any principal $\U(m)$- or $\SU(m)$-bundle $Q\ra X$, and also $\cC_\al$ is orientable for all $\al\in K^0(X)$, if and only if $\B_P$ is orientable when $P=X\t\SU(4)$ is the trivial $\SU(4)$-bundle over $X$.
\smallskip

\noindent{\bf Step 2.} Let $P=X\t\SU(4)\ra X$ be the trivial $\SU(4)$-bundle and $\nabla^0$ the trivial connection on $P$, so that $[\nabla^0]$ is a base-point in $\B_P$. The fundamental group $\pi_1(\B_P)$ is the set of homotopy classes $[\ga]$ of loops $\ga:\cS^1\ra\B_P$ with $\ga(1)=[\nabla^0]$. As in \cite[\S 2]{JTU} there is a group morphism $\Th:\pi_1(\B_P)\ra\Z_2=\{\pm 1\}$ such that $\Th([\ga])$ is the monodromy of the principal $\Z_2$-bundle $\check O_P^{E_\bu}\ra\B_P$ around $\ga$, and $\B_P$ is orientable if and only if~$\Th\equiv 1$.

We establish (already known) natural 1-1 correspondences between:
\begin{itemize}
\setlength{\itemsep}{0pt}
\setlength{\parsep}{0pt}
\item[(a)] Elements $[\ga]\in \pi_1(\B_P)$.
\item[(b)] Isomorphism classes $[Q,q]$ of pairs $(Q,q)$, where $Q\ra X\t\cS^1$ is a principal $\SU(4)$-bundle and $q:Q\vert_{X\t\{1\}}\,{\buildrel\cong\over\longra}\,(X\t\{1\})\t\SU(4)=P$ is a trivialization of $Q$ over $X\t\{1\}$.
\item[(c)] Homotopy classes $[\Phi]$ of smooth maps $\Phi:X\ra\SU(4)$.
\end{itemize}

Write $[X,\SU(4)]$ for the set of homotopy classes $[\Phi]$ of smooth maps $\Phi:X\ra\SU(4)$. Then the 1-1 correspondence gives a bijection $\pi_1(\B_P)\cong [X,\SU(4)]$. This is an isomorphism of groups, where $[X,\SU(4)]$ has group operation $[\Phi]+[\Phi']=[\mu(\Phi,\Phi')]$, for $\mu:\SU(4)\t\SU(4)\ra\SU(4)$ the multiplication map. In fact $\pi_1(\B_P)$, $[X,\SU(4)]$ are abelian, as $\SU(4)$ is in the stable range for 8-manifolds. Let $\hat\Th:[X,\SU(4)]\ra \Z_2$ be identified with $\Th$ under $\pi_1(\B_P)\cong [X,\SU(4)]$. We must prove that~$\hat\Th\equiv 1$.
\smallskip

\noindent{\bf Step 3.} Define subsets $Y_k\subset\SU(4)$ for $k=0,\ldots,3$ by
\e
Y_k=\bigl\{A\in\SU(4):\dim_\C\{\bs x=(x_1,x_2,x_3,0)^T\in\C^4: A\bs x=-\bs x\}=k\bigr\},
\label{ss2eq1}
\e
so that $\SU(4)=Y_0\amalg\cdots\amalg Y_3$. We prove that:
\begin{itemize}
\setlength{\itemsep}{0pt}
\setlength{\parsep}{0pt}
\item[(i)] $Y_k$ is a connected, simply-connected, oriented, embedded submanifold of $\SU(4)$ (which is also oriented) of real codimension $k(k+2)$. Hence $Y_0$ is open in $\SU(4)$, and $Y_1,Y_2,Y_3$ have codimensions~$3,8,15$.
\item[(ii)] The closure of $Y_k$ in $\SU(4)$ is $\ovY_k=Y_k\amalg Y_{k+1}\amalg\cdots\amalg Y_3$.
\item[(iii)] There is a smooth family of smooth maps $\Psi_t:Y_0\ra\SU(4)$ for $t\in[0,1]$ with $\Psi_0$ the inclusion $Y_0\hookra \SU(4)$, and $\Psi_1\equiv\Id$ the constant map with value $\Id\in\SU(4)$. That is, $Y_0$ retracts to $\{\Id\}$ in~$\SU(4)$.
\item[(iv)] We may define a smooth map $\phi:Y_1\ra\CP^2$ by $\phi(A)=[x_1,x_2,x_3]$ if $A\bs x=-\bs x$ for $\bs x=(x_1,x_2,x_3,0)^T$. The normal bundle $\nu$ of $Y_1$ in $\SU(4)$ is isomorphic to $\R\op\phi^*(\O(1))$, for $\O(1)\ra\CP^2$ the standard line bundle.
\end{itemize}

It is known that the cohomology of $\SU(4)$ may be written as a graded ring
\e
H^*(\SU(4),\Z)\cong \La_\Z[p_3,p_5,p_7],
\label{ss2eq2}
\e
where $p_3,p_5,p_7$ are odd generators in degrees $3,5,7$, which satisfy
\e
\begin{split}
\mu^*(p_k)=p_k\bt 1+1\bt p_k \quad\text{in}\quad &H^*(\SU(4)\t\SU(4),\Z)\\
&\cong H^*(\SU(4),\Z)\ot H^*(\SU(4),\Z).
\end{split}
\label{ss2eq3}
\e

Under Poincar\'e duality $\Pd:H^k(\SU(4),\Z)\,{\buildrel\cong\over\longra}\, H_{15-k}(\SU(4),\Z)$ we have
\e
\Pd(p_3)=[\ovY_1], \qquad \Pd(p_3\cup p_5)=[\ovY_2].
\label{ss2eq4}
\e

\noindent{\bf Step 4.} Using the notation of Steps 2--3, define maps $\la_k:[X,\SU(4)]\ra H^k(X,\Z)$ for $k=3,5,7$ by $\la_k([\Phi])=\Phi^*(p_k)$. Equation \eq{ss2eq3} and $[\Phi]+[\Phi']=[\mu(\Phi,\Phi')]$ imply that $\la_3,\la_5,\la_7$ are group morphisms. We can also define a map $\ka:[X,\SU(4)]\ra\Z$ by
\e
\ka:[\Phi]\longmapsto \bigl(\la_3([\Phi])\cup\la_5([\Phi])\bigr)\cdot[X].
\label{ss2eq5}
\e
Note that this is {\it not\/} a group morphism, but is quadratic in $\Phi$.

We prove that for any $\al\in H^5(X,\Z)$ we can construct $[\Phi']\in [X,\SU(4)]$ with $\la_3([\Phi'])=0$ and $\la_5([\Phi'])=\al$.

Therefore any $[\Phi]\in [X,\SU(4)]$ may be written $[\Phi]=[\Phi']+[\Phi'']$ with $\la_3([\Phi'])=0$ and $\la_5([\Phi''])=0$, since we can take $[\Phi']$ as above with $\la_3([\Phi'])=0$ and $\la_5([\Phi'])=\al=\la_5([\Phi])$, and $[\Phi'']=[\Phi]-[\Phi']$, and use the fact that $\la_3,\la_5$ are group morphisms. Note that $\ka([\Phi'])=\ka([\Phi''])=0$. Since $\hat\Th([\Phi])=\hat\Th([\Phi'])\cdot\hat\Th([\Phi''])$, we see from Step 2 that it is sufficient to prove that $\hat\Th([\Phi])=1$ for all $[\Phi]\in[X,\SU(4)]$ with~$\ka([\Phi])=0$.
\smallskip

\noindent{\bf Step 5.} Suppose $X$ is connected, and $[\Phi]\in[X,\SU(4)]$ with $\ka([\Phi])=0$. Choose a generic representative $\Phi:X\ra\SU(4)$ for $[\Phi]$. As $\Phi$ is a generic smooth map from an 8-manifold to a 15-manifold, it need not be an embedding, but it fails to be an embedding only on a 1-dimensional subset of $X$, and we can assume by genericness that the image of this subset avoids the codimension 3 subset $Y_1\amalg Y_2\amalg Y_3$ of $\SU(4)$. Hence $\Phi$ is an embedding near $Y_1\amalg Y_2\amalg Y_3$, and $\Phi(X)$ intersects $Y_k$ transversely in $\SU(4)$ for $k=1,2,3$ by genericness. Thus $\Phi(X)\cap Y_1$ is an oriented embedded 5-manifold, and $\Phi(X)\cap Y_2$ an oriented embedded 0-manifold, and $\Phi(X)\cap Y_3=\es$, so $\Phi(X)\cap Y_2$ is compact as~$\ovY_2=Y_2\amalg Y_3$.

From \eq{ss2eq4}--\eq{ss2eq5} we see that the number of points in $\Phi(X)\cap Y_2$, counted with signs, is $\ka([\Phi])=0$. Using this we show that we can perturb $\Phi$ in its homotopy class to make $\Phi(X)\cap Y_2=\es$. Then $\Phi(X)\cap Y_1$ is compact, as~$\ovY_1=Y_1\amalg Y_2\amalg Y_3$.

Define $Z=\{x\in X:\Phi(x)\in Y_1\}$. Then $Z$ is a compact, oriented, embedded 5-submanifold in $X$ diffeomorphic to $\Phi(X)\cap Y_1$. Define $\psi:Z\ra\CP^2$ by $\psi=\phi\ci\Phi\vert_Z$, for $\phi$ as in Step 3(iv). The normal bundle $\nu_Z$ of $Z$ in $X$ satisfies
\e
\nu_Z\cong\Phi\vert_Z^*(\nu)\cong \R\op\psi^*(\O(1)).
\label{ss2eq6}
\e
As $TX\vert_Z=TZ\op\nu_Z$, and $X$ is spin so that $w_2(TX)=0$, we see that the second Stiefel--Whitney class $w_2(Z)\in H^2(Z,\Z_2)$ satisfies
\begin{equation*}
w_2(Z)=w_2(TZ)=w_2(\nu_Z)=w_2(\R\op\psi^*(\O(1)))=\psi^*(c_1(\O(1)))\mod 2.
\end{equation*}
That is, $w_2(Z)$ is the image in $H^2(Z,\Z_2)$ of the integral class $\psi^*(c_1(\O(1)))$ in $H^2(Z,\Z)$. This implies that $Z$ admits a $\Spinc$-structure, and simplifies the classification of possible 5-manifolds $Z$ up to diffeomorphism.

If $X$ is connected, or simply-connected, we show that we can perturb $\Phi$ in its homotopy class to make $Z$ connected, or simply-connected, respectively.

The importance of $Z$ is that $\Phi$ maps $X\sm Z\ra Y_0\subset\SU(4)$, where $Y_0$ retracts to $\{\Id\}\subset\SU(4)$ by Step 3(iii). Hence $\Phi\vert_{X\sm Z}$ is homotopic to the constant map $\Id$, and if $[Q,q]$ corresponds to $\Phi$ as in Step 2, then $(Q,q)$ is trivial over $(X\sm Z)\t\cS^1$. This allows us to use excision techniques in Steps 6 and~7.
\smallskip
 
\noindent{\bf Step 6.} Suppose $X$ is connected and simply-connected, and $[\ga]\in\pi_1(\B_P)$ corresponds to $[\Phi]\in [X,\SU(4)]$ as in Step 2 with $\ka([\Phi])=0$ as in Step 4, and define $\Phi,Z,\psi,\nu_Z$ with $Z$ connected and simply-connected as in Step~5.

Using the classification of compact, simply-connected 5-manifolds in Crowley \cite{Crow}, we show we can choose a tubular neighbourhood $U$ of $Z$ in $X$, an explicit compact, oriented, spin Riemannian 8-manifold $X'$ with $H^{\rm odd}(X',\Z)=0$, and an embedding $\io:U\hookra X'$ of $U$ as an open submanifold of $X'$, where $\io$ preserves orientations and spin structures.

Using the Excision Theorem, Theorem \ref{ss2thm1}, we show that the monodromy $\Th([\ga])$ of $\check O_P^{E_\bu}\ra\B_P$ around $\ga$ equals the monodromy of $\check O_{P'}^{E'_\bu}\ra\B_{P'}$ around some loop $\ga'$ in $\B_{P'}$, where $P'=X'\t\SU(4)$. Since $H^{\rm odd}(X',\Z)=0$, $\B_{P'}$ is orientable by \cite[Cor.~2.25]{JTU}, so $\Th([\ga])=\hat\Th([\Phi])=1$. As in Step 4 it is sufficient to prove this for $[\Phi]$ with $\ka([\Phi])=0$, so $\B_P$ is orientable. This proves Theorem \ref{ss1thm1} in the case $X$ is simply-connected.
\smallskip

\noindent{\bf Step 7.} For $X$ not simply-connected, by doing surgeries on finitely many disjoint embedded circles $L_1,\ldots,L_k$ in $X$ we can modify $X$ to a simply-connected, compact, oriented, spin Riemannian 8-manifold $X'$, with open covers $X=U\cup V$, $X'=U'\cup V'$ for $V$ a small tubular neighbourhood of $L_1\amalg\cdots\amalg L_k$ in $X$, and a diffeomorphism $\io:U\ra U'$ preserving orientations and spin structures.

Let $[\ga]\in \pi_1(\B_P)$ correspond to $(Q,q)$ as in Step 2. Then $Q\ra X\t\cS^1$ is trivial over $(L_1\amalg\cdots\amalg L_k)\t\cS^1$, as any $\SU(4)$-bundle over a 2-manifold is trivial, so $Q$ is trivial over $V\t\cS^1$ as $V$ retracts onto $L_1\amalg\cdots\amalg L_k$, and we can choose this trivialization compatible with $q$ on $V\t\{1\}$.

Using the Excision Theorem as in Step 6, we find that $\Th([\ga])=\Th'([\ga'])$ for some loop $\ga'$ in $\B_{P'}$, where $P'=X'\t\SU(4)$. But $\Th'([\ga'])=1$ by Step 6, as $X'$ is simply-connected, so $\Th([\ga])=1$. Thus $\B_P$ is orientable, completing the proof of Theorem~\ref{ss1thm1}.
\smallskip

We will give more details on Steps 1--6 in \S\ref{ss21}--\S\ref{ss26}. Step 7 is very similar to Step 6, and we leave it as an exercise for the reader.

\subsection{\texorpdfstring{Step 1: Reduction to the case $P=X\t\SU(4)$}{Step 1: Reduction to the case  P= X x SU(4)}}
\label{ss21}

We first recall the material in \cite{JTU,Upme} we need in the rest of the proof. Let $X$ be a compact, connected $n$-manifold and $E_\bu$ an elliptic complex on $X$, and use the notation of Definitions \ref{ss1def1}--\ref{ss1def2}. Joyce, Tanaka and Upmeier \cite[\S 2]{JTU} explain:
\begin{itemize}
\setlength{\itemsep}{0pt}
\setlength{\parsep}{0pt}
\item[(i)] If $P\ra X$ is a principal $\SU(m)$-bundle, define $Q=(P\t\U(m))/\SU(m)$, so that $Q\ra X$ is a principal $\U(m)$-bundle, and let $R=X\t\U(1)\ra X$ be the trivial $\U(1)$-bundle. Then as in \cite[Ex.~2.9]{JTU} there is a natural isomorphism $\B_Q\cong\B_P\t\B_R$ and a corresponding isomorphism $\check O_Q^{E_\bu}\cong\check O_P^{E_\bu}\bt_{\Z_2}\check O_R^{E_\bu}$. But $\check O_R^{E_\bu}$ is trivial by \cite[\S 2.2.3]{JTU} as $\U(1)$ is abelian. Hence $\check O_P^{E_\bu}$ is trivializable if and only if $\check O_Q^{E_\bu}$ is, so $\B_P$ is orientable if and only if $\B_Q$ is orientable. 
\item[(ii)] Write $\cC=\Map_{C^0}(X,B\U\t\Z)$, where $B\U\t\Z$ is the classifying space for complex K-theory with $B\U=\varinjlim_{n\ra\iy}B\U(n)$. There is a natural isomorphism $K^0(X)\cong \pi_0(\cC)$. Write $\cC_\al$ for the connected component of $\cC$ corresponding to $\al\in K^0(X)$. In \cite[\S 2.4.2]{JTU} and \S\ref{ss32} below we construct principal $\Z_2$-bundles $\check O^{E_\bu}\ra\cC$ and $\check O^{E_\bu}_\al\ra\cC_\al$ with $\check O^{E_\bu}_\al=\check O^{E_\bu}\vert_{\cC_\al}$. We say that $\cC_\al$ is {\it orientable\/} if $\check O^{E_\bu}_\al$ is trivializable. By \cite[Prop.~2.24(b)]{JTU}, $\cC_\al$ is orientable if and only if $\cC_0$ is orientable for any $\al$ in $K^0(X)$. 
\item[(iii)] Let $P\ra X$ be a principal $\U(m)$-bundle. It has a K-theory class $\lb P\rb$ in $K^0(X)$, the class of the complex vector bundle $(P\t\C^m)/\U(m)$. Set $\al=\lb P\rb$. In \cite[\S 2.4.2]{JTU} we relate the principal $\Z_2$-bundles $\check O_P^{E_\bu}\ra\B_P$ and $\check O_\al^{E_\bu}\ra\cC_\al$ as follows: let $\pi^\cla:\B_P^\cla\ra\B_P$ be a `classifying space' for $\B_P$ in the sense of Noohi \cite{Nooh2}, that is, $\B_P^\cla$ is a paracompact topological space and $\pi^\cla:\B_P^\cla\ra\B_P$ is a morphism in $\Ho(\TopSta)$ which in a weak sense is a fibration with contractible fibres. Then there is a continuous map $\Si_P^\cC:\B_P^\cla\ra\cC_\al$, natural up to homotopy, and a natural isomorphism of principal $\Z_2$-bundles $\si_P^\cC:(\pi^\cla)^*(\check O_P^{E_\bu})\ra(\Si_P^\cC)^*(\check O_\al^{E_\bu})$. Hence
\begin{align*}
&\text{$\cC_\al$ is orientable}\;\Longleftrightarrow\; \text{$\check O_\al^{E_\bu}$ is trivial} \; \Longra\; \text{$(\Si_P^\cC)^*(\check O_\al^{E_\bu})$ is trivial}\; \smash{\buildrel\raisebox{3pt}{$\st\si_P^\cC$}\over\Longleftrightarrow} \\
&\text{$(\pi^\cla)^*(\check O_P^{E_\bu})$ is trivial}\;{\buildrel\raisebox{3pt}{$\st\text{$\pi^\cla$ is $\simeq$}$}\over\Longleftrightarrow}\;\text{$\check O_P^{E_\bu}$ is trivial}\;\Longleftrightarrow\;\text{$\B_P$ is orientable,}
\end{align*}
so $\cC_\al$ orientable implies $\B_P$ is orientable. Also $\pi_1(\Si_P^\cC):\pi_1(\B_P^\cla)\ra\pi_1(\cC_\al)$ is surjective if $2m\ge n$, and an isomorphism if $2m>n$. Thus, if $2m\ge n$ then $\check O_\al^{E_\bu}$ is trivial if and only if $(\Si_P^\cC)^*(\check O_\al^{E_\bu})$ is trivial, so $\cC_\al$ is orientable if and only if $\B_P$ is orientable.
\item[(iv)] Let $P=X\t\SU(m)\ra X$ be the trivial $\SU(m)$-bundle, for any $m$ with $2m\ge n$, and suppose $\B_P$ is orientable. Then (i) implies that $\B_Q$ is orientable for $Q=X\t\U(m)\ra X$ the trivial $\U(m)$-bundle, so (iii) shows $\cC_\be$ is orientable for $\be=\lb Q\rb\in K^0(X)$, and (ii) proves $\cC_\al$ is orientable for any $\al\in K^0(X)$, and (iii),(i) imply that $\B_R$ is orientable for any principal $\U(m')$- or $\SU(m')$-bundle $R\ra X$ and any $m'\ge 0$.
\item[(v)] By \cite[Prop.~2.24(c)]{JTU} we have $\pi_1(\cC_0)\cong K^1(X)$, the odd complex K-theory group of $X$. Hence if $K^1(X)=0$ then any principal $\Z_2$-bundle over $\cC_0$ is trivial, and $\cC_0$ is orientable, so $\B_P$ is orientable for any principal $\U(m)$- or $\SU(m)$-bundle $P\ra X$ as in~(iv).

There is an Atiyah--Hirzebruch spectral sequence $H^{\rm odd}(X,\Z)\Ra K^1(X)$. Thus if $H^{\rm odd}(X,\Z)=0$ then~$K^1(X)=0$.
\end{itemize}
Step 1 now follows from (iv) with $n=8$ and $m=4$.

The next theorem is proved in Upmeier \cite[Th.~2.13]{Upme}, based on Donaldson \cite[\S II.4]{Dona1}, \cite[\S 3(b)]{Dona2}, and will be used in Steps 6 and~7. 

\begin{thm}[Excision Theorem]
\label{ss2thm1}
Suppose we are given the following data:
\begin{itemize}
\setlength{\itemsep}{0pt}
\setlength{\parsep}{0pt}
\item[{\bf(a)}] Compact\/ $n$-manifolds $X^+,X^-$.
\item[{\bf(b)}] Elliptic complexes $E_\bu^\pm$ on $X^\pm$.
\item[{\bf(c)}] A Lie group $G,$ and principal\/ $G$-bundles $P^\pm\ra X^\pm$ with connections\/~$\nabla_{P^\pm}$.
\item[{\bf(d)}] Open covers $X^+=U^+\cup V^+,$ $X^-=U^-\cup V^-$.
\item[{\bf(e)}] A diffeomorphism $\io:U^+\ra U^-,$ such that $E_\bu^+\vert_{U^+}$ and\/ $\io^*(E_\bu^-\vert_{U^-})$ are isomorphic elliptic complexes on $U^+$.
\item[{\bf(f)}] An isomorphism $\si:P^+\vert_{U^+}\ra \io^*(P^-\vert_{U^-})$ of principal\/ $G$-bundles over $U^+,$ which identifies $\nabla_{P^+}\vert_{U^+}$ with\/~$\io^*(\nabla_{P^-}\vert_{U^-})$.
\item[{\bf(g)}] Trivializations of principal\/ $G$-bundles $\tau^\pm:P^\pm\vert_{V^\pm}\ra V^\pm\t G$ over $V^\pm,$ which identify $\nabla_{P^\pm}\vert_{V^\pm}$ with the trivial connections, and satisfy 
\begin{equation*}
\smash{\io\vert_{U^+\cap V^+}^*(\tau^-)\ci\si\vert_{U^+\cap V^+}=\tau^+\vert_{U^+\cap V^+}}.
\end{equation*}
\end{itemize}
Then we have a canonical identification of\/ $\Z_2$-torsors
\e
\Om^{+-}:\check O_{P^+}^{E_\bu^+}\big\vert_{[\nabla_{P^+}]}\,{\buildrel\cong\over\longra}\,\check O_{P^-}^{E_\bu^-}\big\vert_{[\nabla_{P^-}]}.
\label{ss2eq7}
\e

The isomorphisms \eq{ss2eq7} are functorial in a very strong sense. For example:
\begin{itemize}
\setlength{\itemsep}{0pt}
\setlength{\parsep}{0pt}
\item[{\bf(i)}] If we vary any of the data in {\bf(a)}--{\bf(g)} continuously in a family over $t\in[0,1],$ then the isomorphisms $\Om^{+-}$ also vary continuously in $t\in[0,1]$. 
\item[{\bf(ii)}] The isomorphisms $\Om^{+-}$ are unchanged by shrinking the open sets $U^\pm,V^\pm$ such that\/ $X^\pm=U^\pm\cup V^\pm$ still hold, and restricting $\io,\si,\tau^\pm$.
\item[{\bf(iii)}] If we are also given a compact\/ $n$-manifold\/ $X^\t,$ elliptic complex $E_\bu^\t,$ bundle $P^\t\ra X^\t,$ connection $\nabla_{P^{\smash{\t}}},$ open cover $X^\t=U^\t\cup V^\t,$ diffeomorphism $\io':U^-\ra U^\t,$ and isomorphisms  $\si':P^-\vert_{U^-}\ra \io^{\prime*}(P^\t\vert_{U^\t}),$ $\tau^\t:P^\t\vert_{V^\t}\ra V^\t\t G$ satisfying the analogues of\/ {\bf(a)\rm--\bf(g)\rm,} then\/ $\Om^{+\t}=\Om^{-\t}\ci\Om^{+-},$ where $\Om^{+\t}$ is defined using $\io'\ci\io:U^+\ra U^\t$ and\/~$\io^*(\si')\ci\si:P^+\vert_{U^+}\ra (\io'\ci\io)^*(P^\t\vert_{U^\t})$.
\end{itemize}
\end{thm}

\subsection{\texorpdfstring{Step 2: Alternative descriptions of $\pi_1(\B_P)$}{Step 2: Alternative descriptions of π₁(ℬᵨ)}}
\label{ss22}

We will justify the 1-1 correspondences between (a),(b),(c) in Step 2. Let $P=X\t\SU(4)\ra X$ be the trivial $\SU(4)$-bundle and $\nabla^0$ the trivial connection on $P$, so that $[\nabla^0]\in\B_P$. As in (a), let $\ga:\cS^1\ra\B_P$ be a smooth path with $\ga(1)=[\nabla^0]$. Then $\ga$ is a smooth path of connections on $P$ modulo gauge. We can think of $\ga$ as a smooth family of pairs $(P_z,\nabla_{P_z})_{z\in\cS^1}$, where $P_z\ra X$ is a principal $\SU(4)$-bundle which is isomorphic to $P$, but not canonically isomorphic to $P$ (since we quotient by the gauge group $\G_P=\Aut(P)$), and $\nabla_{P_z}$ is a connection on $P_z$, with $P_1=P$ and~$\nabla_{P_1}=\nabla^0$.

We can assemble the $(P_z)_{z\in\cS^1}$ into a principal $\SU(4)$-bundle $Q\ra X\t\cS^1$ with $Q\vert_{X\t\{z\}}=P_z$, and then $Q\vert_{X\t\{1\}}=P_1=P$ gives a trivialization $q:Q\vert_{X\t\{1\}}\,{\buildrel\cong\over\longra}\,(X\t\{1\})\t\SU(4)$ as required. The connections $(\nabla_{P_z})_{z\in\cS^1}$ assemble into a partial connection $\nabla_Q^X$ on $Q$ in the $X$ directions in $X\t\cS^1$. Note that although each $P_z$ is (noncanonically) trivial, $Q$ need not be a trivial bundle on $X\t\cS^1$, as it can have nontrivial topological twisting in the $\cS^1$ directions. Changing the loop $\ga$ by smooth homotopies deforms $Q,q,\nabla_Q^X$ smoothly, and so preserves the pair $(Q,q)$ up to isomorphism. This gives a well-defined map $[\ga]\mapsto[Q,q]$ from objects (a) to objects~(b). 

Conversely, given $[Q,q]$ choose a representative $(Q,q)$ and a partial connection $\nabla_Q^X$ on $Q$ in the $X$ directions in $X\t\cS^1$ with $\nabla_Q^X\vert_{X\t\{1\}}=\nabla^0$, and define $\ga:\cS^1\ra\B_P$ by $\ga(z)=[\nabla_Q^X\vert_{X\t\{z\}}]$. This is well defined as $Q\vert_{X\t\{z\}}$ is noncanonically isomorphic to $P$, since $Q\vert_{X\t\{1\}}\cong P$. Then $\ga$ is a smooth loop in $\B_P$ with $\ga(1)=[\nabla^0]$, so $[\ga]\in\pi_1(\B_P)$. The space of partial connections $\nabla_Q^X$ on $Q$ is an infinite-dimensional affine space, so any two choices $\nabla_Q^X,\ti\nabla_Q^X$ are joined by a smooth path, and the corresponding loops $\ga,\ti\ga$ are smoothly homotopic, giving $[\ga]=[\ti\ga]$. Hence the inverse map $[Q,q]\mapsto[\ga]$ is defined, and (a),(b) are in 1-1 correspondence.

Now let $(Q,q)$ be as in (b). Choose a connection $\nabla_Q$ on $Q\ra X\t\cS^1$. For each $x\in X$, consider the path $\de_x:[0,2\pi]\ra X\t\cS^1$ mapping $\de_x:\th\mapsto(x,e^{i\th})$. The holonomy of $\nabla_Q$ around $\de_x$ is a smooth map $\Hol_{\de_x}(\nabla_Q):Q\vert_{\ga_x(0)}\ra Q\vert_{\ga_x(2\pi)}$ which is equivariant under the $\SU(4)$-actions on $Q\vert_{\de_x(0)},Q\vert_{\de_x(2\pi)}$ from the principal $\SU(4)$-bundle. In this case $\de_x(0)=\de_x(2\pi)=(x,1)$, and $q$ identifies $Q\vert_{\de_x(0)}=Q\vert_{\de_x(2\pi)}\cong\SU(4)$, where $\SU(4)$ acts by left multiplication on itself. 

Hence $q$ identifies $\Hol_{\de_x}(\nabla_Q)$ with a smooth map $\SU(4)\ra\SU(4)$ equivariant under left multiplication by $\SU(4)$, which must be right multiplication by some $\Phi(x)\in \SU(4)$. This defines the map $\Phi:X\ra\SU(4)$ in (c), which is smooth as $\nabla_Q$ is smooth. Any two connections $\nabla_Q,\nabla_Q'$ on $Q$ are smoothly homotopic, so $\Phi$ is unique up to homotopy, and $[\Phi]$ is unique. This defines the map $[Q,q]\mapsto[\Phi]$ from objects (b) to objects~(c).

Conversely, let $\Phi:X\ra\SU(4)$ be a smooth map. Let $\sim$ be the equivalence relation $0\sim 2\pi$ on $[0,2\pi]$, and identify $[0,2\pi]/\mathbin{\sim}$ with $\cS^1$ by $\th\mapsto e^{i\th}$. Hence we also identify $X\t[0,2\pi]/\mathbin{\sim}$ with $X\t\cS^1$. Define $Q'\ra X\t\cS^1$ to be the principal $\SU(4)$-bundle $\bigl(X\t[0,2\pi]\t\SU(4)\bigr)/\mathbin{\approx}$, where $\approx$ is the equivalence relation $(x,0,\ep)\approx (x,2\pi,\ep\,\Phi(x))$ for $x\in X$ and $\ep\in\SU(4)$, where the projection $Q'\ra X\t\cS^1$ maps $[x,\th,\ep]\mapsto[x,\th]\in X\t[0,2\pi]/\mathbin{\sim}\cong X\t\cS^1$, and the $\SU(4)$-action on $Q'$ is by left multiplication on the $\SU(4)$ factor. Define $q':Q'\vert_{X\t\{1\}}\,{\buildrel\cong\over\longra}\,(X\t\{1\})\t\SU(4)$ to map $[x,0,\ep]\mapsto[x,\ep]$. Changing $\Phi$ by smooth homotopy changes $(Q',q')$ by smooth isotopy, and hence by isomorphism, so we have a map $[\Phi]\mapsto[Q,q]$ from objects (b) to objects (c). It is easy to see this is inverse to the map $[Q,q]\mapsto[\Phi]$ above, so (b),(c) are in 1-1 correspondence. 

The rest of Step 2 is clear.

\subsection{\texorpdfstring{Step 3: The geometry of $\SU(4)$}{Step 3: The geometry of SU(4)}}
\label{ss23}

Let $Y_k\subset\SU(4)$ for $k=0,\ldots,3$ be as in \eq{ss2eq1}. Write $\Gr(\C^k,\C^3)$ for the Grassmannian of vector subspaces $V\subset\C^3$ with $V\cong\C^k$, a compact complex manifold of dimension $k(3-k)$. Define a map $\phi_k:Y_k\ra \Gr(\C^k,\C^3)$ by
\begin{equation*}
\phi_k(A)=\bigl\{(x_1,x_2,x_3)\in\C^3:A\bs x=-\bs x\quad\text{for}\quad	\bs x=(x_1,x_2,x_3,0)^T\bigr\},
\end{equation*}
where the right hand side is a $k$-dimensional subspace of $\C^3$ by \eq{ss2eq1}, and thus a point of $\Gr(\C^k,\C^3)$. Note that $\Gr(\C^1,\C^3)=\CP^2$ and $\phi_1$ is $\phi$ in Step~3(iv).

The fibre of $\phi_k$ over $\bigl\{(x_1,\ldots,x_k,0,\ldots,0):x_j\in\C\bigr\}$ is
\e
\begin{gathered}
\left\{\begin{pmatrix} -1 & 0 & 0 & 0 \\
0 & \ddots & 0 & 0 \\
0 & 0 & -1 & 0 \\
0 & 0 & 0 & (-1)^kB	
\end{pmatrix}
\begin{aligned} :\, & \text{$B\in\SU(4-k)$, $B$ has no}\\ 
&\text{eigenvectors in $\C^{3-k}\subset\C^{4-k}$} \\
&\text{with eigenvalue $(-1)^{k+1}$}\end{aligned}\right\}.	
\end{gathered}
\label{ss2eq8}
\e
This is diffeomorphic to an open subset of $\SU(4-k)$, the complement of a codimension 3 subset, and so is connected and simply-connected. By considering the action of $\SU(3)\subset\SU(4)$ on $\SU(4)$ by conjugation, which preserves $Y_k$ and acts on $\Gr(\C^k,\C^3)$, we see that $Y_k$ is an embedded submanifold of $\SU(4)$ and $\phi_k$ is a fibre bundle with fibre \eq{ss2eq8}. Hence
\begin{align*}
\dim Y_k&=\dim_\R\Gr(\C^k,\C^3)+\dim\SU(4-k)=2k(3-k)+(4-k)^2-1\\
&=15-k(k+2)=\dim\SU(4)-k(k+2),
\end{align*}
so the codimension of $Y_k$ is $k(k+2)$. As $\Gr(\C^k,\C^3)$ and \eq{ss2eq8} are connected, simply-connected and oriented, $Y_k$ is embedded, connected, simply-connected and oriented. This proves Step~3(i).

Part (ii) is obvious. For (iii), for each $A\in Y_0\subset \SU(4)$ define vectors $e_t^j(A)\in\C^4$ for $j=1,2,3$ and $t\in[0,1]$ by 
\begin{equation*}
e_t^j(A)=te_j+(1-t)Ae_j,	
\end{equation*}
where $e_1=(1\, 0\, 0\, 0)^T$, $e_2=(0\, 1\, 0\, 0)^T$, and $e_3=(0\, 0\, 1\, 0)^T$. We claim that for each $A\in Y_0$ and $t\in[0,1]$, the vectors $e_t^1(A),e_t^2(A),e_t^3(A)$ are $\C$-linearly independent in $\C^4$. For if not, there would exist $0\ne(x_1,x_2,x_3)\in\C^3$ such that 
\begin{equation*}
(1-t)A(x_1e_1+x_2e_2+x_3e_3)=-t(x_1e_1+x_2e_2+x_3e_3),
\end{equation*}
so that $-t/(1-t)$ is an eigenvalue of $A$. As eigenvalues of $A$ have norm 1, this forces $t=\ha$, so $A\bs x=-\bs x$ for $0\ne\bs x=(x_1,x_2,x_3,0)^T$, contradicting~\eq{ss2eq1}. 

Next define vectors $f_t^j(A)\in\C^4$ for $j=1,2,3$ and $t\in[0,1]$ by
\begin{gather*}
f_t^1(A)=\frac{e_t^1(A)}{\bmd{e_t^1(A)}},\qquad f_t^2(A)=\frac{e_t^2(A)-\an{e_t^2(A),f_t^1(A)}_\C f_t^1(A)}{\bmd{e_t^2(A)-\an{e_t^2(A),f_t^1(A)}_\C f_t^1(A)}},\\
f_t^3(A)=\frac{e_t^3(A)-\an{e_t^3(A),f_t^1(A)}_\C f_t^1(A)-\an{e_t^3(A),f_t^2(A)}_\C f_t^2(A)}{\bmd{e_t^3(A)-\an{e_t^3(A),f_t^1(A)}_\C f_t^1(A)-\an{e_t^3(A),f_t^2(A)}_\C f_t^2(A)}}.
\end{gather*}
That is, we use the Gram--Schmidt process to make $f_t^1(A),f_t^2(A),f_t^3(A)$ Hermitian orthonormal in $\C^4$. There is then a unique $f_t^4(A)$ such that the matrix $\Psi_t(A)=\bigl(f_t^1(A)\,\cdots\,f_t^4(A)\bigr)$ with columns $f_t^1(A),\ldots,f_t^4(A)$ lies in $\SU(4)$. Then $\Psi_t(A)\in\SU(4)$ depends smoothly on $A\in Y_0$ and $t\in[0,1]$, with $\Psi^0(A)=A$ and $\Psi^1(A)=\Id$, so $\Psi_t:Y_0\ra\SU(4)$ for $t\in[0,1]$ satisfies Step~3(iii).

For (iv), for simplicity let $A\in Y_1$ with $\phi(A)=[1,0,0]\in\CP^2$. Then
\begin{equation*}
A\begin{pmatrix} 1 & 0 & 0 & 0 \end{pmatrix}{}^T=\begin{pmatrix} -1 & 0 & 0 & 0 \end{pmatrix}{}^T.
\end{equation*}
Let $\de A\in T_A\SU(4)$, so that we think of $\de A$ as a small $4\t 4$ complex matrix, with 
$A+\de A$ an infinitesimal perturbation of $A$ in $\SU(4)$. Then 
\begin{equation*}
(A+\de A)\begin{pmatrix} 1 & 0 & 0 & 0 \end{pmatrix}{}^T=-\begin{pmatrix} -1+\de A_{11} & \de A_{21} & \de A_{31} & \de A_{41} \end{pmatrix}{}^T,
\end{equation*}
for $\de A_{j1}\in\C$ small, with $\de A_{11}\in i\R$ as $(A+\de A)$ preserves lengths. Then the fibre $\nu\vert_A\subset T_A\SU(4)$ at $A$ of the normal bundle $\nu$ of $Y_1$ in $\SU(4)$ may be identified with $\R\op\C$ with coordinates $(\Im(\de A_{11}),\de A_{41})$. 

Here $\Im(\de A_{11})$ measures the tangent direction in $\SU(4)$ which varies the eigenvalue $-1$ of $A$ to $e^{i\th}$ in $\cS^1$ close to $-1$, so we should think of $\de A_{11}$ as lying in $T_{-1}\cS^1=i\R$. And $\de A_{41}$ measures the tangent directions in $\SU(4)$ which vary the eigenspace $[x_1,x_2,x_3,0]\in\CP^3$ of $A$ normal to $\CP^2=\bigl\{[y_1,y_2,y_3,0] \in\CP^3\bigr\}$ in $\CP^3$, where $A\in Y_1$ must have a $-1$-eigenvector in $\CP^2$, so we should think of $\de A_{41}$ as lying in $\ti\nu\vert_{[1,0,0,0]}$, where $\ti\nu$ is the normal bundle of $\CP^2=\bigl\{[y_1,y_2,y_3,0]\in\CP^3\bigr\}$ in $\CP^3$. Note that $\ti\nu\cong\O(1)$, so~$\ti\nu\vert_{[1,0,0,0]}\cong\phi^*(\O(1))\vert_A$.

More generally, if $A\in Y_1$ with $\phi(A)=[x_1,x_2,x_3]\in\CP^2$ for $x_1,x_2,x_3\in\C$ with $\md{x_1}+\ms{x_2}+\ms{x_3}=1$ and $\de A\in T_A\SU(4)$, then we can identify $\nu\vert_A$ with $\R\op\C$ with coordinates $(y,z)$, where in matrix notation
\begin{align*}
y&=-i\begin{pmatrix} \bar x_1 & \bar x_2 & \bar x_3 & 0\end{pmatrix}\de A  \begin{pmatrix} x_1 & x_2 & x_3 & 0\end{pmatrix}{}^T,\\
z&=\begin{pmatrix} 0 & 0 & 0 & 1\end{pmatrix}\de A  \begin{pmatrix} x_1 & x_2 & x_3 & 0\end{pmatrix}{}^T.	
\end{align*}
Multiplying the representative $(x_1,x_2,x_3)$ for $[x_1,x_2,x_3]$ by $e^{i\th}$ fixes $y$, but multiplies $z$ by $e^{i\th}$. So the invariant thing is to regard $z$ as lying in $\ti\nu\vert_{[x_1,x_2,x_3,0]}=\phi^*(\O(1))\vert_A$. This defines an isomorphism $\nu\cong\R\op\phi^*(\O(1))$, proving~(iv).

To prove \eq{ss2eq2}, by a well known calculation for all $m\ge 2$ we show that
\e
H^*(\SU(m),\Z)\cong \La_\Z[p_3,p_5,\ldots,p_{2m-1}]
\label{ss2eq9}
\e
by induction on $m$, where the first step $m=2$ follows from $\SU(2)\cong\cS^3$, and the inductive step from the Leray--Serre spectral sequence for the fibration $\SU(m-1)\hookra\SU(m)\twoheadrightarrow\cS^{2m-1}$. For \eq{ss2eq3}, the K\"unneth Theorem gives
\begin{equation*}
\mu^*(p_k)\in H^k(\SU(4)\t\SU(4),\Z)=\an{p_k\bt 1,1\bt p_k}_\Z,
\end{equation*}
so $\mu^*(p_k)=a_k\cdot (p_k\bt 1)+b_k\cdot (1\bt p_k)$ for $a_k,b_k\in\Z$, and $a_k=b_k=1$ follows by restricting $\mu$ to $\SU(4)\t\{\Id\}$ and $\{\Id\}\t\SU(4)$ in $\SU(4)\t\SU(4)$. 

For \eq{ss2eq4}, note that as $H^3(\SU(4),\Z)=\an{p_3}_\Z$ and $H^8(\SU(4),\Z)=\an{p_3\cup p_5}_\Z$, we have $[\ovY_1]=c\cdot\Pd(p_3)$ and $[\ovY_2]=d\cdot\Pd(p_3\cup p_5)$ for some $c,d\in\Z$. From \eq{ss2eq9} for $m=2,3,4$ we see that under the embeddings $\io:\SU(2)\hookra\SU(4)$ and $\jmath:\SU(3)\hookra\SU(4)$ given by
\begin{equation*}
\io:B\longmapsto \begin{pmatrix} 1 & 0 & 0 \\ 0 & 1 & 0 \\ 0 & 0 & B \end{pmatrix},\qquad
\jmath:C\longmapsto \begin{pmatrix} 1 & 0 \\ 0 & C \end{pmatrix},
\end{equation*}
we have $p_3\cdot\io_*([\SU(2)])=1$ and $(p_3\cup p_5)\cdot\jmath_*([\SU(3)])=1$. So using the intersection product $\bu$ on $H_*(\SU(4),\Z)$ we have $c=[\ovY_1]\bu\io_*([\SU(2)])$ and $d=[\ovY_2]\bu\jmath_*([\SU(3)])$. But $\ovY_1$ intersects $\io(\SU(2))$ transversely in one point $\diag(1,1,-1,-1)$ in $Y_1$, and $\ovY_2$ intersects $\jmath(\SU(3))$ transversely in one point $\diag(1,-1,-1,1)$ in $Y_2$, so these intersection numbers are $\pm 1$, and choosing orientations on $Y_1,Y_2$ appropriately we can ensure that $c=d=1$, proving~\eq{ss2eq4}. 

\subsection{\texorpdfstring{Step 4: Reduction to the case $\ka([\Phi])=0$}{Step 4: Reduction to the case κ([Φ])=0}}
\label{ss24}

Define maps $\la_3,\la_5,\la_7$ and $\ka$ on $[X,\SU(4)]$ as in Step 4. Let $\al\in H^5(X,\Z)$, so that $\Pd(\al)\in H_3(X,\Z)$. We can choose a compact, oriented, embedded 3-submanifold $W\subset X$ with $[W]=\Pd(\al)$. Write $\hat\nu\ra W$ for the normal bundle of $W$ in $X$. Now $W$ admits a spin structure, as any oriented 3-manifold does, and $X$ is spin, so $\hat\nu$ admits a spin structure on its fibres. Hence $\hat\nu$ is trivial, as any $\Spin(5)$-bundle on a 3-manifold is trivial. Thus we may choose a tubular neighbourhood $T$ of $W$ in $X$ and a diffeomorphism $T\cong W\t B^5$, where $B^5\subset\R^5$ is the open unit ball.

By Mimura and Toda \cite{MiTo} we have $\pi_5(\SU(4))\cong\Z$, \color{red}
and the natural map $\pi_5(\SU(4))\ra H_5(\SU(4),\Z)$ is an isomorphism.\color{black}\footnote{\color{red} Unfortunately, this is false. The natural map is multiplication by $4!=24:\Z\ra\Z$. This makes the proof of this step wrong. See the \hyperref[erratum]{Erratum} for more details.\color{black}} Thus there exists a smooth map $\Psi:\cS^5=\R^5\cup\{\iy\}\ra\SU(4)$ with $\Psi^*(p_5)\cdot[\cS^5]=1$. We may choose $\Psi$ with $\Psi\equiv\Id$ outside the ball $B^5_{1/2}$ of radius $\ha$ in $\R^5\subset\cS^5$. Define $\Phi':X\ra\SU(4)$ by $\Phi'\vert_{X\sm T}\equiv\Id$, and $\Phi'\vert_T$ is identified under $T\cong W\t B^5$ with the map $W\t B^5\ra\SU(4)$, $(w,b)\mapsto\Psi(b)$. As $\Psi\equiv 1$ on $B^5\sm B^5_{1/2}$, this $\Phi'$ is smooth.

Since $\Psi^*(p_5)\cdot[\cS^5]=1$, it follows that $\Pd\ci\la_5([\Phi'])=\Pd\ci\Phi^{\prime *}(p_5)=[W]=\Pd(\al)$, so $\la_5([\Phi'])=\al$. As $\Phi'\vert_{X\sm T}\equiv\Id$, and the morphism $H_3(X\sm T,\Z)\ra H_3(X,\Z)$ induced by the inclusion $X\sm T\hookra X$ is an isomorphism for dimensional reasons, we see that $\la_3([\Phi'])=0$, as we want. The rest of Step 4 is clear.

\subsection{\texorpdfstring{Step 5: A 5-submanifold $Z\subset X$ with $\Phi\simeq 1$ on $X\sm Z$}{Step 5: A 5-submanifold Z⊂X with Φ≃1 on X \ Z}}
\label{ss25}

Suppose $X$ is connected, and $[\Phi]\in[X,\SU(4)]$ with $\ka([\Phi])=0$. Choose a generic representative $\Phi:X\ra\SU(4)$ for $[\Phi]$. Then as in Step 5 at the beginning of \S\ref{ss2}, $\Phi$ is an embedding near $Y_1\amalg Y_2\amalg Y_3$ in $\SU(4)$, and $\Phi(X)\cap Y_2$ is a compact, oriented 0-manifold, that is, a finite set of points with signs $\pm 1$, and the number of points in $\Phi(X)\cap Y_2$ counted with signs is $\ka([\Phi])=0$. Thus we may write $\Phi(X)\cap Y_2=\{r_1,\ldots,r_k,s_1,\ldots,s_k\}$, where the $r_i$ have positive orientation and the $s_i$ negative orientation. There are unique disjoint $p_1,\ldots,p_k,q_1,\ldots,q_k\in X$ with $\Phi(p_i)=r_i$ and~$\Phi(q_i)=s_i$.

As $X$ is connected we may choose smooth embedded paths $\ga_i:[0,1]\ra X$ with $\ga_i(0)=p_i$ and $\ga_i(1)=q_i$ for $i=1,\ldots,k$, and as $\dim X>2$ we may choose $\ga_1([0,1]),\ldots,\ga_k([0,1])$ to be disjoint. Then $\Phi\ci\ga_i:[0,1]\ra\SU(4)$ are smooth embedded paths in $\SU(4)$ with $\Phi\ci\ga_i(0)=r_i$ and $\Phi\ci\ga_i(1)=s_i$, with $r_i,s_i$ the only intersection points of $\Phi\ci\ga_i([0,1])$ with~$Y_2$.

As $Y_2$ is connected we may choose smooth embedded paths $\de_i:[0,1]\ra Y_2$ with $\de_i(0)=r_i$ and $\de_i(1)=s_i$. Then $\Phi\ci\ga_i$ and $\de_i$ are both smooth paths $[0,1]\ra\SU(4)$ with end points $r_i,s_i$. Since $\SU(4)$ is simply-connected we may choose smooth homotopies $\ep_i:[0,1]^2\ra\SU(4)$ with $\ep_i(0,t)=\Phi\ci\ga_i(t)$, $\ep_i(1,t)=\de_i(t)$, $\ep_i(s,0)=r_i$, $\ep_i(s,1)=s_i$ for all $s,t\in[0,1]$, where we may take $\ep_i$ to be an embedding on $[0,1]\t(0,1)$, and to map to $Y_2$ only at $(s,0),(s,1)$ and~$(1,t)$.

We now use the `Whitney trick' (as used in the proof of the Whitney Embedding Theorem): we modify $\Phi$ in small open neighbourhoods of the paths $\ga_1([0,1]),\ldots,\ga_k([0,1])$ in $X$, deforming $\Phi$ along the disks $\ep_i([0,1]^2)$ in $\SU(4)$, so as to eliminate the intersection points $r_i,s_i$ of $\Phi(X)\cap Y_2$ in pairs. Thus we may perturb $\Phi$ in its homotopy class so that $\Phi(X)\cap Y_2=\es$. We can also suppose $\Phi$ is an embedding near $Y_1$ in $\SU(4)$, and $\Phi(X)$ intersects $Y_1$ transversely.

As in Step 5, it now follows that $Z=\{x\in X:\Phi(x)\in Y_1\}$ is a compact, oriented, embedded 5-submanifold in $X$ diffeomorphic to $\Phi(X)\cap Y_1$, and  defining $\psi:Z\ra\CP^2$ by $\psi=\phi\ci\Phi\vert_Z$, the normal bundle $\nu_Z$ of $Z$ in $X$ satisfies $\nu_Z\cong \R\op\psi^*(\O(1))$ as in \eq{ss2eq6}, and $w_2(Z)$ is the image in $H^2(Z,\Z_2)$ of the integral class $\psi^*(c_1(\O(1)))$ in~$H^2(Z,\Z)$.

Next suppose $X$ is connected. We will show that we can perturb $\Phi$ in its homotopy class to make $Z$ connected. Suppose first that $Z$ has two connected components $Z_0$ and $Z_1$. As $X$ is connected we can choose points $z_0\in Z_0$, $z_1\in Z_1$ and a smooth path $\ga:[0,1]\ra X$ with $\ga(0)=z_0$ and $\ga(1)=z_1$, where we suppose that $\ga([0,1])$ is embedded and meets $Z$ transversely only at $z_0,z_1$. Then $\Phi\ci\ga:[0,1]\ra\SU(4)$ is a smooth path in $\SU(4)$ with end-points $\Phi(z_0),\Phi(z_1)$ in $Y_1$. As $Y_1$ is connected we can choose a smooth path $\de:[0,1]\ra Y_1$ with $\de(0)=\Phi(z_0)$, $\de(1)=\Phi(z_1)$, where we suppose that $\de([0,1])$ is embedded and meets $\Phi(X)\cap Y_1=\Phi(Z)$ transversely only at~$\Phi(z_0),\Phi(z_1)$.

Then $\Phi\ci\ga([0,1]),\de([0,1])$ are two paths from $\Phi(z_0)$ to $\Phi(z_1)$ in $\SU(4)$, so $\Phi\ci\ga([0,1])\cup\de([0,1])$ is a piecewise-smooth embedded circle in $\SU(4)$. As $\SU(4)$ is simply-connected we may choose a smooth embedded 2-disc $D$ in $\SU(4)$, with boundary $\Phi\ci\ga([0,1])\cup\de([0,1])$, and corners at~$\Phi(z_0),\Phi(z_1)$.

In a similar way to the use of the `Whitney trick' above, we may modify $\Phi:X\ra\SU(4)$ in a small open neighbourhood of $\ga([0,1])$ in $X$ to a new $\Phi':X\ra\SU(4)$, where we deform $\Phi$ near $\ga([0,1])$ along the disc $D$ in $\SU(4)$, so that $\Phi'$ near $\ga([0,1])$ is close to the path $\de([0,1])$ in $\SU(4)$. We can arrange that $Z'=\Phi^{\prime -1}(Y_1)$ near $\ga([0,1])$ is a tube $[0,1]\t\cS^4$, where $\cS^4$ is a small 4-sphere. That is, we replace $Z=Z_1\amalg Z_2$ by the connected sum $Z'=Z_1\# Z_2$, joining $Z_1,Z_2$ by a narrow neck $[0,1]\t\cS^4$ close to $\ga([0,1])$ in $X$, and making $Z'$ connected. If $Z$ has $k>2$ connected components, we use the trick above $k-1$ times to make $Z'$ connected.

Finally, suppose $X$ is simply-connected. We will show that we can perturb $\Phi$ in its homotopy class to make $Z$ simply-connected. By surgery theory, as $Z$ is a compact, oriented 5-manifold, we can choose disjoint embedded circles $L_1,\ldots,L_k$ in $Z$, and small tubular neighbourhoods $T_1,\ldots,T_k$ of $L_1,\ldots,L_k$ with $T_i$ diffeomorphic to $L_i\t B^4$, such that deleting $T_1,\ldots,T_k$ and gluing in $k$ copies of $B^2\t\cS^3$ along the common boundary $\cS^1\t\cS^3$ of $L_i\t B^4$ and $B^2\t\cS^3$, gives a compact, simply-connected 5-manifold~$\hat Z$.

As $X$ is simply-connected, each circle $L_i$ in $Z\subset X$ may be written as $L_i=\pd D_i$, for $D_i\subset X$ a 2-disc in $X$. By perturbing $D_i$ generically we can suppose that $D_i$ is embedded, that it intersects $Z$ transversely only at $\pd D_i=L_i$, and that $D_1,\ldots,D_k$ are disjoint.

Also $\Phi(L_i)$ is an embedded circle in $Y_1$. As $Y_1$ is simply-connected, we may write $\Phi(L_i)=\pd E_i$, for $E_i\subset Y_1$ a 2-disc in $Y_1$. By perturbing $E_i$ generically we can suppose that $E_i$ is embedded, that it intersects $\Phi(Z)=\Phi(X)\cap Y_1$ transversely only at $\pd E_i=\Phi(L_i)$, and that $E_1,\ldots,E_k$ are disjoint.

We now have embedded 2-discs $\Phi(D_i),E_i$ in $\SU(4)$ with common boundary $\Phi(L_i)$, so $\Phi(D_i)\cup E_i$ is a piecewise-smooth $\cS^2$ in $\SU(4)$. Since $\pi_2(\SU(4))=0$ by \cite{MiTo}, we may choose smooth embedded 3-discs $F_1,\ldots,F_k$ in $\SU(4)$, with boundary $\pd F_i=\Phi(D_i)\cup E_i$, and a codimension 2 corner along~$\Phi(L_i)$.

Again, in a similar way to the use of the `Whitney trick' above, we may modify $\Phi:X\ra\SU(4)$ in small open neighbourhoods of $D_1,\ldots,D_k$ in $X$ to a new $\Phi':X\ra\SU(4)$, where we deform $\Phi$ near $D_i$ along the 3-disc $F_i$ in $\SU(4)$, so that $\Phi'$ near $D_i$ is close to the disc $E_i$ in $Y_1\subset\SU(4)$. We also need to consider the (trivial) normal bundles of $L_i$ and $D_i$ in $X$, and their images as subbundles of the (trivial) normal bundles of $\Phi(L_i)$ and $\Phi(D_i)$ in $\SU(4)$, and how these subbundles deform along $F_i$ to~$E_i$.

If we only deform $\Phi$ along $F_i$ to $\Phi'$ such that $\Phi'(D_i)=E_i$, then $\Phi'(X)$ may intersect $Y_1$ non-transversely along $E_i$, and $Z'=\Phi^{\prime -1}(Y_1)$ will not be a submanifold of $X$. However, if we deform $\Phi'$ a little way further, pushing $\Phi'(D_i)$ a little way beyond the boundary of $F_i$ at $E_i$, then $Z'=\Phi^{\prime -1}(Y_1)$ becomes a 5-submanifold of $X$ locally modelled near $D_i$ on $D_i\t\cS^3$, where the $\cS^3$ factors are small spheres in a (trivial) rank 4 subbundle of the normal bundle of $E_i$ in $Y_1$. That is, $Z'$ is diffeomorphic to the 5-manifold $\hat Z$ constructed above by surgery on $L_1,\ldots,L_k$, so $Z'$ is simply-connected.

Therefore if $X$ is simply-connected, we can perturb $\Phi$ in its homotopy class to make $Z$ simply-connected, completing Step~5.

\subsection{\texorpdfstring{Step 6: $\B_P$ is orientable if $X$ is simply-connected}{Step 6: ℬᵨ is orientable if X is simply-connected}}
\label{ss26}

Suppose $X$ is connected and simply-connected, let $[\ga]\in\pi_1(\B_P)$ correspond to $[\Phi]$ in $[X,\SU(4)]$ as in Step 2 with $\ka([\Phi])=0$ as in Step 4, and choose $\Phi,Z,\psi,\nu_Z$ with $Z$ connected and simply-connected as in Step 5. Then $Z$ is a compact, oriented 5-manifold, $\phi:Z\ra\CP^2$ is smooth, the normal bundle $\nu_Z$ of $Z$ in $X$ is $\nu_Z\cong\R\op\psi^*(\O(1))$, and $w_2(Z)$ is the image in $H^2(Z,\Z_2)$ of the integral class $\psi^*(c_1(\O(1)))$ in~$H^2(Z,\Z)$.

In the next proposition, using results of Crowley \cite{Crow} on the diffeomorphism classification of compact, simply-connected 5-manifolds, we will construct a compact, oriented, spin 8-manifold $X'$ with $H^{\rm odd}(X',\Z)=0$, and an embedding $\jmath:Z\hookra X'$, such that the normal bundle $\nu_Z'$ of $Z$ in $X'$ is $\nu'_Z\cong\R\op\psi^*(\O(1))$. Thus, tubular neighbourhoods of $Z$ in $X$ and $X'$ are diffeomorphic. 

\begin{prop} Suppose $Z$ is a compact, connected, simply-connected, oriented $ 5$-manifold, and\/ $L\ra Z$ is a complex line bundle, such that the second Stiefel--Whitney class $w_2(Z)$ is the image of\/ $c_1(L)$ under the projection $H^2(Z,\Z)\ra H^2(Z,\Z_2)$. Then there exist group isomorphisms
\begin{equation*}
H^2(Z,\Z)\cong\Z^r,\qquad H_2(Z,\Z)\cong\Z^r\op G\op G,
\end{equation*}
for some $r\ge 0$ and finite abelian group $G,$ such that the pairing $H^2(Z,\Z)\t H_2(Z,\Z)\ra\Z$ maps $(a_1,\ldots,a_r)\cdot (b_1,\ldots,b_r,g_1,g_2)\mapsto a_1b_1+\cdots+a_rb_r,$ and\/ $c_1(L)$ is identified with\/ $(k,0,\ldots,0)$ for some $k\in\Z$. Furthermore:
\begin{itemize}
\setlength{\itemsep}{0pt}
\setlength{\parsep}{0pt}
\item[{\bf(a)}] If\/ $k=0,$ so that\/ $c_1(L)=0$ and\/ $L$ is trivial, and\/ $w_2(Z)=0$ so $Z$ is spin, there exists an embedding $\io:Z\hookra\cS^6$ with trivial normal bundle $\R$. Composing this with the obvious embedding $\cS^6\hookra\cS^8$ mapping $(x_1,\ldots,x_7)\mapsto (x_1,\ldots,x_7,0,0)$ gives an embedding $\jmath:Z\hookra X'=\cS^8$ with trivial normal bundle\/~$\nu'_Z\cong\R^3\cong\R\op L$.

\item[{\bf(b)}] If\/ $k\ne 0$ is even, so that\/ $r\ge 1,$ and\/ $L$ is nontrivial, and\/ $w_2(Z)=0$ so $Z$ is spin, there exists an embedding $\io:Z\hookra\CP^1\t\cS^4$ with trivial normal bundle $\R,$ such that\/ $L\cong (\pi_{\CP^1}\ci\io)^*(\O(k)),$ for $\O(1)\ra\CP^1$ the standard line bundle, $\O(k)=\O(1)^{\ot^k},$ and\/ $\pi_{\CP^1}:\CP^1\t\cS^4\ra\CP^1$ the projection.

Define $X'\ra\CP^1\t\cS^4$ to be the $\CP^1$-bundle $\mathbb{P}\bigl(\pi_{\CP^1}^*(\O(0)\op\O(k))\bigr)$. This bundle has a natural section $[1,0]:\CP^1\t\cS^4\ra X'$ embedding $\CP^1\t\cS^4$ as a submanifold of\/ $X'$ with normal bundle\/ $\pi_{\CP^1}^*(\O(k))$. Hence $\jmath=[1,0]\ci\io$ is an embedding $\jmath:Z\hookra X'$ with normal bundle $\nu'_Z\cong\R\op L$. 
\item[{\bf(c)}] Let\/ $k$ be odd, so that\/ $r\ge 1,$ and\/ $L$ is nontrivial, and\/ $w_2(Z)\ne 0$ so $Z$ is not spin. Write\/ $\pi_{\CP^1}:Y\ra\CP^1$ for the nontrivial\/ $\cS^4$ bundle, constructed by writing $\CP^1=D^+\cup_{\cS^1}D^-$ as the union of closed\/ $2$-discs $D^\pm$ along their boundary $\cS^1,$ and defining $Y=(D^+\t\cS^4)\cup_{\cS^1\t\cS^4}(D^-\t\cS^4),$ where the boundaries\/ $\cS^1\t\cS^4$ of\/ $D^\pm\t\cS^4$ are glued using a map $\cS^1\ra\SO(5)$ representing the nontrivial element of\/~$\pi_1(\SO(5))\cong\Z_2$. 

Then there exists an embedding $\io:Z\hookra  Y$ with trivial normal bundle $\R,$ such that\/~$L\cong (\pi_{\CP^1}\ci\io)^*(\O(k))$.

Define $X'\ra Y$ to be the $\CP^1$-bundle $\mathbb{P}\bigl(\pi_{\CP^1}^*(\O(0)\op\O(k))\bigr)$. This bundle has a natural section $[1,0]:Y\ra X'$ embedding $Y$ as a submanifold of\/ $X'$ with normal bundle\/ $\pi_{\CP^1}^*(\O(k))$. Hence $\jmath=[1,0]\ci\io$ is an embedding $\jmath:Z\hookra X'$ with normal bundle $\nu'_Z\cong\R\op L$. 
\end{itemize}
In each of\/ {\bf(a)\rm--\bf(c)\rm,} $X'$ is compact, oriented and spin, with\/~$H^{\rm odd}(X',\Z)=0$.
\label{ss2prop1}	
\end{prop}

\begin{proof} The Universal Coefficient Theorem implies that the torsions of $H_1(Z,\Z)$ and $H^2(Z,\Z)$ are isomorphic. But $H_1(Z,\Z)=0$ as $Z$ is simply-connected, so $H^2(Z,\Z)$ is torsion-free, and thus $H^2(Z,\Z)\cong\Z^r$ for $r=b^2(Z)$. We may choose the isomorphism $H^2(Z,\Z)\cong\Z^r$ to identify $c_1(L)$ with $(k,0,\ldots,0)$ for some $k\in\Z$, as any element of $\Z^r$ is conjugate to some $(k,0,\ldots,0)$ under $\SL(r,\Z)$. Then $H_2(Z,\Z)\cong \Z^r\op K$ for some finite abelian group $K$, such that the pairing $H^2(Z,\Z)\t H_2(Z,\Z)\ra\Z$ maps $(a_1,\ldots,a_r)\cdot (b_1,\ldots,b_r,k)\mapsto a_1b_1+\cdots+a_rb_r$. The results of Crowley \cite{Crow} discussed next imply that $K$ is of the form $G\op G$. This proves the first part of the proposition.

Crowley \cite{Crow} describes the classification of compact, connected, simply-conn\-ect\-ed 5-manifolds $Z$ up to diffeomorphism. To each such $Z$ we associate a pair $(\Ga,w)$ of a finitely generated abelian group $\Ga$ and a morphism $w:\Ga\ra\Z_2$, by $\Ga=H_2(Z,\Z)$ and $w(\ga)=w_2(Z)\cdot\ga$. Using results of Smale and Barden, Crowley notes that the map from diffeomorphism classes of 5-manifolds $Z$ to isomorphism classes of pairs $(\Ga,w)$ is injective. He then characterizes which pairs $(\Ga,w)$ lie in the image of this map, and in some cases gives an embedding $\io:Z\hookra Y$ into an explicit 6-manifold $Y$, and thus writes $Z=\pd\hat Y$ for a 6-manifold with boundary~$\hat Y$.

Any finitely generated abelian group $\Ga$ is of the form $\Z^r\op K$ for $K$ a finite abelian group, and $w:\Ga\ra\Z_2$ is the sum of morphisms $\Z^r\ra\Z_2$ and $K\ra\Z_2$. If $w_2(Z)$ lies in the image of $\Z^r\cong H^2(Z,\Z)\ra H^2(Z,\Z_2)$, as in our situation, the morphism $K\ra\Z_2$ is zero. This excludes many cases in Crowley's classification. In particular, Crowley allows either $K\cong G\op G$ or $K\cong G\op G\op\Z_2$ for $G$ a finite abelian group, but $K\cong G\op G\op\Z_2$ occurs only if $w\vert_K\ne 0$, and so does not happen in our case. This justifies~$K\cong G\op G$.

Note too that Crowley's classification is compatible with connected sums: if $Z$ corresponds to $(\Ga,w)$ with $(\Ga,w)\cong(\Ga_1,w_1)\op(\Ga_2,w_2)$ for $(\Ga_i,w_i)$ corresponding to $Z_i$, $i=1,2$, then~$Z\cong Z_1\# Z_2$.

For (a), given any $r\ge 0$ and finite abelian group $G$, Crowley \cite[\S 2.1]{Crow} constructs a compact, connected, simply-connected, spin 5-manifold $Z'$ with $H_2(Z')\cong\Z^r\op G\op G$ as follows: starting from $\Z^r\op G\op G$ he constructs a finite CW-complex $C$ with only 2- and 3-cells, chooses an embedding $C\hookra\R^6$, takes $D$ to be a regular open neighbourhood of $C$ in $\R^6$, so that $\kern .15em\ov{\kern -.15em D}$ is a compact 6-manifold with boundary, and defines $Z'=\pd\kern .15em\ov{\kern -.15em D}$. The theorem of Smale referred to above implies that for $Z$ as in the proposition with $k=0$, there is a diffeomorphism $\io:Z\ra Z'$. Then $\io:Z\hookra\R^6\subset\cS^6$ is an embedding with trivial normal bundle $\R$. The rest of (a) is immediate.

For (b)--(c), for $Z$ as in the proposition with $k\ne 0$ so $r\ge 1$, we may split
\begin{equation*}
\bigl(H_2(Z,\Z),w_2(Z)\,\cdot\,\bigr)\cong(\Z,k\mathbin{\rm mod} 2)\op (\Z^{r-1}\op G\op G,0).	
\end{equation*}
Then as above we have $Z\cong Z_1\# Z_2$ for $Z_1,Z_2$ with $H_2(Z_1,\Z)\cong\Z$, $w_2(Z_1)=k\mathbin{\rm mod} 2$, $H_2(Z_2,\Z)\cong\Z^{r-1}\op G\op G$, $w_2(Z_2)=0$. This is only valid if $Z_1,Z_2$ exist with these invariants, but we will justify this shortly.

In case (b), when $w_1(Z_1)=0$ as $k$ is even, we may take $Z_1=\CP^1\t\cS^3$, with an embedding $\io_1:Z_1\hookra\CP^1\t\cS^4$ from the identity on $\CP^1$ and the equator embedding $\cS^3\hookra\cS^4$. Part (a) gives a 5-manifold $Z_2$ with $H_2(Z_2,\Z)\cong\Z^{r-1}\op G\op G$ and $w_2(Z_2)=0$, and an embedding $\io_2:Z_2\hookra\cS^6$. Taking connected sums of both 5- and 6-manifolds gives an embedding
\begin{equation*}
Z_1\# Z_2\hookra (\CP^1\t\cS^4)\#\cS^6\cong \CP^1\t\cS^4.
\end{equation*}
For $Z$ as in part (b), the diffeomorphism $Z\cong Z_1\# Z_2$ gives an embedding $\io:Z\hookra\CP^1\t\cS^4$, with trivial normal bundle $\R$. The pullback $\io^*$ in
\begin{equation*}
\io^*:H^2(\CP^1\t\cS^4)\cong\Z\longra H^2(Z,\Z)\cong H^2(Z_1,\Z) \op H^2(Z_2,\Z)\cong \Z\op\Z^{r-1}
\end{equation*}
acts by $a\mapsto (a,0,\ldots,0)$. Hence $\io^*(c_1(\pi_{\CP^1}^*(\O(k)))=(k,0,\ldots,0)=c_1(L)$, which implies that $(\pi_{\CP^1}\ci\io)^*(\O(k))\cong L$. The rest of (b) is immediate.

In case (c), when $w_1(Z_1)=1\mathbin{\rm mod} 2$ as $k$ is odd so $Z_1$ is not spin, as in Crowley \cite[\S 2]{Crow} we take $\pi_{\CP^1}:Z_1\ra\CP^1$ to be the nontrivial $\cS^3$-bundle over $\CP^1$ (this is $X_\iy$ in Crowley's notation), defined as in the proposition with $\cS^3,\SO(4)$ in place of $\cS^4,\SO(5)$. For $\pi_{\CP^1}:Y\ra\CP^1$ as in the proposition, there is a natural embedding $\io_1:Z_1\hookra Y$ with $\pi_{\CP^1}\ci\io_1=\pi_{\CP^1}$, which embeds the $\cS^3$ fibres of $\pi_{\CP^1}:Z_1\ra\CP^1$ as equators in the $\cS^4$ fibres of $\pi_{\CP^1}:Y\ra\CP^1$. The rest of (c) follows (b), replacing $\pi_{\CP^1}:\CP^1\t\cS^4\ra\CP^1$ by~$\pi_{\CP^1}:Y\ra\CP^1$.

For the last part, in (a) we have $X'=\cS^8$, which is compact, oriented and spin, with $H^{\rm odd}(X',\Z)=0$. In (b) we have a fibration $\CP^1\hookra X'\twoheadrightarrow \CP^1\t\cS^4$, so $X'$ is compact and oriented as $\CP^1,\CP^1\t\cS^4$ are, and $H^{\rm odd}(X',\Z)=0$ by the Leray--Serre spectral sequence as $H^{\rm odd}(\CP^1,\Z)=H^{\rm odd}(\CP^1\t\cS^4,\Z)=0$. We can also show $w_2(X')=0$ as $k$ is even, so $X'$ is spin.

In case (c) we have fibrations $\CP^1\hookra X'\twoheadrightarrow Y$ and $\cS^4\hookra Y\twoheadrightarrow\CP^1$, so $X'$ is compact and oriented with $H^{\rm odd}(X',\Z)=0$ as in (b). It is less obvious that $X'$ is spin, since $Y$ is not. The composition $X'\ra Y\ra\CP^1$ induces a pullback map $\{0,1\}=H^2(\CP^1,\Z_2)\ra H^2(X',\Z_2)$. We can compute $w_2(X')$, and we find that $Y$ being non-spin, and $k$ being odd, both contribute the image of $1\in H^2(\CP^1,\Z_2)$ to $w_2(X')\in H^2(X',\Z_2)$, but the sum of these contributions is 0, so $w_2(X')=0$ and $X'$ is spin. This completes the proof.
\end{proof}

Now let us return to the situation of Step 6, with $Z\subset X$ a compact, connected, simply-connected, oriented, embedded 5-submanifold, and $\psi:Z\ra\CP^2$ a smooth map, such that the normal bundle $\nu_Z$ of $Z$ in $X$ is $\nu_Z\cong\R\op\psi^*(\O(1))$, and $w_2(Z)$ is the image of $c_1(\psi^*(\O(1)))$ in $H^2(Z,\Z_2)$. Proposition \ref{ss2prop1} with $L=\psi^*(\O(1))$ constructs a compact, oriented, spin 8-manifold $X'$ with $H^{\rm odd}(X',\Z)=0$ and an embedding $\jmath:Z\hookra X'$ with normal bundle $\nu'_Z\cong\R\op\psi^*(\O(1))$, so the normal bundles of $Z$ in $X$ and $X'$ agree. Hence we can choose tubular neighbourhoods $U,U'$ of $Z$ in $X,X'$ and an orientation-preserving diffeomorphism~$\io:U\ra U'$.

Choose Riemannian metrics $g,g'$ on $X,X'$ such that $\io$ identifies $g\vert_U$ with $g'\vert_{U'}$, and let $E_\bu,E'_\bu$ be the positive Dirac operators of $g,g'$. Since $Z$ and hence $U,U'$ are simply-connected, the spin structures on $U,U'$ are unique, and so $\io$ is also spin-preserving, and thus identifies $E_\bu\vert_U$ and~$E'_\bu\vert_{U'}$. 

As in Step 5 we have $\Phi:X\ra\SU(4)$ with $\Phi^{-1}(Y_1)=Z$ and $\Phi^{-1}(Y_2)=\Phi^{-1}(Y_3)=\es$, so that $X\sm Z=\Phi^{-1}(Y_0)$. But $Y_0$ retracts to $\{\Id\}$ in $\SU(4)$ by Step 3(iii). Thus we may deform $\Phi$ in its homotopy class to make $\Phi\equiv\Id$ except close to $Z$, so we can choose an open set $V$ in $X$ such that $X=U\cup V,$ and $Z\cap\ov V=\es,$ and deform $\Phi$ so that~$\Phi\vert_V\equiv\Id$.

Define $V'=\io(U\cap V)\cup(X'\sm U')$. Then $V'\subset X'$ is open with $X'=U'\cup V'$, and $\io$ identifies $U\cap V$ with $U'\cap V'$. Define $\Phi':X'\ra\SU(4)$ by $\Phi'\vert_{U'}=\Phi\ci\io^{-1}$ and $\Phi'\vert_{V'}\equiv\Id$. Then $\Phi'$ is smooth, and $\io$ identifies $\Phi\vert_U$ and~$\Phi'\vert_{U'}$.

Let $Q\ra X\t\cS^1$, $q:Q\vert_{X\t\{1\}}\,{\buildrel\cong\over\longra}\,(X\t\{1\})\t\SU(4)=P$ and $Q'\ra X'\t\cS^1$, $q':Q'\vert_{X'\t\{1\}}\,{\buildrel\cong\over\longra}\,(X'\t\{1\})\t\SU(4)=P'$ correspond to $\Phi:X\ra\SU(4)$ and $\Phi':X'\ra\SU(4)$ by the 1-1 correspondence between (b),(c) in Step 2. Then the diffeomorphism $\io:U\ra U'$ identifying $\Phi\vert_U$ and $\Phi'\vert_{U'}$ induces an isomorphism $\io\t\id_{\cS^1}:U\t\cS^1\ra U'\t\cS^1$, and an isomorphism $\si:Q\vert_{U\t\cS^1}\ra (\io\t\id_{\cS^1})^*(Q'\vert_{U'\t\cS^1})$ of principal $\SU(4)$-bundles over $U\t\cS^1$ compatible with $q\vert_{U\t\{1\}},q'\vert_{U'\t\{1\}}$. Also $\Phi\vert_V\equiv\Id$ and $\Phi'\vert_{V'}\equiv\Id$ induce trivializations $\tau:Q\vert_{V\t\cS^1}\ra V\t\cS^1\t\SU(4)$, $\tau':Q'\vert_{V'\t\cS^1}\ra V'\t\cS^1\t\SU(4)$, compatible with $q\vert_{V\t\{1\}},q'\vert_{V'\t\{1\}}$ and~$\si\vert_{(U\cap V)\t\cS^1}$.

Choose a partial connection $\nabla_Q^X$ on $Q\ra X\t\cS^1$ in the $X$ directions, such that $\nabla_Q^X\vert_{X\t\{1\}}$ is identified with $\nabla^0$ under $q$, and $\nabla_Q^X\vert_{V\t\cS^1}$ is identified with $\nabla^0\vert_{V\t\cS^1}$ under $\tau$. Then there is a unique partial connection $\nabla_{Q'}^{X'}$ on $Q'\ra X'\t\cS^1$ in the $X'$ directions, such that $\nabla_{Q'}^{X'}\vert_{X'\t\{1\}}$ is identified with $\nabla^0$ under $q'$, and $\nabla_{Q'}^{X'}\vert_{V'\t\cS^1}$ is identified with $\nabla^0\vert_{V'\t\cS^1}$ under $\tau'$, and $\nabla_Q^X\vert_{U\t\cS^1}$ is identified with $\nabla_{Q'}^{X'}\vert_{U'\t\cS^1}$ under~$\si$.

The 1-1 correspondence between (a),(b) in Step 2 identifies $Q,q,\nabla_Q^X$ and $Q',q',\nabla_{Q'}^{X'}$ with loops $\ga:\cS^1\ra\B_P$ and $\ga':\cS^1\ra\B_{P'}$ based at $[\nabla^0]$, where $P=X\t\SU(4)$ and $P'=X'\t\SU(4)$ are the trivial $\SU(4)$-bundles over~$X,X'$.

For each $z\in\cS^1$ we have principal $\SU(4)$-bundles $Q\vert_{X\t\{z\}}\ra X$, $Q'\vert_{X'\t\{z\}}\ab\ra X'$ with connections $\nabla_Q^X\vert_{X\t\{z\}}$, $\nabla_{Q'}^{X'}\vert_{X'\t\{z\}}$ representing points $\ga(z)\in\B_P$ and $\ga'(z)\in\B_{P'}$. Apply the Excision Theorem, Theorem \ref{ss2thm1}, to these, with $X,\ab X',E_\bu,E_\bu',\ab\SU(4),\ab Q\vert_{X\t\{z\}},\ab Q'\vert_{X'\t\{z\}},\ab U,\ab V,\ab U',V',\si\vert_{U\t\{z\}},\tau\vert_{V\t\{z\}},\tau'\vert_{V\t\{z\}}$ in place of $X^+,\ab X^-,\ab E_\bu^+,\ab E_\bu^-,\ab G,\ab P^+,\ab P^-,U^+,V^+,U^-,V^-,\si,\tau^+,\tau^-$. This gives an isomorphism of $\Z_2$-torsors~$\Om_z:\check O_P^{E_\bu}\vert_{\ga(z)}\,{\buildrel\cong\over\longra}\,\check O_{P'}^{E_\bu'}\vert_{\ga'(z)}$.

Theorem \ref{ss2thm1}(i) implies that $\Om_z$ varies continuously with $z\in\cS^1$. Hence the monodromy of $\check O_P^{E_\bu}$ around $\ga$ in $\B_P$, which is $\Th([\ga])$ in the notation of Step 2, equals the monodromy $\Th'([\ga'])$ of $\check O_{P'}^{E'_\bu}$ around $\ga'$ in $\B_{P'}$. But $\B_{P'}$ is orientable by \S\ref{ss21}(v) as $H^{\rm odd}(X',\Z)=0$, so $\Th'([\ga'])=1$, and thus $\Th([\ga])=\hat\Th([\Phi])=1$. Since this holds for all $[\Phi]\in[X,\SU(4)]$ with $\ka([\Phi])=0$, Step 4 implies that $\B_P$ is orientable when $X$ is simply-connected, completing Step~6.

\section{Proof of Theorem \ref{ss1thm2}}
\label{ss3}

\subsection{Background on H-spaces}
\label{ss31}

We first introduce H-spaces, following Hatcher \cite[\S 3.C]{Hatc} and Stasheff~\cite{Stas}.

\begin{dfn}
\label{ss3def1}
Let $X,Y$ be topological spaces. Continuous $f_0,f_1:X\ra Y$ are called {\it homotopic}, written $f_0\simeq f_1$, if there is a continuous $h:X\t[0,1]\ra Y$ with $h(x,0)=f_0(x)$ and $h(x,1)=f_1(x)$. Writing $f_t(x)=h(x,t)$, this means there is a continuous family $(f_t:X\ra Y)_{t\in[0,1]}$ interpolating between $f_0$ and $f_1$. We write this as $h:f_0\,{\buildrel\simeq\over\Longra}\,f_1$. Homotopy is an equivalence relation.

Write $\Topho$ for the category with objects topological spaces $X,Y$ and morphisms homotopy equivalence classes $[f]:X\ra Y$ of continuous maps $f:X\ra Y$. Then $f:X\ra Y$ is a {\it homotopy equivalence\/} in $\Top$ if $[f]:X\ra Y$ is an isomorphism in~$\Topho$.

An {\it H-space} is a triple $(X,e_X,\mu_X)$ where $X$ is a topological space, $e_X\in X$ is a base-point, and $\mu_X:X\t X \ra X$ is a continuous map such that the maps $x\mapsto\mu_X(e_X,x)$ and $x\mapsto\mu_X(x,e_X)$ are both homotopic to~$\id_X:X\ra X$. 

Write $X_0$ for the path-connected component of $X$ containing $e_X$. Then $X_0$ is determined uniquely by $\mu_X$, and if $e_X'\in X$ then $(X,e_X',\mu_X)$ is an H-space if and only if $e_X'\in X_0$. So the choice of base-point $e_X$ is not important, and we often omit it from the notation, referring to $(X,\mu_X)$ as an H-space. If $\mu_X$ is clear from context, we will call $X$ an H-space.

An H-space $(X,e_X,\mu_X)$ is called {\it commutative\/} (or {\it associative\/}) if $\mu_X$ is commutative (or associative) up to homotopy, that is, if $[\mu_X]$ is  commutative (or associative) in $\Topho$. From now on, all H-spaces in this paper will be assumed to be commutative and associative.

Let $(X,e_X,\mu_X)$ be an H-space. Then the set $\pi_0(X)$ of path-connected components of $X$ has the structure of a commutative, associative monoid, with identity $\pi_0(e_X)$ and multiplication $\pi_0(\mu_X)$. We say that $(X,e_X,\mu_X)$ is {\it grouplike\/} if $\pi_0(X)$ is an abelian group.

A {\it morphism of H-spaces}, or {\it H-map}, $f:(X,e_X,\mu_X)\ra(Y,e_Y,\mu_Y)$ is a continuous map $f:X\ra Y$ such that $f(e_X)\simeq e_Y$ (i.e.\ $f(e_X),e_Y$ are joined by a path in $Y$; equivalently, $f(X_0)\subseteq Y_0$) and~$f\ci\mu_X\simeq \mu_Y\ci(f\t f):X\t X\ra Y$. 	
\end{dfn}

Two more-or-less equivalent enhancements of H-spaces are $\Ga$-{\it spaces}, as in Segal \cite[\S 1]{Sega}, and $E_\iy$-{\it spaces}, as in May \cite{May2}. All the H-spaces we deal with in the proof of Theorem \ref{ss1thm2} may be enhanced to $\Ga$-spaces and $E_\iy$-spaces.

The next definition comes from May \cite[\S 1]{May1} and Caruso et al.~\cite[\S 1]{CCMT}.

\begin{dfn}
\label{ss3def2}
A {\it homotopy-theoretic group completion\/} of an H-space $(X,\mu_X)$ is an H-map $f:(X,\mu_X)\ra(Y,\mu_Y)$ to a grouplike H-space $(Y,\mu_Y)$ such that:
\begin{itemize}
\setlength{\itemsep}{0pt}
\setlength{\parsep}{0pt}
\item[(i)] The map on connected components $\pi_0(f) : \pi_0(X) \ra \pi_0(Y)$ is the group completion of the abelian monoid $\pi_0(X)$, and
\item[(ii)] The ring morphism $H_*(f):H_*(X,R)\ra H_*(Y,R)$ is localization by the action of $\pi_0(X)$ on $H_*(X,R)$ for all commutative rings $R$, where $H_*(X,R),H_*(Y,R)$ are $R$-algebras with multiplications~$H_*(\mu_X),H_*(\mu_Y)$.
\end{itemize}
\end{dfn}

May \cite[Lem.~2.1]{May1} shows that if an H-space $(X,\mu_X)$ can be enhanced to an $E_\iy$-space then it has a group completion. Caruso et al.~\cite[Prop.~1.2]{CCMT} prove a weak universal property of homotopy-theoretic group completions.

The next definition and proposition are new.

\begin{dfn}
\label{ss3def3}
Let $(X,\mu_X)$ be an H-space. 
\begin{itemize}
\setlength{\itemsep}{0pt}
\setlength{\parsep}{0pt}
\item[(a)] A {\it weak H-principal\/ $\Z_2$-bundle\/} on $(X,\mu_X)$ is a principal $\Z_2$-bundle $P\ra X$, such that there exists an isomorphism $p:P\bt_{\Z_2}P\ra\mu_X^*(P)$ of principal $\Z_2$-bundles on $X\t X$.

Two weak H-principal $\Z_2$-bundles $P,Q\ra X$ are {\it isomorphic\/} if there exists an isomorphism $P\cong Q$ of principal $\Z_2$-bundles on $X$.
\item[(b)] A {\it strong H-principal\/ $\Z_2$-bundle\/} $(P,p)$ on $(X,\mu_X)$ is a trivializable principal $\Z_2$-bundle $P\ra X$, with a choice of isomorphism $p:P\bt_{\Z_2}P\ra\mu_X^*(P)$, such that in principal $\Z_2$-bundles on $X\t X\t X$ we have
\ea
&(\mu_X\t\id_X)^*(p)\ci(p\bt\id_P)\cong (\id_X\t\mu_X)^*(p)\ci(\id_P\bt p):
\label{ss3eq1}\\
&P\bt_{\Z_2}P\bt_{\Z_2}P\longra
(\mu_X\ci(\mu_X\t\id_X))^*(P)\cong(\mu_X\ci(\id_X\t\mu_X))^*(P),
\nonumber
\ea
interpreted as for \eq{ss1eq9} using a homotopy $h:\mu_X\ci(\mu_X\t\id_X){\buildrel\simeq\over\Longra}\,\mu_X\ci(\id_X\t\mu_X)$, which exists as $\mu_X$ is homotopy associative, and \eq{ss3eq1} is independent of $h$ as $P$ is trivializable.

An {\it isomorphism\/} $\io:(P,p)\ra(Q,q)$ of strong H-principal $\Z_2$-bundles on $X$ is an isomorphism $\io:P\ra Q$ of $\Z_2$-bundles on $X$ with~$\io\ci p=q\ci(\io\bt\io)$.
\end{itemize}

Weak and strong H-principal $\Z_2$-bundles pull back along H-maps $f:(X,\mu_X)\ab\ra(Y,\mu_Y)$ in the obvious way. Note that a principal $\Z_2$-bundle $P\ra X$ (up to isomorphism) is equivalent to a map $f_P:X\ra B\Z_2$ (up to homotopy), and $P$ is a weak H-principal $\Z_2$-bundle if and only if $f_P$ is an H-map.
\end{dfn}

\begin{rem} Continuing Remark \ref{ss1rem5}, if we enhanced $(X,\mu_X)$ to a $\Ga$-space \cite{Sega} or $E_\iy$-space \cite{May2}, we could define the analogue of strong H-principal $\Z_2$-bundle $(P,p)$ without assuming $P$ is trivializable.
\label{ss3rem1}	
\end{rem}

\begin{prop} Suppose $f:(X,\mu_X)\ra(Y,\mu_Y)$ is a homotopy-theoretic group completion of H-spaces. Then
\begin{itemize}
\setlength{\itemsep}{0pt}
\setlength{\parsep}{0pt}
\item[{\bf(a)}] If\/ $P$ is a weak H-principal\/ $\Z_2$-bundle on $X,$ there exists a weak H-principal\/ $\Z_2$-bundle $Q$ on $Y,$ unique up to isomorphism, with\/ $P\cong f^*(Q)$. 
\item[{\bf(b)}] If\/ $(P,p)$ is a strong H-principal\/ $\Z_2$-bundle on $X,$ there exists a strong H-principal\/ $\Z_2$-bundle $(Q,q)$ on $Y,$ unique up to canonical isomorphism, with an isomorphism\/ $\io:(P,p)\ra f^*(Q,q)$. 
\end{itemize}
\label{ss3prop1}	
\end{prop}

\begin{proof} For (a), let $g:X\ra B\Z_2$ be the H-map (natural up to homotopy) corresponding to $P\ra X$. As $B\Z_2$ is a grouplike H-space, Caruso et al.\ \cite[Prop.~1.2]{CCMT} implies that there is a weak H-map $h:Y\ra B\Z_2$, natural up to weak homotopy, such that $g$ is weakly homotopic to $h\ci f$. Here continuous maps $a,b:S\ra T$ are {\it weakly homotopic\/} if $a\ci c\simeq b\ci c$ whenever $c:R\ra S$ is continuous with $R$ a finite CW-complex. For {\it weak H-maps\/} we replace the homotopy $f\ci\mu_X\simeq \mu_Y\ci(f\t f)$ in Definition \ref{ss3def1} by weak homotopy.

Let $Q\ra Y$ be the principal $\Z_2$-bundle corresponding to $h:Y\ra B\Z_2$. Then $h$ a weak H-map means that $Q\bt_{\Z_2}Q,\mu_Y^*(Q)$ become isomorphic on pull-back by $c:R\ra Y\t Y$ for any finite CW-complex $R$ and map $c$. Also $g$ weakly homotopic to $h\ci f$ means that $P,f^*(Q)$ become isomorphic on pull-back by $c:R\ra X$ for any finite CW-complex $R$ and $c$. And $h$ natural up to weak homotopy means that if $Q'$ is an alternative choice for $Q$ then $Q,Q'$ become isomorphic on pull-back by $c:R\ra Y$ for any finite CW-complex $R$ and~$c$. 

Now principal $\Z_2$-bundles $Q\ra Y$ are determined up to isomorphism by their pullbacks along all maps $c:\cS^1\ra Y$, where $\cS^1$ is a finite CW-complex. Hence $Q\bt_{\Z_2}Q\cong\mu_Y^*(Q)$, so $Q\ra Y$ is a weak H-principal bundle, and $P\cong f^*(Q)$, and any alternative choice $Q'$ has $Q\cong Q'$, so $Q$ is unique up to isomorphism. 
 
For (b), consider the map $\pi_0(\pi):\pi_0(P)\ra\pi_0(X)$. As $P\ra X$ is trivializable, there are two connected components of $P$ for each connected component of $X$, so $\pi_0(\pi)$ is a 2:1 map, which we regard as a principal $\Z_2$-bundle. We may write 
\begin{equation*}
\pi_0(P)=\bigl\{(\al,\om):\al\in\pi_0(X),\;\> \om\in\Ga(P\vert_\al)\bigr\},\qquad \pi_0(\pi):(\al,\om)\longmapsto\al,	
\end{equation*}
where $\Ga(\cdots)$ means global sections. Define a binary operation $\star$ on $\pi_0(P)$ by $(\al_1,\om_1)\star(\al_2,\om_2)=\bigl(\al_1+\al_2,p(\om_1\bt\om_2)\bigr)$. Then \eq{ss3eq1} implies that $\star$ is associative. It has identity $(0,\om_0)$, where $\om_0$ is unique with $(0,\om_0)\star(0,\om_0)=(0,\om_0)$. This makes $\pi_0(P)$ into an associative (though not necessarily commutative) monoid, and $\pi_0(\pi):\pi_0(P)\ra\pi_0(X)$ a monoid morphism.

As $f:(X,\mu_X)\ra(Y,\mu_Y)$ is a homotopy-theoretic group completion, $\pi_0(f):\pi_0(X)\ra\pi_0(Y)$ is the group completion of $\pi_0(X)$. Explicitly we may write
\e
\begin{split}
&\pi_0(Y)\cong\bigl\{(\al,\be):\al,\be\in\pi_0(X)\bigr\}/\!\sim,\quad\text{for $\sim$ the equivalence relation}\\
&(\al_1,\be_1)\sim(\al_2,\be_2)\;\>\text{if}\;\>\exists\ga,\de\in\pi_0(X),\;\al_1+\ga=\al_2+\de,\;\be_1+\ga=\be_2+\de.
\end{split}
\label{ss3eq2}
\e
Write $\rho:\pi_0(P)\ra\pi_0(P)^+$ for the group completion of $\pi_0(P)$. Then we have a commutative diagram of monoids:
\begin{equation*}
\xymatrix@C=150pt@R=15pt{ *+[r]{\pi_0(P)} \ar[d]^{\pi_0(\pi)} \ar[r]_\rho & *+[l]{\pi_0(P)^+} \ar[d]_\si \\
*+[r]{\pi_0(X)} \ar[r]^{\pi_0(f)} & *+[l]{\pi_0(Y).}}	
\end{equation*}
Here $\si$ exists by the universal property of the group completion $\pi_0(P)^+$. Using the explicit expression for $\pi_0(P)^+$ analogous to \eq{ss3eq2}, written in terms of pairs $(\al,\om)$, we see that the group completion process is compatible with the $\Z_2$-fibration structure, so $\si$ is also a principal $\Z_2$-bundle.

Define $Q=Y\t_{\Pi,\pi_0(Y),\si}\pi_0(P)^+$, where $\Pi:Y\ra\pi_0(Y)$ maps a point to its connected component. Then $Q$ is a trivializable principal $\Z_2$-bundle. Define an isomorphism $q:Q\bt_{\Z_2}Q\ra\mu_Y^*(Q)$ of principal $\Z_2$-bundles over $Y\t Y$ by
\begin{equation*}
q\bigl((y_1,p_1^+),(y_2,p_2^+)\bigr)=\bigl(\mu_Y(y_1,y_2),p_1^+*p_2^+\bigr),
\end{equation*}
where $y_i\in Y$, $p_i^+\in \pi_0(P)^+$ with $\si(p_i^+)=\Pi(y_i)$, and $*$ is the group operation in $\pi_0(P)^+$. Then associativity of $*$ implies that $q$ satisfies \eq{ss3eq1}, so $(Q,q)$ is a strong H-principal $\Z_2$-bundle. Define $\io:P\ra f^*(Q)$ by $\io:p\mapsto\bigl(\pi(p),\rho\ci\Pi(p)\bigr)$. Then $\rho$ a monoid morphism implies that $\io\ci p=f^*(q)\ci(\io\bt\io)$. Hence $\io:(P,p)\ra f^*(Q,q)$ is an isomorphism of strong H-principal $\Z_2$-bundles on $X$, giving existence in~(b).

For uniqueness, suppose $(Q',q')$ is a strong H-principal $\Z_2$-bundle $(Q,q)$ on $Y$  with an isomorphism $\io':(P,p)\ra f^*(Q',q')$. Then $Q'$ is trivializable, so $\pi_0(\pi):\pi_0(Q')\ra\pi_0(Y)$ is a principal $\Z_2$-bundle. Equation \eq{ss3eq1} implies that $\pi_0(q'):\pi_0(Q')\t\pi_0(Q')\ra\pi_0(Q')$ is associative, and as $\pi_0(Y)$ is a group, it is easy to see that $\pi_0(Q')$ is a group, and $\pi_0(\io'),\pi_0(\pi')$ are group morphisms. Consider the commutative diagram:
\e
\begin{gathered}
\xymatrix@C=75pt@R=6pt{ *+[r]{\pi_0(P)} \ar[dr]_{\pi_0(\io')} \ar[dd]^{\pi_0(\pi)} \ar[rr]_\rho && *+[l]{\pi_0(P)^+} \ar@{.>}[dl]_\tau^\cong \ar[dd]_\si \\
& \pi_0(Q')  \ar[dr]^{\pi_0(\pi')} \\
*+[r]{\pi_0(X)} \ar[rr]^{\pi_0(f)} && *+[l]{\pi_0(Y).}}	
\end{gathered}
\label{ss3eq3}
\e

Here the group morphism $\tau$ exists by the universal property of $\pi_0(P)^+$. Since $\si,\pi_0(\pi')$ are both principal $\Z_2$-bundles and $\rho,\pi_0(\io')$ are injective on $\Z_2=\Ker\pi_0(\pi)$, we see that $\tau$ is an isomorphism. Since $Q'\cong Y\t_{\Pi,\pi_0(Y),\si}\pi_0(Q')$ and $Q=Y\t_{\Pi,\pi_0(Y),\si}\pi_0(P)^+$, it follows that $\tau$ lifts to an isomorphism $\up:Q\ra Q'$. This satisfies $\up\ci q=q'\ci(\up\bt\up)$ as $\tau$ is a group morphism, so $\up$ is an isomorphism of strong H-principal $\Z_2$-bundles, and $q'=f^*(\up)\ci q$ since $\pi_0(\io')=\tau\ci\rho$. Also $\up$ is unique under these conditions, as $\tau$ is unique in \eq{ss3eq3}. Thus $Q$ is unique up to canonical isomorphism, proving~(b).
\end{proof}

\subsection{\texorpdfstring{Background from Joyce--Tanaka--Upmeier \cite{JTU}}{Background from Joyce--Tanaka--Upmeier [32]}}
\label{ss32}

The next three definitions come from~\cite[\S 2.3--\S 2.4]{JTU}.

\begin{dfn} We summarize some well known material which can be found in Milnor and Stasheff \cite{MiSt}, May \cite[\S\S 16.5, 23, 24]{May3}, and Husem\"oller et al.\ \cite[Part II]{HJJS}. Let $G$ be a topological group. A {\it classifying space\/} for $G$ is a topological space $BG$ and a principal $G$-bundle $\pi:EG\ra BG$ such that $EG$ is contractible. Classifying spaces exist for any $G$, and are unique up to homotopy equivalence. Classifying spaces have the property that if $X$ is a paracompact topological space and $P\ra X$ is a principal $G$-bundle, then there exists a continuous map $f_P:X\ra BG$, unique up to homotopy, and an isomorphism $P\cong f_P^*(EG)$ of principal $G$-bundles on~$X$. 
 
Write $B\U$ for the (homotopy) direct limit $B\U=\varinjlim_{n\ra\iy}B\U(n)$. Then $B\U\t\Z$ {\it is the classifying space for complex K-theory}. That is, as in May \cite[p.~204-5]{May3}, for compact topological spaces $X$ there is a natural bijection
\e
K^0(X)\cong [X,B\U\t\Z]=\pi_0\bigl(\Map_{C^0}(X,B\U\t\Z)\bigr).
\label{ss3eq4}
\e

Define $\Pi_n:B\U(n)\ra B\U\t\Z$ for $n\ge 0$ to map $B\U(n)\ra B\U$ from the direct limit $B\U=\varinjlim_{n\ra\iy}B\U(n)$, and to map $B\U(n)\ra n\in\Z$. Then if a principal $\U(n)$-bundle $P\ra X$ corresponds to $f_P:X\ra B\U(n)$, its K-theory class $\lb P\rb\in K^0(X)$ corresponds to~$\Pi_n\ci f_P:X\ra B\U\t\Z$.

The inclusion $\U(n)\t\U(n')\ra\U(n+n')$ mapping $(A,B)\mapsto\bigl(\begin{smallmatrix} A & 0 \\ 0 & B \end{smallmatrix}\bigr)$ induces a morphism $\mu_{n,n'}:B\U(n)\t B\U(n')\ra B\U(n+n')$. We interpret this in terms of direct sums: if $P\ra X$, $Q\ra X$ are principal $\U(n)$, $\U(n')$-bundles corresponding to $f_P:X\ra B\U(n)$, $f_Q:X\ra B\U(n')$ then $\mu_{n,n'}\ci(f_P,f_Q):X\ra B\U(n+n')$ corresponds to the principal $\U(n+n')$-bundle~$P\op Q\ra X$.

Let $\mu=\varinjlim_{n,n'\ra\iy}\mu_{n,n'}:B\U\t B\U\ra B\U$. Then $\mu$ is homotopy commutative and associative, and makes $B\U$ into an H-space. We can also define $\mu':(B\U\t\Z)\t(B\U\t\Z)\ra B\U\t\Z$ as the product of $\mu:B\U\t B\U\ra B\U$ and $+:\Z\t\Z\ra\Z$. Then $\mu'$ induces the operation of addition on $K^0(X)\cong[X,B\U\t\Z]$, from direct sum of vector bundles.
\label{ss3def4}	
\end{dfn}

\begin{dfn}
\label{ss3def5}
Let $X$ be a compact, connected manifold, and use the notation of Definition \ref{ss3def4}. Write $\cC=\Map_{C^0}(X,B\U\t\Z)$ for the topological space of continuous maps $X\ra B\U\t\Z$, with the compact-open topology. Equation \eq{ss3eq4} identifies the set $\pi_0(\cC)$ of path-connected components of $\cC$ with $K^0(X)$. Write $\cC_\al$ for the connected component of $\cC$ corresponding to $\al\in K^0(X)$ under \eq{ss3eq4}, so that $\cC=\coprod_{\al\in K^0(X)}\cC_\al$.

Define $\Psi:\cC\t\cC\ra\cC$ by $\Psi:(f,g)\mapsto\mu'\ci(f,g)$. Then $\Psi$ is homotopy commutative and associative, as $\mu'$ is, and makes $\cC$ into an H-space. Write $\Psi_{\al,\be}=\Psi\vert_{\cC_\al\t\cC_\be}:\cC_\al\t\cC_\be\ra\cC_{\al+\be}$ for $\al,\be\in K^0(X)$. By \cite[Prop.~2.24]{JTU}, $\cC_\al$ is homotopy equivalent to $\cC_0$ for all $\al\in K^0(X)$, and there is a canonical isomorphism~$\pi_1(\cC_\al)\cong K^1(X)$.

Now let $P\ra X$ be a principal $\U(n)$-bundle. Then we have a topological stack $\B_P=[\A_P/\G_P]$ as in \S\ref{ss11}. Define a topological space $\B_P^\cla=(\A_P\t E\G_P)/\G_P$, where $E\G_P\ra B\G_P$ is a classifying space for $\G_P$, and let $\pi^\cla:\B_P^\cla\ra\B_P$ be the obvious projection in $\Ho(\TopSta)$. Then $\pi^\cla$ is a fibration with contractible fibre $E\G_P$. Noohi \cite{Nooh2} develops a homotopy theory for topological stacks, and in Noohi's language $\B_P^\cla$ is a {\it classifying space\/} for $\B_P$, and $\pi^\cla$ is a homotopy equivalence. Classifying spaces are a functor $(-)^\cla:\Ho(\TopSta_{\bf hp})\ra\Topho$ from the category of `hoparacompact' topological stacks (which include $\B_P$) to topological spaces up to homotopy.

There is a universal principal $\U(n)$-bundle $U_P=(P\t\A_P)/\G_P\ra X\t\B_P$, so $(\id_X\t\pi^\cla)^*(U_P)\ra X\t\B_P^\cla$ is a principal $\U(n)$-bundle over a paracompact topological space, and corresponds to some $f_P:X\t\B_P^\cla\ra B\U(n)$. Write $\Si_P:\B_P^\cla\ra\Map_{C^0}(X,B\U(n))$ for the corresponding map. Then $f_P,\Si_P$ are unique up to homotopy.

Connected components of $\Map_{C^0}(X,B\U(n))$ correspond to isomorphism classes $[Q]$ of principal $\U(n)$-bundles $Q\ra X$. Write $\Map_{C^0}(X,B\U(n))_{[P]}$ for the component corresponding to $[P]$. Using the arguments of Donaldson--Kronheimer \cite[Prop.~5.1.4]{DoKr} and Atiyah--Bott \cite[Prop.~2.4]{AtBo}, we see that $\Si_P:\B_P^\cla\ra\Map_{C^0}(X,B\U(n))_{[P]}$
is a homotopy equivalence. Define $\Si_P^\cC:\B_P^\cla\ra\cC$ by $\Si_P^\cC:b\mapsto \Pi_n\ci\Si_P(b)$. Then $\Si_P^\cC$ maps $\B_P^\cla\ra\cC_\al$, where~$\al=\lb P\rb\in K^0(X)$. 

Suppose $Q\ra X$ is a principal $\U(n')$-bundle with $\lb Q\rb=\be\in K^0(X)$, so $P\op Q\ra X$ is a principal $U(n+n')$-bundle with $\lb P\op Q\rb=\al+\be$. As in \cite[Ex.~2.11]{JTU} we define a morphism of topological stacks $\Psi_{P,Q}:\B_P\t\B_Q\ra\B_{P\op Q}$ mapping $\Psi_{P,Q}:\bigl([\nabla_P],[\nabla_Q]\bigr)\mapsto[\nabla_P\op\nabla_Q]$. These $\Psi_{P,Q}$ are commutative and associative. Applying the classifying space functor \cite{Nooh2} gives $\Psi_{P,Q}^\cla:\B_P^\cla\t\B_Q^\cla\ra\B_{P\op Q}^\cla$, which is natural, commutative, and associative, up to homotopy.

Then the following diagram commutes up to homotopy:
\e
\begin{gathered}
\xymatrix@!0@C=145pt@R=40pt{
*+[r]{\B_P^\cla\t\B_Q^\cla} \ar@<1.5ex>@/^.7pc/[rr]^{\Si_P^\cC\t\Si_Q^\cC} \ar@<-1ex>@{}[dr]^\simeq \ar[d]^{\Psi_{P,Q}^\cla} \ar[r]_(0.44){\Si_P\t\Si_Q} & {\begin{subarray}{l}\ts \Map_{C^0}(X,B\U(n))_{[P]} \t \\ \ts \Map_{C^0}(X,B\U(n'))_{[Q]}\end{subarray}} \ar@<-1ex>@{}[dr]^\simeq\ar[r]_(0.55){(\Pi_n\ci)\t(\Pi_{n'}\ci)} \ar[d]^(0.55){\mu_{n,n'}\ci} & *+[l]{\cC_\al\t\cC_\be} \ar[d]_{\Psi_{\al,\be}=\mu'\ci}
\\
*+[r]{\B^\cla_{P\op Q}} \ar@<-1ex>@/_.5pc/[rr]_{\Si_{P\op Q}^\cC} \ar[r]^(0.33){\Si_{P\op Q}} & \Map_{C^0}(X,B\U(n\!+\!n'))_{[P\op Q]} \ar[r]^(0.6){\Pi_{n+n'}\ci} & *+[l]{\cC_{\al+\be}.}
}
\end{gathered}
\label{ss3eq5}
\e

We now have a commutative diagram
\ea
\begin{gathered}
\xymatrix@!0@C=250pt@R=45pt{
*+[r]{\coprod\limits_{\begin{subarray}{l}\text{iso. classes $[P]$ of }\\ \text{principal $\U(n)$-bundles}\\ \text{$P\ra X$, $n\ge 0$}\end{subarray}\!\!\!\!\!\!\!\!\!\!\!\!\!\!\!\!\!\!\!\!\!\!\!\!\!\!}\B_P^\cla\,\,\,\,\,\,\,\,\,\,} \ar[r]^(0.43){\coprod_{[P]}\Si_P}_(0.43)\simeq \ar[dr]_(0.43){\coprod_{[P]}\Si_P^\cC} & *+[l]{\begin{subarray}{l} \ts \coprod\limits_{[P],\;n\ge 0\!\!\!\!\!\!\!}\Map_{C^0}(X,B\U(n))_{[P]}\\ \ts =\coprod\limits_{\smash{n\ge 0}}\Map_{C^0}(X,B\U(n))\end{subarray}\;\>} \ar[d]_(0.55){\coprod_{n\ge 0}\Pi_n\ci} \\
& *+[l]{\cC=\Map_{C^0}(X,B\U\t\Z),}} 
\end{gathered}
\label{ss3eq6}
\ea
with the top row a homotopy equivalence. The three objects are H-spaces, with H-space multiplications $\coprod_{[P],[Q]}\Psi_{P,Q}^\cla$, $\coprod_{n,n'\ge 0}\mu_{n,n'}\ci$, and $\mu'\ci$, and the morphisms in \eq{ss3eq6} are H-space morphisms by \eq{ss3eq5}. Also $\cC$ is a grouplike H-space, as $\pi_0(\cC)=K^0(X)$ is a group.

Here on the left hand side of \eq{ss3eq6}, we sum over isomorphism classes $[P]$ of principal $\U(n)$-bundles $P\ra X$, and we pick one representative $P$ in each isomorphism class $[P]$ to give the corresponding $\B_P^\cla$. If $P'\in[P]$ is an alternative choice then $\B_P,\B_{P'}$ are canonically isomorphic.
\end{dfn}

\begin{prop} The H-space morphisms\/ $\coprod_{[P]}\Si_P^\cC$ and\/ $\coprod_{n\ge 0}\Pi_n\ci$ in\/ \eq{ss3eq6} are homotopy-theoretic group completions, in the sense of Definition\/~{\rm\ref{ss3def2}}.
\label{ss3prop2}	
\end{prop}

\begin{proof} We use the $\Ga$-spaces of Segal \cite{Sega}. Given a $\Ga$-space $S$, Segal \cite[\S 4]{Sega} constructs a $\Ga$-space morphism $S\ra\Om BS$ which is a homotopy-theoretic group completion at the level of H-spaces, where $BS$ is the geometric realization of $S$ as a simplicial complex, and $\Om BS$ its loop space. If $X$ is compact and $S$ is a $\Ga$-space then $\Map_{C^0}(X,S)$ is a $\Ga$-space, by applying the $\Ga$-space operations pointwise in $X$. Now $\coprod_{n\ge 0}B\U(n)$ is a $\Ga$-space. Let $X$ be compact and connected. Then $\Map_{C^0}(X,\coprod_{n\ge 0}B\U(n))=\coprod_{n\ge 0}\Map_{C^0}(X,B\U(n))$ is a $\Ga$-space, with homotopy-theoretic group completion
\e
\ts\coprod_{n\ge 0}\Map_{C^0}(X,B\U(n))\longra \Om B\bigl(\Map_{C^0}(X,\coprod_{n\ge 0}B\U(n)\bigr).
\label{ss3eq7}
\e

We have homotopy equivalences
\e
\begin{split}
&\ts\Om B\bigl(\Map_{C^0}(X,\coprod_{n\ge 0}B\U(n)\bigr)\simeq
\Om\Map_{C^0}\bigl(X,B(\coprod_{n\ge 0}B\U(n))\bigr)\\
&\ts\qquad\simeq
\Map_{C^0}\bigl(X,\Om B(\coprod_{n\ge 0}B\U(n))\bigr)\simeq
\Map_{C^0}(X,B\U\t\Z),
\end{split}
\label{ss3eq8}
\e
where the first two are elementary consequences of the definitions, and the third $\Om B(\coprod_{n\ge 0}B\U(n))\simeq B\U\t\Z$ follows from 
Segal \cite[top p.~305 for $R=\C$]{Sega}. Combining \eq{ss3eq7}--\eq{ss3eq8} shows $\coprod_{n\ge 0}\Pi_n\ci$ in \eq{ss3eq6} is a homotopy-theoretic group completion. Then $\coprod_{[P]}\Si_P^\cC$ is too, as $\coprod_{[P]}\Si_P$ is a homotopy equivalence.
\end{proof}

\begin{dfn}
\label{ss3def6}
Continue in the situation of Definition \ref{ss3def5}, and let $E_\bu$ be a real elliptic operator on $X$. In \cite[Def.~2.22]{JTU} we construct a principal $\Z_2$-bundle $O^{E_\bu}\ra\cC$ with $O^{E_\bu}_\al:=O^{E_\bu}\vert_{\cC_\al}\ra\cC_\al$ for $\al\in K^0(X)$, unique up to canonical isomorphism, and isomorphisms of principal $\Z_2$-bundles on $\cC\t\cC$, $\cC_\al\t\cC_\be$:
\e
\begin{split}
\psi&:O^{E_\bu}\bt_{\Z_2}O^{E_\bu}\longra\Psi^*(O^{E_\bu}),\\
\psi_{\al,\be}:=\psi\vert_{\cC_\al\t\cC_\be}&:O_\al^{E_\bu}\bt_{\Z_2}O_\be^{E_\bu}\longra\Psi_{\al,\be}^*(O_{\al+\be}^{E_\bu}),
\end{split}
\label{ss3eq9}
\e
canonical if $O^{E_\bu}$ is trivializable, with the following properties: firstly, we have
\e
(\Psi\t\id_\cC)^*(\psi\ci(\psi\bt\id_{O^{E_\bu}})
\simeq(\id_\cC\t\Psi)^*(\psi)\ci(\id_{O^{E_\bu}}\bt\psi).
\label{ss3eq10}
\e

Here two sides of \eq{ss3eq10} are isomorphisms of $\Z_2$-bundles on $\cC\t\cC\t\cC$ from $O^{E_\bu}\bt_{\Z_2}O^{E_\bu}\bt_{\Z_2}O^{E_\bu}$ to $(\Psi\ci(\Psi\t\id_\cC))^*(O^{E_\bu})$ and to $(\Psi\ci(\id_\cC\t\Psi))^*(O^{E_\bu})$. By homotopy associativity there is a homotopy $\Psi\ci(\Psi\t\id_\cC)\simeq \Psi\ci(\id_\cC\t\Psi)$. As in Remark \ref{ss1rem5}, we can identify the two pullback $\Z_2$-bundles by parallel translation along this homotopy, and then \eq{ss3eq10} makes sense.

[{\bf Aside:} as in Remark \ref{ss1rem5}, if $O^{E_\bu}$ is nontrivial then the identification above may depend on the choice of homotopy $h:\Psi\ci(\Psi\t\id_\cC)\,{\buildrel\simeq\over\Longra}\,\Psi\ci(\id_\cC\t\Psi)$. To choose $h$ correctly, one should enhance the H-spaces to $\Ga$-spaces or $E_\iy$-spaces, which we do not do. However, we will only use \eq{ss3eq10} in the proof of Theorem \ref{ss1thm2}(c) when $O^{E_\bu}$ is assumed trivial, so this problem will not arise.]

Secondly, let $P\ra X$ be a principal $\U(n)$-bundle, with $\lb P\rb=\al\in K^0(X)$. As in \S\ref{ss11} we have a $\Z_2$-bundle $O_P^{E_\bu}\ra \B_P$. There is an isomorphism
\e
\si_P^\cC:(\pi^\cla)^*(O_P^{E_\bu})\longra(\Si_P^\cC)^*(O_\al^{E_\bu}),
\label{ss3eq11}
\e
which is natural if $O_\al^{E_\bu}$ is trivializable. (In fact if $2n>\dim X$ then $\pi_k(\Si_P^\cC):\pi_k(\B_P^\cla)\ra\pi_k(\cC_\al)$ is an isomorphism for $k=0,1$, which implies that \eq{ss3eq11} determines $O^{E_\bu}_\al\ra\cC_\al$ uniquely up to canonical isomorphism. This is how we begin constructing the~$O^{E_\bu}_\al$.)

Thirdly, let $Q\ra X$ be a principal $\U(n')$-bundle with $\lb Q\rb=\be\in K^0(X)$. Let $\nabla_P,\nabla_Q$ be connections on $P,Q$, and $\nabla_{P\op Q}=\nabla_P\op\nabla_Q$ the corresponding connection on $P\op Q$. Then there is a natural isomorphism
\begin{equation*}
\Ad(P\op Q)\cong \Ad(P)\op\Ad(Q)\op \bigl((P\t_{\U(n)}\ov\C^n)\ot_\C (Q\t_{\U(n')}\C^{n'})\bigr),
\end{equation*}
so the twisted operators in \eq{ss1eq2} satisfy
\e
D^{\nabla_{\Ad(P\op Q)}}\cong D^{\nabla_{\Ad(Q)}}\op D^{\nabla_{\Ad(P)}} \op D^{\nabla_{\smash{(P\t_{\U(n)}\ov\C^n)\ot_\C (Q\t_{\U(n')}\C^{n'})}}}.
\label{ss3eq12}
\e

Now $D^{\smash{\nabla_{(P\t_{\U(n)}\ov\C^n)\ot_\C (Q\t_{\U(n')}\C^{n'})}}}$ is a complex linear elliptic operator, so its real determinant has a natural orientation, since complex vector spaces have natural orientations considered as real vector spaces. Thus, taking orientations of real determinant line bundles in \eq{ss3eq12}, as in \cite[Ex.~2.11]{JTU} we obtain a natural isomorphism of principal $\Z_2$-bundles on~$\B_P\t\B_Q$:
\e
\smash{\psi_{P,Q}:O_P^{E_\bu}\bt_{\Z_2}O_Q^{E_\bu}\longra\Psi_{P,Q}^*(O_{P\op Q}^{E_\bu}).}
\label{ss3eq13}
\e
The $\psi_{P,Q}$ are associative, and commutative up to signs in \cite[(2.11)]{JTU}. Then we have a commutative diagram of $\Z_2$-bundles on $\B_P^\cla\t\B_Q^\cla$, parallel to~\eq{ss3eq5}:
\e
\begin{gathered}
\!\!\!\!\!\!\!\!\!\xymatrix@!0@C=148pt@R=45pt{
*+[r]{(\pi^\cla)^*(O_P^{E_\bu})\bt_{\Z_2}(\pi^\cla)^*(O_Q^{E_\bu})} \ar[rr]_{\si_P^\cC\bt\si_Q^\cC} \ar[d]^(0.4){(\pi^\cla)^*(\psi_{P,Q})} && *+[l]{(\Si_P^\cC)^*(O^{E_\bu}_\al)\bt_{\Z_2}(\Si_Q^\cC)^*(O^{E_\bu}_\be)} \ar@{.>}[d]_(0.4){(\Si_P^\cC\t\Si_Q^\cC)^*(\psi_{\al,\be})}
\\
*+[r]{\!\!\quad\begin{subarray}{l}\ts (\pi^\cla)^*(\Psi_{P,Q}^*(O_{P\op Q}^{E_\bu}))\!\simeq \\ \ts (\Psi_{P,Q}^\cla)^*\!\ci\!(\pi^\cla)^*(O_{P\op Q}^{E_\bu}) \end{subarray}} \ar[r]^(0.8){\begin{subarray}{l}(\Psi_{P,Q}^\cla)^* \\ (\si_{P\op Q}^\cC)\end{subarray}} & *+[r]{\begin{subarray}{l}\ts \qquad(\Psi_{P,Q}^\cla)^*\ci \\ \ts (\Si_{P\op Q}^\cC)^*(O^{E_\bu}_{\al+\be}) \end{subarray}} \ar[r]^\simeq_{\eq{ss3eq5}}
& *+[l]{\begin{subarray}{l}\ts (\Si_P^\cC\t\Si_Q^\cC)^*\ci \\ \ts \Psi_{\al,\be}^*(O^{E_\bu}_{\al+\be}).\end{subarray}\quad\!\!}
}\!\!\!\!\!\!\!\!\!
\end{gathered}
\label{ss3eq14}
\e
Here the two `$\simeq$' are isomorphisms relating pullbacks of the same bundle by homotopic morphisms, and are interpreted as for \eq{ss1eq9} and~\eq{ss3eq10}.

Fourthly, by \cite[Prop.~2.24(b)]{JTU}, $O^{E_\bu}$ is trivializable if and only if $O^{E_\bu}_\al\ra\cC_\al$ is trivializable for all $\al\in K^0(X)$, if and only if $O^{E_\bu}_0\ra\cC_0$ is trivializable.

In the obvious way, we say that $\cC,\cC_\al$ are {\it orientable\/} if $O^{E_\bu},O_\al^{E_\bu}$ are trivializable, and an {\it orientation\/} $o^{E_\bu}_\al$ for $\cC_\al$ is a trivialization~$O_\al^{E_\bu}\cong\cC_\al\t\Z_2$.
\end{dfn}         

Equation \eq{ss3eq11} shows that if $\cC_\al$ is orientable then $\B_P$ is orientable for any principal $\U(n)$-bundle $P\ra X$ with $\lb P\rb=\al$ in $K^0(X)$, and an orientation for $\cC_\al$ induces orientations on $\B_P$ for all such $P$. Hence, if we can construct orientations on $\cC_\al$ for all $\al\in K^0(X)$, we obtain orientations on $\B_P$ for all $\U(n)$-bundles $P\ra X$, for all~$n\ge 0$.

Comparing Definition \ref{ss3def3} with \eq{ss3eq9}--\eq{ss3eq10} yields:

\begin{lem}
\label{ss3lem1}
In Definition\/ {\rm\ref{ss3def6},} $O^{E_\bu}\ra\cC$ is a weak H-principal\/ $\Z_2$-bundle. If\/ $O^{E_\bu}$ is trivializable then $(O^{E_\bu},\psi)$ is a strong H-principal\/ $\Z_2$-bundle.
\end{lem}

\subsection{Background on Algebraic Geometry, proof of (a)}
\label{ss33}

We now give a bit more detail on the algebro-geometric side of Theorem~\ref{ss1thm2}(a).

Let $X$ be a smooth projective $\C$-scheme. Then we can consider coherent sheaves on $X$, as in Hartshorne \cite[\S II.5]{Hart} and Huybrechts and Lehn \cite{HuLe}. We write $\coh(X)$ for the abelian category of coherent sheaves on $X$, and $D^b\coh(X)$ for its bounded derived category, as a triangulated category with a full subcategory $\coh(X)\subset D^b\coh(X)$. See Gelfand and Manin \cite{GeMa} on triangulated categories, and Huybrechts \cite{Huyb} on~$D^b\coh(X)$.

Higher and derived $\C$-stacks form $\iy$-categories $\HSta_\C,\DSta_\C$, as in To\"en and Vezzosi \cite{Toen1,ToVe1,ToVe2} and Simpson \cite{Simp2}. We write $\Ho(\HSta_\C),\Ho(\DSta_\C)$ for their homotopy categories. Write $\bs\M$ for the derived moduli stack of objects in $D^b\coh(X)$, as a derived $\C$-stack, and $\M=t_0(\bs\M)$ for its classical truncation, as a higher $\C$-stack. These exist by \cite{ToVa}. 

Actually this is slightly misleading: as in \cite[\S 1]{ToVa}, just the triangulated category structure on $D^b\coh(X)$ is not sufficient to define the moduli stacks. Rather, $\bs\M,\M$ are moduli stacks of objects in a dg-category $\text{dgmod}^{\rm ft}\text{-}\O_X$, and $D^b\coh(X)\cong\Ho(\text{dgmod}^{\rm ft}\text{-}\O_X)$, so $\C$-points of $\bs\M,\M$ correspond to isomorphism classes of objects in $D^b\coh(X)$. But we will ignore this point.

We may write $\M$ as an internal mapping stack
\begin{equation*}
\smash{\M=\Map_{\HSta_\C}(X,\Perf_\C),}
\end{equation*}
where $\Perf_\C$ is a higher stack which classifies perfect complexes, as in To\"en and Vezzosi \cite[Def.~1.3.7.5]{ToVe2}, which is just $\M$ for $X=\Spec\C$ the point. There is a tautological morphism $u:X\t\M\ra\Perf_\C$ in $\Ho(\HSta_\C)$. On $\Perf_\C$ there is a tautological perfect complex $\cU_0^\bu$. Write $\cU^\bu=u^*(\cU_0^\bu)$, a perfect complex on $X\t\M$ which we call the {\it universal complex}. It has the property that $\cU^\bu\vert_{X\t\{[F^\bu]\}}\cong F^\bu$ for any object $F^\bu$ in~$D^b\coh(X)$.

There is a morphism $\Phi_0:\Perf_\C\t\Perf_\C\ra\Perf_\C$ mapping $([F^\bu],[G^\bu])\mapsto[F^\bu\op G^\bu]$ on $\C$-points. Write $\Phi:\M\t\M\ra\M$ for the composition
\begin{equation*}
\smash{\xymatrix@C=14pt{ \M\!\t\!\M \ar[r]^(0.25)\cong & \Map_{\HSta_\C}(X,\Perf_\C\!\t\!\Perf_\C) \ar[rr]^(0.55){\Phi_0\ci} && \Map_{\HSta_\C}(X,\Perf_\C) \ar@{=}[r] & \M. }}
\end{equation*}
Then the following diagram commutes:
\e
\begin{gathered}
\xymatrix@C=180pt@R=15pt{ *+[r]{X\t\M\t\M} \ar[r]_(0.48){[u\ci(\Pi_1,\Pi_2)]\t[u\ci(\Pi_1,\Pi_3)]} \ar[d]^{\id_X\t\Phi } & *+[l]{\Perf_\C\t\Perf_\C} \ar[d]_{\Phi_0} \\
*+[r]{X\t\M} \ar[r]^(0.48){u} & *+[l]{\Perf_\C.\!}  }
\end{gathered}
\label{ss3eq15}
\e
Also $\Phi$ is commutative and associative in $\Ho(\HSta_\C),$ with identity $[0]\in\M$. As $\Phi_0$ maps $([F^\bu],[G^\bu])\mapsto[F^\bu\op G^\bu]$, there is an isomorphism on~$\Perf_\C\t\Perf_\C$:
\e
\Phi_0^*(\cU_0^\bu)\cong\Pi_1^*(\cU_0^\bu)\op \Pi_2^*(\cU_0^\bu).
\label{ss3eq16}
\e
Pulling $\cU_0^\bu$ back round the two routes around \eq{ss3eq15} and using $\cU^\bu=u^*(\cU_0^\bu)$ and \eq{ss3eq16} gives an isomorphism on~$X\t\M\t\M$:
\e
(\id_X\t\Phi)^*(\cU^\bu)\cong (\Pi_1,\Pi_2)^*(\cU^\bu)\op (\Pi_1,\Pi_3)^*(\cU^\bu).	
\label{ss3eq17}
\e

As in \S\ref{ss14}(v), Simpson \cite{Simp1} and Blanc \cite[\S 3.1]{Blan} define a {\it topological realization functor\/} $(-)^\top:\Ho(\HSta_\C)\ra\Topho$, the analogue of the classifying space functor $(-)^\cla:\Ho(\TopSta)\ra\Topho$ discussed in \S\ref{ss32}. Applying this to $\M,\Phi$ gives a topological space $\M^\top$ and continuous map $\Phi^\top:\M^\top\t\M^\top\ra\M^\top$, which up to homotopy is commutative and associative with identity. That is, $(\M^\top,[0]^\top,\Phi^\top)$ is an H-space.

It follows from Blanc \cite[Th.s 4.7 \& 4.21]{Blan} that there is a homotopy equivalence $\Perf_\C^\top\simeq B\U\t\Z$, which lifts to an equivalence of symmetric spectra, and thus of H-spaces. As topological realizations matter only up to homotopy equivalence, we may take $\Perf_\C^\top=B\U\t\Z$, and $\Phi_0^\top=\mu'$. Then applying $(-)^\top$ to \eq{ss3eq15} yields a homotopy commutative diagram
\e
\begin{gathered}
\xymatrix@C=180pt@R=15pt{ *+[r]{X^\ran\t\M^\top\t\M^\top} \ar@{}[dr]^(0.74)\simeq\ar[r]_(0.48){\raisebox{-8pt}{$\st[u^\top\ci(\Pi_1,\Pi_2)]\t[u^\top\ci(\Pi_1,\Pi_3)]$}} \ar[d]^{\id_{X^\ran}\t\Phi^\top } & *+[l]{(B\U\t\Z)\t(B\U\t\Z)} \ar[d]_{\mu'} \\
*+[r]{X^\ran\t\M^\top} \ar[r]^(0.48){u^\top} & *+[l]{B\U\t\Z.\!}  }
\end{gathered}
\label{ss3eq18}
\e

Translating \eq{ss3eq18} into a diagram of mapping spaces $\Map_{X^0}(X^\ran,-)$ and using the notation $\cC,\Ga,\Psi$ in Theorem \ref{ss1thm2}(a), where $\Ga,\Psi$ correspond to $u^\top,\mu'$, gives a homotopy commutative diagram:
\begin{equation*}
\xymatrix@C=170pt@R=15pt{ *+[r]{\M^\top\t\M^\top} \ar@{}[dr]_(0.41)\simeq\ar[r]_(0.33){\Ga\t\Ga} \ar[d]^{\Phi^\top } & *+[l]{\cC\t\cC=\Map_{C^0}(X^\ran,B\U\t\Z)^2} \ar[d]_{\Psi=\mu'\ci} \\
*+[r]{\M^\top} \ar[r]^(0.33)\Ga & *+[l]{\cC=\Map_{C^0}(X^\ran,B\U\t\Z).\!}  }
\end{equation*}
This gives the homotopy $\Ga\ci\Phi^\top\simeq\Psi\ci(\Ga\t\Ga)$ claimed in Theorem \ref{ss1thm2}(a). The rest of Theorem \ref{ss1thm2}(a) is immediate.

The next two definitions define and study $\phi:O^\om\bt_{\Z_2} O^\om\ra\Phi^*(O^\om)$ in equation \eq{ss1eq8} of Theorem~\ref{ss1thm2}(c).

\begin{dfn} 
\label{ss3def7}
Let $X$ be a smooth projective $\C$-scheme, and use the notation above. Write $\Pi_1,\Pi_2,\Pi_3$ for the projection to the first--third factors of $X\t\M\t\M$. Define a perfect complex $\cExt^\bu$ on $\M\t\M$, the {\it Ext complex}, by
\e
\cExt^\bu=(\Pi_2,\Pi_3)_*\bigl[(\Pi_1,\Pi_2)^*(\cU^\bu)^\vee\ot(\Pi_1,\Pi_3)^*(\cU^\bu)\bigr],
\label{ss3eq19}
\e
using derived pushforward, pullback and tensor product functors as in Huybrechts \cite{Huyb}. Then for $F^\bu,G^\bu$ in $D^b\coh(X)$ and $k\in\Z$ we have
\e
H^k\bigl(\cExt^\bu\vert_{([F^\bu],[G^\bu])}\bigr)\cong \Ext^k(F^\bu,G^\bu):=\Hom_{D^b\coh(X)}\bigl(F^\bu,G^\bu[k]\bigr).
\label{ss3eq20}
\e

As in \S\ref{ss14}, we can consider the cotangent complex $\bL_{\bs\M}$ and its restriction $\bL_{\bs\M}\vert_\M$. It is well known that
\e
H^k\bigl(\bL_{\bs\M}\vert_{[F^\bu]}\bigr)\cong \Ext^{1-k}(F^\bu,F^\bu)^*.
\label{ss3eq21}
\e
Since $\bL_{\bs\Perf_\C}\vert_{\Perf_\C}\cong\cU_0^\bu\ot^L(\cU_0^\bu)^\vee[-1]$, using facts about (co)tangent complexes of mapping stacks in \cite[\S 2.1]{PTVV} we can prove that
\e
\bL_{\bs\M}\vert_\M\cong\De_\M^*\bigl((\cExt^\bu)^\vee[-1]\bigr),
\label{ss3eq22}
\e
as we would expect from comparing \eq{ss3eq20} and \eq{ss3eq21}, where $\De_\M:\M\ra\M\t\M$ is the diagonal morphism. Taking determinant line bundles in \eq{ss3eq22} gives a canonical isomorphism of line bundles on~$\M$:
\e
\smash{\det(\bL_{\bs\M}\vert_\M)\cong\De_\M^*\bigl(\det(\cExt^\bu)\bigr).}
\label{ss3eq23}
\e

If $F_1^\bu,\ldots,F_4^\bu\in D^b\coh(X)$ there is a natural isomorphism
\e
\Ext^k(F_1^\bu\op F_2^\bu,F_3^\bu\op F_4^\bu)\cong\ts\bigop_{i=1,2}^{j=3,4}\Ext^k(F_i^\bu,F_j^\bu).
\label{ss3eq24}
\e
Corresponding to this we have an isomorphism of complexes on~$\M\t\M\t\M\t\M$:
\e
\begin{split}
(\Phi\t\Phi)^*(\cExt^\bu)\cong\,& (\Pi_1,\Pi_3)^*(\cExt^\bu)\op (\Pi_1,\Pi_4)^*(\cExt^\bu)\op{}\\
&(\Pi_2,\Pi_3)^*(\cExt^\bu)\op (\Pi_2,\Pi_4)^*(\cExt^\bu).	
\end{split}
\label{ss3eq25}
\e
To prove \eq{ss3eq25}, note that 
\ea
&(\Phi\t\Phi)^*(\cExt^\bu)=(\Phi\t\Phi)^*\bigl((\Pi''_2,\Pi''_3)_*\bigl[(\Pi''_1,\Pi''_2)^*(\cU^\bu)^\vee\ot(\Pi''_1,\Pi''_3)^*(\cU^\bu)\bigr]\bigr)\nonumber\\
&\cong(\Pi'_1,\ldots,\Pi'_4)_*\ci(\id_X\t\Phi\t\Phi)^*\bigl[(\Pi''_1,\Pi''_2)^*(\cU^\bu)^\vee\ot(\Pi''_1,\Pi''_3)^*(\cU^\bu)\bigr]
\nonumber\\
&\cong(\Pi'_1,\ldots,\Pi'_4)_*\bigl[(\Pi_0',\Pi_1',\Pi_2')^*\ci(\id_X\t\Phi)^*(\cU^\bu)^\vee
\nonumber\\
&\qquad\qquad\qquad\ot (\Pi_0',\Pi_3',\Pi_4')^*\ci(\id_X\t\Phi)^*(\cU^\bu)\bigr]
\nonumber\\
&\cong(\Pi'_1,\ldots,\Pi'_4)_*\bigl[(\Pi_0',\Pi_1',\Pi_2')^*
[(\Pi''_1,\Pi''_2)^*(\cU^\bu)\op (\Pi''_1,\Pi''_3)^*(\cU^\bu)]^\vee
\nonumber\\
&\qquad\qquad\qquad \ot(\Pi_0',\Pi_3',\Pi_4')^*[(\Pi''_1,\Pi''_2)^*(\cU^\bu)\op (\Pi''_1,\Pi''_3)^*(\cU^\bu)]\bigr]
\label{ss3eq26}\\
&\cong(\Pi'_1,\ldots,\Pi'_4)_*\bigl[(\Pi_0',\Pi_1')^*(\cU^\bu)^\vee\ot(\Pi_0',\Pi_3')^*(\cU^\bu)
\nonumber\\
&\qquad\op(\Pi_0',\Pi_1')^*(\cU^\bu)^\vee\ot(\Pi_0',\Pi_4')^*(\cU^\bu)\op
\nonumber\\
&\qquad\op
(\Pi_0',\Pi_2')^*(\cU^\bu)^\vee\ot(\Pi_0',\Pi_4')^*(\cU^\bu)\bigr]
\nonumber\\
&\cong(\Pi_1,\Pi_3)^*(\cExt^\bu)\op (\Pi_1,\Pi_4)^*(\cExt^\bu)\op
(\Pi_2,\Pi_3)^*(\cExt^\bu)\op (\Pi_2,\Pi_4)^*(\cExt^\bu),
\nonumber
\ea
where we write $\Pi_1,\ldots,\Pi_4$, and $\Pi_0',\ldots,\Pi_4'$, and $\Pi_1'',\ldots,\Pi_3''$, for the projections to the factors of $\M\t\M\t\M\t\M$, and $X\t\M\t\M\t\M\t\M$, and $X\t\M\t\M$, respectively. Here we use \eq{ss3eq19} in the first and sixth steps, and \eq{ss3eq17} in the fourth, and properties of pullbacks and pushforwards in the rest.

Write $\De_\M^{13}\t\De_\M^{24}:\M\t\M\ra\M\t\M\t\M\t\M$ for the diagonal morphism acting on points by $([F^\bu],[G^\bu])\mapsto([F^\bu],[G^\bu],[F^\bu],[G^\bu])$. Then pulling back \eq{ss3eq27} by $\De_\M^{13}\t\De_\M^{24}$, taking determinant line bundles, and using \eq{ss3eq23}, yields an isomorphism of line bundles on $\M\t\M$:
\e
\begin{split}
\Phi^*(\det(\bL_{\bs\M}\vert_\M))\cong\,& \Pi_1^*(\det(\bL_{\bs\M}\vert_\M))\ot \det(\cExt^\bu)\ot{} \\
&\si^*(\det(\cExt^\bu))\ot\Pi_2^*(\det(\bL_{\bs\M}\vert_\M)),
\end{split}
\label{ss3eq27}
\e
where $\si:\M\t\M\ra\M\t\M$ exchanges the factors.
\end{dfn}

\begin{dfn}
\label{ss3def8}
Suppose $(X,\th)$ is an algebraic Calabi--Yau $2m$-fold, and use the notation of Definition \ref{ss3def7} for $X$. Pantev--To\"en--Vaqui\'e--Vezzosi \cite[Cor.~2.13]{PTVV} give $\bs\M$ a $(2-2m)$-shifted symplectic structure $\om$, so Definition \ref{ss1def6} defines a notion of orientation on $(\bs\M,\om)$, which form an algebraic principal $\Z_2$-bundle $\pi:O^\om\ra\M$. Write $O^\om_\al=O^\om\vert_{\M_\al}$ for each~$\al\in K^0(X^\ran)$.

Serre duality on the Calabi--Yau $2m$-fold $X$ gives isomorphisms
\e
\Ext^k(F^\bu,G^\bu)\cong \Ext^{2m-k}(G^\bu,F^\bu)^*
\label{ss3eq28}
\e
for $F^\bu,G^\bu$ in $D^b\coh(X)$ and $k\in\Z$. This corresponds to an isomorphism
\e
\cExt^\bu\cong\si^*\bigl((\cExt^\bu)^\vee[-2m]\bigr),
\label{ss3eq29}
\e
which may be proved from \eq{ss3eq19} using the fact that the dualizing complex of a Calabi--Yau $2m$-fold $X$ is $\O_X[2m]$. Taking determinant line bundles in \eq{ss3eq29} gives an isomorphism
\e
\det\cExt^\bu\cong\si^*\bigl(\det(\cExt^\bu)\bigr)^\vee.
\label{ss3eq30}
\e
Substituting this into \eq{ss3eq27} yields an isomorphism 
\e
\Phi^*(\det(\bL_{\bs\M}\vert_\M))\cong \Pi_1^*(\det(\bL_{\bs\M}\vert_\M))\ot \Pi_2^*(\det(\bL_{\bs\M}\vert_\M)).
\label{ss3eq31}
\e

The isomorphism $\om\cdot{}:\bT_{\bs\M}[-1]\ra\bL_{\bs\M}[1-2m]$ used in Definition \ref{ss1def6} to define $\io^\om:\det(\bL_{\bs\M}\vert_\M)\ra \det(\bL_{\bs\M}\vert_\M)^{-1}$ and the orientation bundle $O^\om\ra\M$ agrees with $\De_\M^*(\eq{ss3eq29})$ under \eq{ss3eq22}. Therefore  \eq{ss3eq31} is compatible with the isomorphisms $\io^\om$ used in Definition \ref{ss1def6} to define the orientation bundles $O^\om\ra\M$. Thus as in \eq{ss1eq8} it induces an isomorphism
\e
\phi:O^\om\bt_{\Z_2}O^\om\longra \Phi^*(O^\om)
\label{ss3eq32}
\e
of principal $\Z_2$-bundles on $\M\t\M$, which is an algebraic analogue of $\psi$ in \eq{ss3eq9}.

As for \eq{ss3eq10}, using analogues of \eq{ss3eq24} for $\Ext^k(F_1^\bu\op F_2^\bu\op F_3^\bu,F_4^\bu\op F_5^\bu\op F_6^\bu)$ and of \eq{ss3eq25}--\eq{ss3eq27} on $\M\t\M\t\M\t\M\t\M\t\M$, we can show that
\e
(\Phi\t\id_\M)^*(\phi)\ci(\phi\bt\id_{O^\om})
=(\id_\M\t\Phi)^*(\phi)\ci(\id_{O^\om}\bt\phi).
\label{ss3eq33}
\e
\end{dfn}

Applying $(-)^\top$ to $O^\om\ra\M$ gives a principal $\Z_2$-bundle $O^{\om,\top}\ra\M^\top$. Then applying $(-)^\top$ to \eq{ss3eq32}--\eq{ss3eq33} (where the analogue of \eq{ss3eq33} holds up to homotopy as for \eq{ss3eq10}) and using Definition \ref{ss3def3} yields:

\begin{lem}
\label{ss3lem2}
In Definition\/ {\rm\ref{ss3def8},} $O^{\om,\top}\ra\M^\top$ is a weak H-principal\/ $\Z_2$-bundle. If\/ $O^{\om,\top}$ is trivializable then $(O^{\om,\top},\phi^\top)$ is a strong H-principal\/ $\Z_2$-bundle.
\end{lem}

\subsection{\texorpdfstring{Set up of the proof, for smooth projective $\C$-schemes}{Set up of the proof, for smooth projective ℂ-schemes}}
\label{ss34}

Although Theorem \ref{ss1thm2}(b),(c) concern algebraic Calabi--Yau $4m$-folds $(X,\th)$, we write the first part of the proof in this section for a general smooth projective $\C$-scheme $X$, as we hope in future to use the same method to prove other things. Then in \S\ref{ss35} we specialize to Calabi--Yau~$4m$-folds.

Throughout we fix a smooth projective $\C$-scheme $X$, and use the notation of Theorem \ref{ss1thm1}(a) and \S\ref{ss31}, \S\ref{ss33} for $X$, and \S\ref{ss32} with $X^\ran$ in place of $X$.

The next definitions and results, until Lemma \ref{ss3lem5}, are heavily based on Friedlander and Walker~\cite[\S 2]{FrWa}.

\begin{dfn} Write $\IndSch_\C$ for the category of $\C$-{\it ind-schemes}, as in Gaitsgory and Rozenblyum \cite{GaRo1}, \cite[\S 2]{GaRo2}. Suppose $T_0\,{\buildrel\tau_0\over\longra}\,T_1\,{\buildrel\tau_1\over\longra}\,T_2\,{\buildrel\tau_2\over\longra}\,\cdots$ is a sequence of closed embeddings of $\C$-schemes. Then we may form the direct limit $T=\varinjlim_{k\ra\iy}T_k$ as a $\C$-ind-scheme. All the $\C$-ind-schemes in this paper are of this type. If $S$ is a finite type $\C$-scheme then 
\e
\Hom_{\IndSch_\C}(S,T)=\varinjlim\nolimits_{k\ra\iy}\Hom_{\Sch_\C}(S,T_k).
\label{ss3eq34}
\e
That is, every morphism $S\ra T$ in $\IndSch_\C$ factors through the inclusion $T_k\hookra T$ from the direct limit for $k\gg 0$.

As a simple example of a $\C$-ind-scheme, consider $\bA^0\,{\buildrel\al_0\over\longra}\,\bA^1\,{\buildrel\al_1\over\longra}\,\bA^2\,{\buildrel\al_2\over\longra}\,\cdots$ in $\Sch_\C$, where $\al_k:(t_1,\ldots,t_k)\mapsto (t_1,\ldots,t_k,0)$. Then $\bA^\iy:=\varinjlim_{k\ra\iy}\bA^k$ is a kind of infinite-dimensional affine space.

Every $\C$-scheme or $\C$-ind-scheme $S$ has an {\it underlying complex analytic space\/} $S^\ran$, a topological space. If $X$ is a smooth projective $\C$-scheme then $X^\ran$ has the structure of a compact complex manifold. If $T=\varinjlim_{k\ra\iy}T_k$ as above then $T^\ran=\varinjlim_{k\ra\iy}T_k^\ran$ as a direct limit in topological spaces.

We can consider $\C$-ind-schemes $T$ as examples of general $\C$-stacks (though not Artin $\C$-stacks), that is, there is an embedding $\IndSch_\C\subset\Ho(\Sta_\C)$, such that $\Hom_{\Sta_\C}(S,T)=\Hom_{\IndSch_\C}(S,T)$ for $S$ a finite type $\C$-scheme, as in \eq{ss3eq34}. Topological realizations work for $\C$-stacks of this type. The topological realization $S^\top$ of a $\C$-ind-scheme $S$, considered as a $\C$-stack, is homotopy-equivalent to $S^\ran$, so we may take $S^\top=S^\ran$.
\label{ss3def9}
\end{dfn}

\begin{dfn}
\label{ss3def10}
An {\it algebraic vector bundle\/} $E\ra X$ on $X$ is a locally free coherent sheaf $E$ on $X$ in the sense of \cite{Hart}. Then $E$ has a corresponding complex vector bundle $E^\ran\ra X^\ran$ and K-theory class $\lb E\rb=\lb E^\ran\rb\in K^0(X^\ran)$. 

We allow vector bundles $E\ra Y$ on disconnected $\C$-schemes $Y$ to have different ranks on different connected components of~$Y$.

We say that $E$ is {\it generated by global sections\/} if there exists a surjective morphism $V\ot_\C\O_X\ra E$ in $\coh(X)$, for $V$ a finite-dimensional $\C$-vector space.

If $E$ is generated by global sections, a {\it generating sequence $(s_1,s_2,\ldots)$ for\/} $E$ is a sequence $s_1,s_2,\ldots$ in $H^0(E)$ with $s_k=0$ for all $k\gg 0$, such that the morphism $(s_1,\ldots,s_N):\C^N\ot\O_X\ra E$ is surjective for~$N\gg 0$.

We write $\M^{\rm vb,gs}$ for the moduli stack of vector bundles $E\ra X$ generated by global sections, as an Artin $\C$-stack. Then $\M^{\rm vb,gs}\subset\M$ is an open substack, considering $E\ra X$ as a complex in $D^b\coh(X)$ concentrated in degree~0.
\end{dfn}

\begin{dfn}
\label{ss3def11}
For $0\le k\le N$, write $\Gr_k(\C^N)$ for the Grassmannian of vector subspaces $V\subset\C^N$ with $\dim V=k$. It is a smooth projective $\C$-scheme of dimension $k(N-k)$. We set~$\Gr(\C^N)=\coprod_{k=0}^N\Gr_k(\C^N)$.

For any $\C$-scheme or $\C$-ind-scheme $Y$ we will write $\ul\C^N=\C^N\ot\O_Y$ for the trivial vector bundle $\ul\C^N\ra Y$ with fibre $\C^N$.

Then on $\Gr(\C^N)$ there is a tautological vector subbundle $F^{\rm taut}_N\subset\ul\C^N$, whose fibre at a $\C$-point $V\in\Gr(\C^N)$ is $F^{\rm taut}_N\big\vert_V=V$. There is a quotient vector bundle $F^{\rm quot}_N:=\ul\C^N/F^{\rm taut}_N$ on~$\Gr(\C^N)$.

Define a closed embedding $i_N:\Gr(\C^N)\hookra\Gr(\C^{N+1})$ to map $V\subset\C^N$ to $V\op\C\subset\C^N\op\C=\C^{N+1}$. Write $\Gr(\C^\iy)=\varinjlim_{N\ra\iy}\Gr(\C^N)$ for the direct limit using these closed embeddings. It is not a scheme, but it is an ind-scheme. It has an underlying complex analytic topological space $\Gr(\C^\iy)^\ran=\varinjlim_{N\ra\iy}\Gr(\C^N)^\ran$, where $\Gr(\C^N)^\ran$ is the complex analytic topological space of the $\C$-scheme $\Gr(\C^N)$, which is a disjoint union of complex manifolds.

In vector bundles on $\Gr(\C^N)$ we have natural isomorphisms $i_N^*(F^{\rm taut}_{N+1})\cong F^{\rm taut}_N\op\ul\C$ and $i_N^*(F^{\rm quot}_{N+1})\cong F^{\rm quot}_N$. Thus there is a natural quotient vector bundle $F^{\rm quot}_\iy\ra \Gr(\C^\iy)$ in the sense of ind-schemes, with canonical isomorphisms $F^{\rm quot}_\iy\vert_{\Gr(\C^N)}\cong F^{\rm quot}_N$ for all $N\ge 0$. The vector bundles $F^{\rm taut}_N$ do not extend to $\Gr(\C^\iy)$ in the same way.

Regarding $\Gr(\C^\iy)$ as a $\C$-stack, for any finite type $\C$-scheme $S$ we have
\begin{equation*}
\Hom_{\Sta_\C}\bigl(S,\Gr(\C^\iy)\bigr)=\varinjlim\nolimits_{N\ra\iy}\Hom_{\Sch_\C}(S,\Gr(\C^N)).
\end{equation*}	

Define $\chi_N:\Gr(\C^N)\t\Gr(\C^N)\ra\Gr(\C^{2N})$ to map $(V_1,V_2)\mapsto V_1\op V_2$ on $\C$-points, where $V_1\op V_2\subset\C^N\op\C^N$, and we identify $\C^N\op\C^N\cong\C^{2N}$ by
\begin{equation*}
\bigl((x_1,\ldots,x_N),(y_1,\ldots,y_N)\bigr)\simeq(x_1,y_1,x_2,y_2,\ldots,x_N,y_N).
\end{equation*}
Then the following commutes:
\begin{equation*}
\xymatrix@C=170pt@R=15pt{ *+[r]{\Gr(\C^N)\t\Gr(\C^N)} \ar[d]^{i_N\t i_N} \ar[r]_{\chi_N} & *+[l]{\Gr(\C^{2N})} \ar[d]_{i_{2N+1}\ci i_{2N}} \\
*+[r]{\Gr(\C^{N+1})\t\Gr(\C^{N+1})} \ar[r]^{\chi_{N+1}} & *+[l]{\Gr(\C^{2N+2}).\!} }
\end{equation*}
Hence by properties of direct limits we have a morphism $\chi_\iy=\varinjlim_{N\ra\iy}\chi_N:\Gr(\C^\iy)\t\Gr(\C^\iy)\ra\Gr(\C^\iy)$. It is homotopy commutative and associative on $\Gr(\C^\iy)^\ran$ by Friedlander--Walker \cite[Prop.~2.8]{FrWa} for~$X=\Spec\C$.
\end{dfn}

The next lemma is an elementary consequence of Definition~\ref{ss3def11}.

\begin{lem}
\label{ss3lem3}
There is a natural\/ {\rm 1-1} correspondence between $\C$-points $z$ in $\Gr(\C^\iy)$ and isomorphism classes $\bigl[V,(v_1,v_2,\ldots)\bigr]$ of pairs $\bigl(V,(v_1,v_2,\ldots)\bigr)$ where $V\cong F^{\rm quot}_\iy\vert_z$ is a finite-dimensional\/ $\C$-vector space and\/ $v_1,v_2,\ldots\in V$ with\/ $v_k=0$ for $k\gg 0$ and\/ $V=\an{v_1,v_2,\ldots}_\C$. Under this, $\chi_\iy$ acts on $\C$-points by
\e
\begin{split}
&\chi_\iy:\bigl((V,(v_1,v_2,\ldots)),(W,(w_1,w_2,\ldots))\bigr)
\longmapsto\\
&\qquad\qquad\bigl(V\op W,(v_1\op 0, 0\op w_1,v_2\op 0,0\op w_2,v_3\op 0,\ldots)\bigr).
\end{split}
\label{ss3eq35}
\e
\end{lem}

\begin{dfn}
\label{ss3def12}	
For $X$ a smooth projective $\C$-scheme as above and $N\ge 0$, we have a mapping $\C$-scheme $\Map_{\Sch_\C}\bigl(X,\Gr(\C^N)\bigr)$, which is a projective $\C$-scheme, whose $\C$-points $\phi$ are $\C$-scheme morphisms $\phi:X\ra\Gr(\C^N)$. It has a tautological morphism
\begin{equation*}
\tau_{X,N}:X\t \Map_{\Sch_\C}\bigl(X,\Gr(\C^N)\bigr)\longra \Gr(\C^N)
\end{equation*}
with $\tau_{X,N}(x,\phi)=\phi(x)$ on $\C$-points. 

On $X\t \Map_{\Sch_\C}\bigl(X,\Gr(\C^N)\bigr)$ we have the trivial vector bundle $\ul\C^N$, with a vector subbundle $\tau_{X,N}^*(F^{\rm taut}_N)\subseteq\ul\C^N$, and quotient vector bundle
\begin{equation*}
F_{X,N}^{\rm quot}:=\ul\C^N\big/\tau_{X,N}^*(F^{\rm taut}_N)=\tau_{X,N}^*(F^{\rm quot}_N).
\end{equation*}
We regard $F_{X,N}^{\rm quot}\ra X\t \Map_{\Sch_\C}\bigl(X,\Gr(\C^N)\bigr)$ as a family of vector bundles on $X$ over the base scheme $\Map_{\Sch_\C}\bigl(X,\Gr(\C^N)\bigr)$. As we have a surjective morphism $\ul\C^N\ra F_{X,N}^{\rm quot}$, these vector bundles are generated by global sections. Thus by definition of $\M^{\rm vb,gs}$, $F_{X,N}^{\rm quot}$ is equivalent to a $\C$-stack morphism
\begin{equation*}
f_{X,N}^{\rm quot}:\Map_{\Sch_\C}\bigl(X,\Gr(\C^N)\bigr)\longra\M^{\rm vb,gs}\subset\M.
\end{equation*}

Composition with $i_N:\Gr(\C^N)\hookra\Gr(\C^{N+1})$ induces a closed embedding
$\Map_{\Sch_\C}\bigl(X,\Gr(\C^N)\bigr)\!\hookra\! \Map_{\Sch_\C}\bigl(X,\Gr(\C^{N+1})\bigr)$ mapping $\phi\mapsto i_N\ci\phi$ on $\C$-points. Taking the direct limit gives the mapping space, as a $\C$-ind-scheme
\begin{equation*}
\T:=\Map_{\IndSch_\C}\bigl(X,\Gr(\C^\iy)\bigr)=\varinjlim\nolimits_{N\ra\iy}\Map_{\Sch_\C}\bigl(X,\Gr(\C^N)\bigr).
\end{equation*}
Its complex analytic topological space $\T^\ran$ is a direct limit in the obvious way.

Taking direct limits of the morphisms $\tau_{X,N}$ as $N\ra\iy$ gives a morphism $\tau_{X,\iy}:X\t \T\longra \Gr(\C^\iy)$. We have a vector bundle $F_{X,\iy}^{\rm quot}$ on $X\t \T$ given by $F_{X,\iy}^{\rm quot}:=\tau_{X,\iy}^*(F^{\rm quot}_\iy)$, the direct limit of the vector bundles $F_{X,N}^{\rm quot}$ on $X\t \Map_{\Sch_\C}\bigl(X,\Gr(\C^N)\bigr)$. The vector bundle $F_{X,\iy}^{\rm quot}$ is equivalent to a $\C$-stack morphism
\begin{equation*}
\De:\T\longra\M^{\rm vb,gs}\subset\M,
\end{equation*}
the direct limit of the $f_{X,N}^{\rm quot}$ as $N\ra\iy$.

Define a morphism $\Xi:\T\t\T\ra\T$ in $\IndSch_\C$ to be the composition
\begin{equation*}
\xymatrix@C=40pt{
\T\t\T \ar[r]^(0.3){\begin{subarray}{l}\text{direct product} \\ \;\>\text{of maps}\end{subarray}} & \Map_{\IndSch_\C}\bigl(X,\Gr(\C^\iy)\t\Gr(\C^\iy)\bigr)
\ar[r]^(0.75){\chi_\iy\ci} & \T. }
\end{equation*}
Then since $\chi_\iy$ acts as direct sum on the underlying vector spaces as in \eq{ss3eq35}, the following commutes in $\Ho(\Sta_\C)$:
\e
\begin{gathered}
\xymatrix@C=150pt@R=15pt{ *+[r]{\T\t\T} \ar[r]_\Xi \ar[d]^{\De\t\De} & *+[l]{\T} \ar[d]_\De \\ 
*+[r]{\M\t\M} \ar[r]^\Phi  & *+[l]{\M.\!} }	
\end{gathered}
\label{ss3eq36}
\e
\end{dfn}

The next lemmas are elementary consequences of Definition~\ref{ss3def12}.

\begin{lem}
\label{ss3lem4}
There is a natural\/ {\rm 1-1} correspondence between $\C$-points $t$ in $\T=\Map_{\IndSch_\C}\bigl(X,\Gr(\C^\iy)\bigr)$ and isomorphism classes $\bigl[F,(s_1,s_2,\ldots)\bigr]$ of pairs $\bigl(F,(s_1,s_2,\ldots)\bigr)$ where $F\ra X$ is a vector bundle on $X$ generated by global sections, with $F\cong t^*(F_\iy^{\rm quot}),$ and\/ $(s_1,s_2,\ldots)$ is a generating series for\/ $F$.
Under this, $\Xi$ acts on $\C$-points by
\begin{align*}
&\Xi:\bigl((F,(s_1,s_2,\ldots)),(F',(s'_1,s'_2,\ldots))\bigr)
\longmapsto\\
&\qquad\qquad\bigl(F\op F',(s_1\op 0, 0\op s'_1,s_2\op 0,0\op s'_2,s_3\op 0,\ldots)\bigr).
\end{align*}
\end{lem}

\begin{lem}
\label{ss3lem5}
In Definition\/ {\rm\ref{ss3def12},} $\T^\ran$ is an H-space, with multiplication\/ $\Xi^\ran,$ and identity $(0,(0,0,\ldots))$ under the {\rm 1-1} correspondence in Lemma\/ {\rm\ref{ss3lem4}}.	Applying $(-)^\top$ to \eq{ss3eq36} shows that\/ $\De^\top:\T^\ran=\T^\top\ra\M^\top$ is an H-space morphism.
\end{lem}

\begin{rem} Our $\T^\ran$ coincides with $\mathscr{M}\kern -.15em or(X,{\rm Grass})^\ran$ in Friedlander--Walker \cite[Def.~2.5]{FrWa}. Friedlander and Walker \cite[Prop.~2.8]{FrWa} give $\T^\ran$ the structure of an $E_\iy$-space \cite{May2}, an enhancement of an H-space.
\label{ss3rem2}	
\end{rem}

\begin{prop} $\De^\top:\T^\ran\ra\M^\top$ is a homotopy-theoretic group completion, in the sense of Definition\/~{\rm\ref{ss3def2}}.\label{ss3prop3}	
\end{prop}

\begin{proof} Friedlander and Walker \cite[\S 2, Th.~3.4]{FrWa} define the {\it semi-topological K-theory space\/} $\Om^\iy{\mathscr K}^{\rm st}(X)$ of $X$ to be the homotopy-theoretic group completion of $\T^\ran$ (their
$\mathscr{M}\kern -.15em or(X,{\rm Grass})^\ran$), which exists as $E_\iy$-spaces have homotopy-theoretic group completions. Thus, there is an H-space morphism $\T^\ran\ra\Om^\iy{\mathscr K}^{\rm st}(X)$ which is a homotopy-theoretic group completion.

Blanc \cite{Blan} defines {\it connective semi-topological K-theory spaces\/} $\Om^\iy{\mathscr K}^{\rm cn,st}(\cD)$ for $\C$-dg-categories $\cD$. In \cite[Prop.~4.17]{Blan} he shows $\Om^\iy{\mathscr K}^{\rm cn,st}(\cD)$ is homotopy equivalent to the topological realization of the moduli stack of objects in $\cD$. This gives an H-space homotopy equivalence~$\Om^\iy {\mathscr K}^{\rm cn,st}(\Perf(X))\ra\M^\top$. Antieau--Heller \cite[Th.~2.3]{AnHe} construct an H-space homotopy equivalence $\Om^\iy{\mathscr K}^{\rm st}(X)\ra\Om^\iy {\mathscr K}^{\rm cn,st}(\Perf(X))$, proving a conjecture of Blanc. 

Thus, the composition $\T^\ran\ra\Om^\iy{\mathscr K}^{\rm st}(X)\ra\Om^\iy{\mathscr K}^{\rm cn,st}(\Perf(X))\ra\M^\top$ is a homotopy-theoretic group completion. From the definitions we can show this composition is homotopic to $\De^\top$. The proposition follows.
\end{proof}

\begin{dfn}
\label{ss3def13}
Considering the complex analytic topological space $\T^\ran$ as a topological stack, we will define a morphism in $\Ho(\TopSta)$:
\e
\La:\T^\ran\longra \coprod\nolimits_{\begin{subarray}{l}\text{iso. classes $[P]$ of  principal}\\ \text{$\U(n)$-bundles $P\ra X$, $n\ge 0$}\end{subarray}}\B_P.
\label{ss3eq37}
\e
Here as in \eq{ss3eq6}, we sum over isomorphism classes $[P]$ of principal $\U(n)$-bundles $P\ra X$, and we pick one representative $P$ in each isomorphism class $[P]$ to give the corresponding $\B_P^\cla$. If $P'\in[P]$ is an alternative choice then $\B_P,\B_{P'}$ are canonically isomorphic.

We also write $\T^\ran_{[P]}=\La^{-1}(\B_P)$, so that $\T^\ran=\coprod_{\text{iso. classes $[P]$}}\T^\ran_{[P]}$, and $\La_{[P]}=\La\vert_{\T^\ran_{[P]}}$, so that $\La_{[P]}:\T^\ran_{[P]}\ra\B_P$.

We first define $\La$ at the level of points. By Lemma \ref{ss3lem4} we may write points of $\T^\ran$, which are $\C$-points of $\T$, as $[F,(s_1,s_2,\ldots)]$, where $F\ra X$ is a rank $n$ algebraic vector bundle generated by global sections, and $(s_1,s_2,\ldots)$ is a generating series for $F$. Write $F^\ran\ra X^\ran$ for the corresponding smooth complex vector bundle. Choose $N\gg 0$ such that $s_i=0$ for $i>N$, so that $(s_1,\ldots,s_N):\ul\C^N \ra F$ is surjective.

Using the Hermitian metric $h_{\C^N}$ on $\C^N$, in vector bundles on $X^\ran$ we split
\e
\ul\C^N=F^\ran\op\Ker(s_1,\ldots,s_N)^\ran 
\label{ss3eq38}
\e
identifying $F^\ran$ with the orthogonal complement $\bigl(\Ker(s_1,\ldots,s_N)^\ran\bigr){}^\perp$. Then $h_{\C^N}$ restricts to a Hermitian metric $h_{F^\ran}$ on $F^\ran$, giving $F^\ran$ a $\U(n)$-structure. The trivial connection $\nabla_0$ on $\ul\C^N$ and the orthogonal splitting \eq{ss3eq38} induces a $\U(n)$-connection $\nabla_{F^\ran}$ on $F^\ran$. Here $h_{F^\ran},\nabla_{F^\ran}$ are independent of the choice of $N\gg 0$ with $s_i=0$ for $i>N$, as increasing $N$ to $N+M$ replaces \eq{ss3eq38} by
\begin{equation*}
\ul\C^N\op\ul\C^M=F^\ran\op\Ker(s_1,\ldots,s_{N+M})^\ran =F^\ran\op\Ker(s_1,\ldots,s_N)^\ran\op\ul\C^M,
\end{equation*}
and the additional $\ul\C^M$ factors do not change $h_{F^\ran},\nabla_{F^\ran}$. Let $P\ra X^\ran$ be the principal $\U(n)$-bundle associated to $F^\ran,h_{F^\ran}$, and $\nabla_P$ the connection on $P$ corresponding to $\nabla_{F^\ran}$. Then~$[\nabla_P]\in\B_P$.

On the right hand side of \eq{ss3eq37} the isomorphism class $[P]$ is represented by some $P'\ra X^\ran$, say. Then there is a non-canonical isomorphism $P\ra P'$ (i.e.\ the isomorphism is canonical up to an element of $\G_{P'}=\Aut(P')$) which induces a canonical isomorphism $\B_P\ra \B_{P'}$ (as the definition of $\B_{P'}$ divides out by $\G_{P'}$). Then $\La$ should map $[F,(s_1,s_2,\ldots)]$ to the image of $[\nabla_P]\in\B_P$ under this isomorphism $\B_P\ra \B_{P'}$. This definition of $\La([F,(s_1,s_2,\ldots)])$ may be extended from points to continuous families of points over a topological space, and hence induces a morphism $\La$ of topological stacks in~\eq{ss3eq37}.

By considering the action on points $\bigl([F,(s_1,s_2,\ldots)],[F',(s'_1,s'_2,\ldots)]\bigr)$ of $\T^\ran$, which are sent to the point corresponding to $(F^\ran\op F^{\prime\ran},\nabla_{F^\ran}\op\nabla_{F^{\prime\ran}})$ around both routes in \eq{ss3eq39}, we see the following commutes in~$\Ho(\TopSta)$:
\e
\begin{gathered}
\xymatrix@C=170pt@R=15pt{ *+[r]{\T^\ran\t\T^\ran} \ar[r]_(0.4){\La\t\La} \ar[d]^{\Xi^\ran} & *+[l]{\coprod_{\text{iso. classes $[P],[Q]$}}\B_P\t\B_Q}\ar[d]_{\coprod_{[P],[Q]}\Phi_{P,Q}} \\
*+[r]{\T^\ran} \ar[r]^(0.4)\La & *+[l]{\coprod_\text{iso. classes $[R]$}\B_R.\!}  }
\end{gathered}
\label{ss3eq39}
\e

Choose a morphism $\La^\cla$ to make the following commute in $\Ho(\TopSta)$:
\e
\begin{gathered}
\xymatrix@C=180pt@R=15pt{ *+[r]{\T^\ran} \ar@/_.5pc/[dr]_\La  \ar[r]_{\La^\cla} & *+[l]{\coprod\limits_{\text{iso. classes $[P]$}}\B_P^\cla} \ar[d]^{\coprod_{[P]}\pi^\cla} 
\\ 
& *+[l]{\coprod\limits_{\text{iso. classes $[P]$}}\B_P.} }	
\end{gathered}
\label{ss3eq40}
\e
This is possible as $\coprod_{[P]}\pi^\cla$ is a fibration with contractible fibre $E\G_P$ over $\B_P$, and $\La^\cla$ is a section over $\La$. We can choose $\La^\cla$ to be the outcome of applying the functor $(-)^\cla$ to \eq{ss3eq37}, taking $(\T^\ran)^\cla=\T^\ran$ as it is a topological space. As in Lemma \ref{ss3lem5} and \eq{ss3eq6}, both sides of the top line of \eq{ss3eq40} are H-spaces, and applying $(-)^\cla$ to \eq{ss3eq39} shows that $\La^\cla$ is an H-space morphism. As for $\La_{[P]}$ above we write~$\La_{[P]}^\cla=\La^\cla\vert_{\T^\ran_{[P]}}:\T^\ran_{[P]}\ra\B_P^\cla$.
\end{dfn}

The following diagram will be important in the rest of the proof:
\e
\begin{gathered}
\xymatrix@C=200pt@R=15pt{ *+[r]{\T^\ran} \ar@<-1ex>@{}[dr]^\simeq \ar[r]_{\La^\cla} \ar[d]^{\De^\top} & *+[l]{\coprod\limits_{\text{iso. classes $[P]$}}\B_P^\cla} \ar[d]_{\coprod_{[P]}\Si_P^\cC}  \\ 
*+[r]{\M^\top} \ar[r]^\Ga  & *+[l]{\cC.\!} }	
\end{gathered}
\label{ss3eq41}
\e
We have proved above that \eq{ss3eq41} is a diagram of H-spaces and H-space morphisms, and the columns are homotopy-theoretic group completions. 

We claim that \eq{ss3eq41} homotopy commutes. To prove this, note that $\De:\T\ra\M$ factors via $\coprod_{n\ge 0}\M^{\rm vb}_n\hookra\M$, for $\M^{\rm vb}_n$ the moduli stack of rank $n$ vector bundles, where $\M^{\rm vb}_n=\Map_{\HSta_\C}(X,[*/\GL(n,\C)])$. Thus we have a commutative diagram in $\Ho(\HSta_\C)$:
\e
\begin{gathered}
\xymatrix@C=80pt@R=15pt{ *+[r]{X\t\T} \ar[r]_{\id_X\t v} \ar@<1ex>@/^.5pc/[rr]^{\id_X\t\De} \ar[d]^{\tau_{X,\iy}} & X\t\coprod_{n\ge 0}\M^{\rm vb}_n \ar[d]_w \ar[r] & *+[l]{X\t\M} \ar[d]_u \\
*+[r]{\Gr(\C^\iy)} \ar[r]^x & \coprod_{n\ge 0}[*/\GL(n,\C)] \ar[r] & *+[l]{\Perf_\C,} }
\end{gathered}
\label{ss3eq42}
\e
with columns the universal morphisms from the mapping stacks. Applying $(-)^\top$ to \eq{ss3eq42}, and taking $X^\top=X^\ran$, $[*/\GL(n,\C)]^\top=B\GL(n,\C)=B\U(n)$ as $\U(n)\hookra\GL(n,\C)$ is a homotopy equivalence, and $\Perf_\C^\top=B\U\t\Z$, gives a homotopy-commutative diagram
\e
\begin{gathered}
\xymatrix@C=65pt@R=15pt{ *+[r]{X^\ran\t\T^\ran} \ar@{}@<-1ex>[dr]^\simeq \ar[r]_{\id_{X^\ran}\t v^\top} \ar@<1ex>@/^.5pc/[rr]^{\id_{X^\ran}\t\De^\top} \ar[d]^{\tau_{X,\iy}} & X^\ran\t\coprod_{n\ge 0}(\M^{\rm vb}_n)^\top \ar@{}@<-1ex>[dr]^\simeq  \ar[d]_(0.42){w^\top} \ar[r] & *+[l]{X^\ran\t\M^\top} \ar[d]_(0.45){u^\top} \\
*+[r]{\Gr(\C^\iy)^\ran} \ar[r]^{x^\top} & {\begin{subarray}{l}\ts\coprod_{n\ge 0}B\GL(n,\C) \\ \ts =\coprod_{n\ge 0}B\U(n)\end{subarray}} \ar[r]^{\coprod_{n\ge 0}\Pi_n} & *+[l]{\begin{subarray}{l}\ts \Perf^\top_\C= \\ \ts B\U\t\Z.\end{subarray}\;\>} }\!\!\!\!
\end{gathered}
\label{ss3eq43}
\e

Converting \eq{ss3eq43} into a diagram of mapping spaces $\Map_{C^0}(X^\ran,-)$ yields
\e
\begin{gathered}
\xymatrix@C=49pt@R=15pt{ *+[r]{\T^\ran} \ar@{}@<-1ex>[dr]^\simeq \ar[r]_{v^\top} \ar@<1ex>@/^.5pc/[rr]^{\De^\top} \ar[d]^(0.45){(\tau_{X,\iy})_*} & \coprod_{n\ge 0}(\M^{\rm vb}_n)^\top \ar@{}@<-1ex>[dr]^\simeq  \ar[d]_(0.43){w^\top_*} \ar[r] & *+[l]{\M^\top} \ar[d]_(0.45){\Ga} \\
*+[r]{\begin{subarray}{l}\ts \Map_{C^0}(X^\ran, \\ \ts \Gr(\C^\iy)^\ran)\end{subarray}} \ar[r]^(0.4){x^\top_*} & {\begin{subarray}{l}\ts\coprod_{n\ge 0}\Map_{C^0}(X^\ran,B\GL(n,\C)) \\ \ts =\coprod_{n\ge 0}\Map_{C^0}(X^\ran,B\U(n))\end{subarray}} \ar[r]^(0.75){\coprod_{n\ge 0}\Pi_n\ci} & *+[l]{\cC.}}\!\!\!\!
\end{gathered}
\label{ss3eq44}
\e
Now consider the diagram
\begin{gather}
\nonumber\\[-26pt]
\begin{gathered}
\xymatrix@C=67pt@R=15pt{ *+[r]{\T^\ran} \ar[dr]_(0.5){w_*^\top\ci v^\top} \ar[drr]^y \ar[rr]_(0.65){\La^\cla} \ar[d]^(0.55){(\tau_{X,\iy})_*} && *+[l]{\raisebox{-18pt}{$\coprod\limits_{\text{iso. classes $[P]$}}\B_P^\cla\quad$}} \ar[d]_(0.55){\coprod_{[P]}\Si_P} \\
*+[r]{\begin{subarray}{l}\ts \Map_{C^0}(X^\ran, \\ \ts \;\>\Gr(\C^\iy)^\ran)\end{subarray}} \ar[r]^{x^\top_*} & {\coprod\limits_{n\ge 0}\begin{aligned}[t]\ts \Map_{C^0}(X^\ran, \\[-3pt] \ts B\GL(n,\C))\end{aligned}} & *+[l]{\vphantom{(}\smash{\coprod\limits_{n\ge 0}\begin{aligned}[t]\ts \Map_{C^0}(X^\ran, \\[-3pt] B\U(n)). \end{aligned}}} \ar[l]_(0.6)=  }\!\!\!\!\!\!\!\!\end{gathered}
\label{ss3eq45}
\end{gather}
Here each map $\T^\ran\ra\cdots$ is induced by a structure on $X^\ran\t\T^\ran$:
\begin{itemize}
\setlength{\itemsep}{0pt}
\setlength{\parsep}{0pt}
\item[(i)] $(\tau_{X,\iy})_*$ corresponds to $\bigl((F_{X,\iy}^{\rm quot})^\ran,s_1^\ran,s_2^\ran,\ldots\bigr)$, for $(F_{X,\iy}^{\rm quot})^\ran\ra X^\ran\t\T^\ran$ the complex vector bundle analytifying $F_{X,\iy}^{\rm quot}\ra X\t\T$, and $s_1^\ran,\ldots$ the family of generating sections. 
\item[(ii)] $\La^\cla$ corresponds to $\bigl((F_{X,\iy}^{\rm quot})^\ran,\ab h_{F_{X,\iy}^{\rm quot}},\nabla_{F_{X,\iy}^{\rm quot}}\bigr)$, for $h_{F_{X,\iy}^{\rm quot}}$ the Hermitian metric and $\nabla_{F_{X,\iy}^{\rm quot}}$ the connection on $(F_{X,\iy}^{\rm quot})^\ran$ defined using $s_1^\ran,\ldots$ as in Definition~\ref{ss3def13}.
\item[(iii)] $y$ corresponds to the Hermitian vector bundle~$\bigl((F_{X,\iy}^{\rm quot})^\ran,h_{F_{X,\iy}^{\rm quot}}\bigr)$.
\item[(iv)] $w_*^\top\ci v^\top$ corresponds to the complex vector bundle~$(F_{X,\iy}^{\rm quot})^\ran$.
\end{itemize}
The other maps $x^\top_*,\coprod_{[P]}\Si_P,=$ in \eq{ss3eq45} correspond to forgetting parts of these structures in the obvious way. Therefore \eq{ss3eq45} homotopy commutes. Comparing \eq{ss3eq44} and \eq{ss3eq45} now shows \eq{ss3eq41} homotopy commutes.

\subsection{Proof of parts (b),(c)}
\label{ss35}

We first relate the positive Dirac operator $\slashed{D}_+$ on a Calabi--Yau $4m$-fold with an operator $D_\C=\db+\db^*$ in \eq{ss3eq49} coming from the complex geometry of $(X^\ran,J)$.

\begin{dfn}
\label{ss3def14}	
Let $(X^\ran,J,g,\th)$ be a differential-geometric Calabi--Yau $4m$-fold. Then $X^\ran$ has an $\SU(4m)$-structure, and as $\SU(4m)$ is simply-connected $(X^\ran,g)$ has a natural spin structure, with real spin bundle $S\ra X^\ran$ and Dirac operator $\slashed{D}:\Ga^\iy(S)\ra\Ga^\iy(S)$, which is self-adjoint. The volume form $\vol_g$ of $g$ acts on $S$ with $\vol_g^2=\id$, so we may split $S=S_+\op S_-$ for $S_\pm$ the $\pm 1$-eigenspaces of $\vol_g$. Also $\slashed{D}\ci\vol_g=-\vol_g\ci \slashed{D}$, so $\slashed{D}$ maps $\Ga^\iy(S_\pm)\ra\Ga^\iy(S_\mp)$. The {\it positive Dirac operator\/} is~$\slashed{D}_+=\slashed{D}\vert_{S^+}:\Ga^\iy(S_+)\ra\Ga^\iy(S_-)$.

For each $q=0,\ldots,4m$ there is a complex antilinear, isometric vector bundle isomorphism $\star_q$, using the same notation for its action on sections:
\begin{align*}
\star_q&:\La^{0,q}T^*X^\ran\longra \La^{0,4m-q}T^*X^\ran,\\
\star_q&:\Ga^\iy(\La^{0,q}T^*X^\ran)\longra \Ga^\iy(\La^{0,4m-q}T^*X^\ran).
\end{align*}
It is a Calabi--Yau version of the Hodge star, characterized by
\begin{equation*}
\alpha\w\star_q\be=\an{\al,\be}\bar\th \qquad\text{for all $\al,\be\in\Ga^\iy(\La^{0,q}T^*X^\ran),$}
\end{equation*}
where $\an{\al,\be}:X^\ran\ra\C$ is the pointwise Hermitian product defined using $g$ which is $\C$-linear in $\al$ and\/ $\C$-antilinear in $\be,$ and $\bar\th\in\Ga^\iy(\La^{0,4m}T^*X^\ran)$ is the complex conjugate of $\th$. It is easy to check that
\ea
\star_{4m-q}\ci\star_q=(-1)^q\id& \qquad\qquad \text{on $\Ga^\iy(\La^{0,q}T^*X^\ran)$,}
\label{ss3eq46}\\
\db^*\ci\star_q\!=\!(-1)^{q+1}\star_{q+1}{}\ci\db&:\Ga^\iy(\La^{0,q}T^*X^\ran)\ra \Ga^\iy(\La^{0,4m-q-1}T^*X^\ran),
\label{ss3eq47}\\
\db\ci\star_q\!=\!(-1)^q\star_{q-1}{}\ci\db^*&:\Ga^\iy(\La^{0,q}T^*X^\ran)\ra \Ga^\iy(\La^{0,4m-q+1}T^*X^\ran).
\label{ss3eq48}
\ea

Define a complex elliptic operator $D_\C$ on the compact manifold~$X^\ran$:
\e
D_\C=\db+\db^*:\Ga^\iy\raisebox{-4pt}{$\biggl[$}\mathop{\ts\bigop\limits_{q=0,\ldots,2m\!\!\!\!\!\!\!\!\!\!\!\!\!\!\!\!\!}\La^{0,2q}T^*X^\ran}\limits_{\llcorner\qquad E^\C_0\qquad\lrcorner}\raisebox{-4pt}{$\biggr]$}\longra\Ga^\iy\raisebox{-4pt}{$\biggl[$}\mathop{\ts\bigop\limits_{q-0,\ldots,2m-1\!\!\!\!\!\!\!\!\!\!\!\!\!\!\!\!\!\!\!\!\!\!\!\!\!\!\!\!}\La^{0,2q+1}T^*X^\ran}\limits_{\llcorner\qquad E^\C_1\qquad\lrcorner}\raisebox{-4pt}{$\biggr]$}.
\label{ss3eq49}
\e
Define complex antilinear, isometric vector bundle isomorphisms on $X^\ran$:
\begin{align*}
&\he_0:E^\C_0\longra E^\C_0, & &\he_1:E^\C_1\longra E^\C_1\qquad\text{by}\\
&\he_0\vert_{\La^{0,2q}T^*X^\ran}=(-1)^q\star_{2q}, & &\he_1\vert_{\La^{0,2q+1}T^*X^\ran}=(-1)^{q+1}\star_{2q+1},
\end{align*}
and also write $\he_a$ for the actions on $\Ga^\iy(E^\C_a)$. Then \eq{ss3eq46}--\eq{ss3eq48} yield 
\e
\he_a^2=\id,\quad a=0,1, \quad\text{and}\quad
D_\C\ci\he_0=\he_1\ci D_\C.
\label{ss3eq50}
\e 
Hence $\he_a$ is a {\it real structure\/} on the complex vector bundle $E^\C_a$. Write $E_a^\R$ for the real vector subbundle of $E^\C_a$ fixed by $\he_a$ for $a=0,1$. Then $E^\C_a\cong E_a^\R\ot_\R\C$. By \eq{ss3eq50} we may write
\e
D_\R=D_\C\vert_{E_0^\R}:\Ga^\iy(E_0^\R)\longra\Ga^\iy(E_1^\R),
\label{ss3eq51}
\e
so that $D_\C=D_\R\ot_\R\id_\C$. One can now show that there are isomorphisms $E_0^\R\cong S_+$, $E_1^\R\cong S_-$ with the positive and negative real spinor bundles of $(X^\ran,g)$, which identify $D_\R$ with the positive Dirac operator $\slashed{D}_+:\Ga^\iy(S_+)\ra \Ga^\iy(S_-)$.
\end{dfn}

Now as in Theorem \ref{ss1thm2}(b) let $(X,\th)$ be an algebraic Calabi--Yau $4m$-fold, and $(X^\ran,J,g,\th)$ a corresponding differential-geometric Calabi--Yau $4m$-fold as in Remark \ref{ss1rem3}(a). As in Definition \ref{ss3def14} the Calabi--Yau structure induces a spin structure on $(X^\ran,g)$. Write $E_\bu$ for the positive Dirac operator $\slashed{D}_+:\Ga^\iy(S_+)\ra\Ga^\iy(S_-)$. Use the notation of \S\ref{ss32} for the compact manifold $X^\ran$ and elliptic operator $E_\bu$, giving an H-space $\cC$ and principal $\Z_2$-bundle $O^{E_\bu}\ra\cC$. Use the notation of \S\ref{ss33}--\S\ref{ss34} for $X$ and $(X,\th)$, giving a moduli stack $\M$ and principal $\Z_2$-bundle~$O^\om\ra\M$.

\begin{prop}
\label{ss3prop4}	
For each isomorphism class $[P]$ in \eq{ss3eq40} we have a canonical isomorphism of principal\/ $\Z_2$-bundles on $\T^\ran_{[P]}\!:$
\e
\la_{[P]}:(\De^\top)^*(O^{\om,\top})\vert_{\T^\ran_{[P]}}\longra
(\La^\cla_{[P]})^*\bigl((\pi^\cla)^*(O^{E_\bu}_P)\bigr).
\label{ss3eq52}
\e
The following commutes on $\T^\ran_{[P]}\t\T^\ran_{[Q]}$ for all isomorphism classes $[P],[Q]\!:$
\e
\begin{gathered}
\!\!\xymatrix@C=142pt@R=25pt{ *+[r]{(\De^\top\!\t\!\De^\top)^*(O^{\om,\top}\!\bt\! O^{\om,\top})\vert_{\T^\ran_{[P]}\t\T^\ran_{[Q]}}} \ar[r]_(0.6){\raisebox{-12pt}{$\st(\cdots)^*(\phi^\top)$}} \ar[d]^{\la_{[P]}\bt\la_{[Q]}} & *+[l]{(\Xi^\ran)^*\!\ci\!(\De^\top)^*(O^{\om,\top})\vert_{\T^\ran_{[P\op Q]}} } \ar@<-3ex>[d]_{\Xi^*(\la_{[P\op Q]})} \\
*+[r]{(\La^\cla_{[P]}\!\t\!\La^\cla_{[Q]})^*\bigl((\pi^\cla)^*(O^{E_\bu}_P\!\bt\!O^{E_\bu}_Q)\bigr)} \ar[r]^(0.5){\raisebox{8pt}{$\st(\cdots)^*(\psi^\cla_{P,Q})$}} & *+[l]{(\Xi^\ran)^*\!\ci\!(\La^\cla_{[P\op Q]})^*\bigl((\pi^\cla)^*(O^{E_\bu}_{P\op Q})\bigr). }}\!\!\!\!\!\!\!\!\!\!\!\!\end{gathered}
\label{ss3eq53}
\e
Here $\psi_{P,Q},\phi,\Xi,\La_{[P]}^\cla$ are as in Definitions\/ {\rm \ref{ss3def6}, \ref{ss3def8}, \ref{ss3def12},} and\/ {\rm\ref{ss3def13}.}
\end{prop}

\begin{proof} It is enough to define $\la_{[P]}$ at each point of $\T^\ran_{[P]}$. By Lemma \ref{ss3lem4} we may write points of $\T^\ran_{[P]}$ as $[F,(s_1,s_2,\ldots)]$, for $F\ra X$ an algebraic vector bundle and $s_1,s_2,\ldots$ generating sections, and set $n=\rank F$. Write $F^\ran\ra X^\ran$ for the associated complex vector bundle, which has the structure of a holomorphic vector bundle on $(X^\ran,J)$, and define $h_{F^\ran},\nabla_{F^\ran}$ as in Definition \ref{ss3def13}. Since $[F,(s_1,s_2,\ldots)]\in\T^\ran_{[P]}$, the principal $\U(n)$-bundle associated to $(F^\ran,h_{F^\ran})$ is isomorphic to $P$, and we identify it with $P$ as $F,F^\ran$ are only given up to isomorphism. Let $\nabla_P$ be the connection on $P$ corresponding to~$\nabla_{F^\ran}$.

By Definition \ref{ss3def13}, $\La_{[P]}:[F,(s_1,s_2,\ldots)]\mapsto[\nabla_P]$, so $\La_{[P]}^\cla:[F,(s_1,s_2,\ldots)]\mapsto(\nabla_P,e)\G_P$ for some $e\in E\G_P$. Thus by Definition \ref{ss1def2} we have
\e
(\La^\cla_{[P]})^*\bigl((\pi^\cla)^*(O^{E_\bu}_P)\bigr)\vert_{[F,(s_1,s_2,\ldots)]}\cong\Or(\det(\slashed{D}_+^{\nabla_P})).
\label{ss3eq54}
\e
Definitions \ref{ss1def6} and \ref{ss3def8} give
\e
\begin{split}
&(\De^\top)^*(O^{\om,\top})\vert_{[F,(s_1,s_2,\ldots)]}=O^\om\vert_{[F]}\\
&\quad =\bigl\{o_F:\det(\bL\vert_{\bs\M})\vert_{[F]}\,{\buildrel\cong\over\longra}\,\C \;\>\text{with}\;\> o_F^*\ci o_F=i^\om\vert_{[F]}\bigr\}.
\end{split}
\label{ss3eq55}
\e
Equation \eq{ss3eq21} implies that
\begin{equation*}
\det(\bL\vert_{\bs\M})\vert_{[F]}\cong\ts\bigot\limits_{k=0}^{4m}\det(\Ext^k(F,F))^{(-1)^k}\cong\ts\bigot\limits_{k=0}^{4m}\det(H^k(F^*\ot F))^{(-1)^k},
\end{equation*}
where $\Ext^k(F,F)\cong H^k(F^*\ot F)$ as $F$ is a vector bundle. Now $H^k(F^*\ot F)$ may be computed as the $k^{\rm th}$ cohomology of the Dolbeault-type complex
\e
\!\!\xymatrix@C=28pt{ \cdots \ar[r]^(0.3){\db^{\Ad(P)}} & \Ad(P)\!\ot_\R\!\La^{0,k}T^*X \ar[r]^(0.49){\db^{\Ad(P)}} & \Ad(P)\!\ot_\R\!\La^{0,k+1}T^*X \ar[r]^(0.7){\db^{\Ad(P)}} & \cdots, }\!\!
\label{ss3eq56}
\e
noting that $\Ad(P)\ot_\R\La^{0,k}T^*X\cong F^*\ot_\C F\ot_\C\La^{0,k}T^*X$. Using adjoints $(\db^{\Ad(P)})^*$ to roll \eq{ss3eq56} up into a single elliptic operator, we see that
\e
\det(\bL\vert_{\bs\M})\vert_{[F]}\cong \det_\C(D_\C^{\Ad(P)}),
\label{ss3eq57}
\e
where $D_\C^{\Ad(P)}$ is $D_\C$ in \eq{ss3eq49} twisted by $\Ad(P),\nabla_{\Ad(P)}$.

The map $\io^\om\vert_{[F]}:\det(\bL\vert_{\bs\M})\vert_{[F]}\ra \det(\bL\vert_{\bs\M})\vert_{[F]}^*$ is induced by Serre duality \eq{ss3eq28} when $F^\bu=G^\bu=F$. The analytic counterpart of this is that the real structure $\he_0,\he_1$ on $E_0^\C,E_1^\C,D_\C$ in Definition \ref{ss3def14} induces an isomorphism
$\he:\det_\C(D_\C^{\Ad(P)})\ra\ov{\det_\C(D_\C^{\Ad(P)})}$, and hence by \eq{ss3eq57} an isomorphism $\he:\det(\bL\vert_{\bs\M})\vert_{[F]}\ra \ov{\det(\bL\vert_{\bs\M})\vert_{[F]}}$, where $\ov{\vphantom{(}\cdots}$ is the complex conjugate. 

To define \eq{ss3eq57} we took adjoints with respect to the $L^2$ Hermitian metrics, which induce a Hermitian metric on $\det(\bL\vert_{\bs\M})\vert_{[F]}$. Then one can show
\begin{equation*}
\he:\det(\bL\vert_{\bs\M})\vert_{[F]}\ra \ov{\det(\bL\vert_{\bs\M})\vert_{[F]}},\quad \io^\om\vert_{[F]}:\det(\bL\vert_{\bs\M})\vert_{[F]}\ra \det(\bL\vert_{\bs\M})\vert_{[F]}^* 
\end{equation*}
are related by composition with the isomorphism $\ov{\det(\bL\vert_{\bs\M})\vert_{[F]}}\ra \det(\bL\vert_{\bs\M})\vert_{[F]}^*$ induced by this Hermitian metric. Thus we have isomorphisms of $\Z_2$-torsors
\e
\begin{split}
&(\De^\top)^*(O^{\om,\top})\vert_{[F,(s_1,s_2,\ldots)]}{\buildrel\text{\eq{ss3eq55} and above}\over\cong} \Or\bigl[(\det(\bL\vert_{\bs\M})\vert_{[F]})^\R\bigr]\\
&\quad{\buildrel\eq{ss3eq57}\over\cong}
\Or\bigl[(\det_\C(D_\C^{\Ad(P)}))^\R\bigr]{\buildrel\eq{ss3eq51}\over\cong}\Or\bigl[\det_\R(D_\R^{\Ad(P)})\bigr]\\
&\quad{\buildrel\text{Def.~\ref{ss3def14}}\over\cong}\Or\bigl[\det_\R(\slashed{D}_+^{\Ad(P)})\bigr]{\buildrel\eq{ss3eq54}\over\cong}(\La^\cla_{[P]})^*\bigl((\pi^\cla)^*(O^{E_\bu}_P)\bigr)\vert_{[F,(s_1,s_2,\ldots)]},
\end{split}
\label{ss3eq58}
\e
taking real parts using $\he$. Define $\la_{[P]}\vert_{[F,(s_1,s_2,\ldots)]}$ to be the composition of the isomorphisms \eq{ss3eq58}. As all the operations we used to define it depend continuously on the base point, $\la_{[P]}$ is continuous, and so defines an isomorphism of principal $\Z_2$-bundles as in~\eq{ss3eq52}.

To prove \eq{ss3eq53}, we restrict to 
$\bigl([F,(s_1,s_2,\ldots)],[G,(t_1,t_2,\ldots)]\bigr)\in\T^\ran_{[P]}\t\T^\ran_{[Q]}$, and use a similar method to the above, but involving the sheaves $F,G$ and $F\op G$, and bundles with connections $(P,\nabla_P),(Q,\nabla_Q)$ and $(P\op Q,\nabla_P\op\nabla_Q)$. The important point is that  the definition of $\psi_{P,Q}\vert_{([\nabla_P],[\nabla_Q])}$ in \eq{ss3eq12} involves taking the real determinant of a complex operator $D^{\smash{\nabla_{(P\t_{\U(n)}\ov\C^n)\ot_\C (Q\t_{\U(n')}\C^{n'})}}}$, and noticing this is naturally oriented. The definition of $\phi\vert_{([F],[G])}$ in Definition \ref{ss3def8} involves noticing by \eq{ss3eq28} and \eq{ss3eq30} that $\det(\Ext^*(F,G))\ot\det(\Ext^*(G,F))$ is naturally trivial. These two ideas correspond under the construction of $\la_{[P]}$ above. We leave the details to the reader.
\end{proof}

For $\La^\cla$ as in \eq{ss3eq40}, define an isomorphism
\begin{equation*}
\la:(\De^\top)^*(O^{\om,\top})\longra
(\La^\cla)^*\bigl(\ts\coprod_{\text{iso. classes $[P]$}}(\pi^\cla)^*(O^{E_\bu}_P)\bigr)
\end{equation*}
of principal $\Z_2$-bundles on $\T^\ran$ by $\la\vert_{\T^\ran_{[P]}}=\la_{[P]}$ in \eq{ss3eq52} for all $[P]$. 

By Lemmas \ref{ss3lem1} and \ref{ss3lem2} we have weak H-principal bundles $O^{\om,\top}\ra\M^\top$ and $O^{E_\bu}\ra\cC$. Thus, $O^{\om,\top}$ and $\Ga^*(O^{E_\bu})$ are both weak H-principal-bundles on $\M^\top$. We will show they are isomorphic. We have a diagram of isomorphisms of principal $\Z_2$-bundles on~$\T^\ran$:
\e
\begin{gathered}
\xymatrix@C=170pt@R=15pt{
*+[r]{(\De^\top)^*(O^{\om,\top})} \ar[r]_(0.45)\la \ar@{.>}[d]^(0.45)\de &
*+[l]{(\La^\cla)^*\bigl(\ts\coprod\limits_{[P]}(\pi^\cla)^*(O^{E_\bu}_P)\bigr)} \ar[d]_{(\La^\cla)^*(\coprod_{[P]}\si_P^\cC)} \\
*+[r]{(\De^\top)^*\ci\Ga^*(O^{E_\bu})} & *+[l]{(\La^\cla)^*\ci(\coprod_{[P]}\Si_P^\cC)^*(O^{E_\bu}),\!} \ar[l]_(0.55){\text{$\cong$ by \eq{ss3eq41}}}
}
\end{gathered}
\label{ss3eq59}
\e
where the bottom line is by parallel translation along the homotopy in \eq{ss3eq41}. Thus composing gives an isomorphism $\de:(\De^\top)^*(O^{\om,\top})\ra (\De^\top)^*\ci\Ga^*(O^{E_\bu})$ of weak principal H-bundles on $\T^\ran$. Since $\De^\top:\T^\ran\ra\M^\top$ is a homotopy-theoretic group completion by Proposition \ref{ss3prop3}, uniqueness in Proposition \ref{ss3prop1}(a) shows there exists an isomorphism $\ga:O^{\om,\top}\ra\Ga^*(O^{E_\bu})$ on $\M^\top$, as in \eq{ss1eq7}. This proves Theorem~\ref{ss1thm2}(b). 

For Theorem \ref{ss1thm2}(c), suppose $O^{E_\bu}\ra\cC$ is trivializable. Then $O^{\om,\top}\ra\M^\top$ is trivializable by Theorem \ref{ss1thm2}(b), so both are strong H-principal $\Z_2$-bundles by Lemmas \ref{ss3lem1} and \ref{ss3lem2}. All $(\pi^\cla)^*(O^{E_\bu}_P)$ are trivializable by~\eq{ss3eq11}. 

We upgrade \eq{ss3eq59} to a diagram of strong H-principal $\Z_2$-bundles on $\T^\ran$. Here $(\De^\top)^*(O^{\om,\top}),(\De^\top)^*\ci\Ga^*(O^{E_\bu}),(\La^\cla)^*\ci(\coprod_{[P]}\Si_P^\cC)^*(O^{E_\bu})$ are strong H-principal $\Z_2$-bundles as $O^{E_\bu},O^{\om,\top}$ are. We make $(\La^\cla)^*(\coprod_{[P]}(\pi^\cla)^*(O^{E_\bu}_P))$ a strong H-principal $\Z_2$-bundle with product $(\La^\cla\t\La^\cla)^*(\coprod_{[P],[Q]}(\pi^\cla)^*(\psi_{P,Q}))$, for $\psi_{P,Q}$ as in \eq{ss3eq13}, where \eq{ss3eq1} holds by \cite[(2.12)]{JTU}. Also $\la$ is an isomorphism of strong H-principal $\Z_2$-bundles by \eq{ss3eq53}, and $(\La^\cla)^*(\coprod_{[P]}\si_P^\cC)$ is by \eq{ss3eq14}. Note too that as $O^{E_\bu}$ is trivializable, the bottom line of \eq{ss3eq59} is independent of the choice of homotopy in \eq{ss3eq41}, as in Remark~\ref{ss1rem5}.

Thus, composing in \eq{ss3eq59} gives an isomorphism $\de:(\De^\top)^*(O^{\om,\top})\ra(\De^\top)^*\ci\Ga^*(O^{E_\bu})$ of strong H-principal $\Z_2$-bundles on $\T^\ran$. Hence uniqueness up to canonical isomorphism in Proposition \ref{ss3prop1}(b) gives a canonical isomorphism $\ga:O^{\om,\top}\ra \Ga^*(O^{E_\bu})$ of strong H-principal $\Z_2$-bundles on $\M^\top$, as in \eq{ss1eq7}, with $(\De^\top)^*(\ga)=\de$. Equation \eq{ss1eq9} follows as $\ga$ is a morphism of strong H-principal $\Z_2$-bundles. The rest of Theorem \ref{ss1thm2}(c) is immediate.

\medskip

\noindent{\small\sc The Mathematical Institute, Radcliffe
Observatory Quarter, Woodstock Road, Oxford, OX2 6GG, U.K.

\noindent E-mails: {\tt yalong.cao@maths.ox.ac.uk, Jacob.Gross@maths.ox.ac.uk, 

\noindent joyce@maths.ox.ac.uk.}}

\end{document}